\documentclass[twoside,10pt]{article}

\usepackage[english]{babel}             
\usepackage[T1]{fontenc}              
\usepackage[latin1]{inputenc}         
\usepackage{amsmath,amsfonts,amssymb} 
\usepackage{graphics,graphicx}        
\usepackage{subfigure}                
\usepackage{epsfig}
\usepackage{ae,aecompl}
\usepackage{color}

\def\R{\mathbb{R}}

\def\Z{\mathbb{Z}}
\def\N{\mathbb{N}}
\def\C{\mathbb{C}}

\def\k{\textit {\textbf k}}

\newcommand\cR{\mathcal R}

\newcommand\Ec{{E_{\rm c}}}
\newcommand\cN{\mathcal N}
\newcommand\cI{\mathcal I}
\newcommand\cG{\mathcal G}
\newcommand\cU{\mathcal U}
\newcommand\cS{\mathcal S}
\newcommand\tr{{\rm Tr} \,}

\newcommand\dps{\displaystyle }

\newtheorem{theorem}{Theorem}[section]
\newtheorem{lemma}[theorem]{Lemma}
\newtheorem{e-proposition}[theorem]{Proposition}

\newtheorem{e-definition}[theorem]{Definition\rm}

\newtheorem{remark}{Remark} 

\def\qed{\relax
     \ifmmode
       ~\hfill\Box
     \else
        \unskip\nobreak ~\hfill$\Box$
      \fi \par}

\newtheorem{Proof}{Proof}

\newenvironment{proof}{\begin{Proof}\rm}{\qed\end{Proof}}

\begin{document}

\title{Numerical analysis of the planewave discretization of
  some orbital-free and Kohn-Sham models}
\author{Eric Canc\`es\footnote{Universit\'e Paris-Est, CERMICS,
    Project-team Micmac, INRIA-Ecole des Ponts,  6 \& 8 avenue Blaise
    Pascal, 77455 Marne-la-Vall\'ee Cedex 2, France.}, Rachida
  Chakir\footnote{UPMC Univ Paris 06, UMR 7598 LJLL, Paris, F-75005 France ;
CNRS, UMR 7598 LJLL, Paris, F-75005 France} 
$\,$ and Yvon Maday$^\dag$\footnote{Division of Applied Mathematics, Brown
  University, Providence, RI, USA} } 
%
%
\maketitle

\begin{abstract}
We provide {\it a priori} error estimates for the spectral and
pseudospectral Fourier (also called planewave) discretizations of the
periodic Thomas-Fermi-von Weizs\"acker (TFW) model and for the spectral
discretization of the Kohn-Sham
model, within the local density approximation (LDA). These models
allow to compute approximations of the ground state energy and density
of molecular systems in the condensed phase. The TFW model is stricly
convex with respect to the electronic density, and allows for a
comprehensive analysis. This is not the case for the Kohn-Sham LDA
model, for which the uniqueness of the ground state electronic density
is not guaranteed. Under a coercivity assumption on the second order
optimality condition, we prove that for large enough energy cut-offs,
the discretized Kohn-Sham LDA problem has a minimizer in the
vicinity of any Kohn-Sham ground state, and that this minimizer is unique up to
unitary transform. We then derive optimal  {\it a priori} error estimates
for the spectral discretization method.
\end{abstract}

\selectlanguage{english}
\section{Introduction}
\label{sec:introduction}

Density Functional Theory (DFT) is a powerful method
for computing ground state electronic energies and
densities in quantum chemistry, materials science, molecular biology and
nanosciences. The models originating from DFT can be classified into two
categories: the orbital-free models and the Kohn-Sham models. 
The Thomas-Fermi-von Weizs\"acker (TFW) model falls into the first
category. It is not very much used in practice, but is
interesting from a mathematical viewpoint~\cite{BenguriaBrezisLieb,CattoLeBrisLions2,Lieb}. It indeed serves as a toy
model for the analysis of the more complex electronic structure models
routinely used by Physicists and Chemists. At the other extremity of the
spectrum, the Kohn-Sham models \cite{DreizlerGross,KohnSham} are among the most widely used models in
Physics and Chemistry, but are much more difficult to deal with. We focus
here on the numerical analysis of the TFW model on the one hand, and of
the Kohn-Sham model, within the local density approximation (LDA), on
the other hand. More precisely, we are interested in the spectral and
pseudospectral Fourier, more commonly called planewave, discretizations
of the periodic versions of these two models. In this context, the
simulation domain, sometimes referred to as the supercell, is the unit cell
of some periodic lattice of $\R^3$. In the TFW model, periodic boundary
conditions 
(PBC) are imposed to the density; in the Kohn-Sham framework, they are
imposed to the Kohn-Sham orbitals (Born-von Karman PBC). Imposing
PBC at the boundary of the simulation cell is a standard method to
compute condensed phase properties with a limited number of atoms in the
simulation cell, hence at a moderate computational cost.
 
This article is organized as follows. In Section~\ref{sec:Fourier}, we
briefly introduce the functional setting used in the formulation and
the analysis of the planewave discretization of orbital-free and
Kohn-Sham models. In Section~\ref{sec:TFW}, we provide {\it a priori}
error estimates for the planewave discretization of the TFW model, including numerical integration. In
Section~\ref{sec:KS}, we deal with the Kohn-Sham LDA model.

\section{Basic Fourier analysis for planewave discretization methods}
\label{sec:Fourier}

Throughout this article, we denote by $\Gamma$ the simulation cell, by
$\cR$ the periodic lattice, and by
$\cR^\ast$ the dual lattice. For
simplicity, we assume that $\Gamma=[0,L)^3$  ($L > 0$), in which case 
$\cR$ is the cubic lattice $L\Z^3$, and $\cR^\ast = \frac{2\pi}L \Z^3$.
Our arguments can be easily extended to the general case. For $k \in
\cR^\ast$, we denote by $e_k(x)=|\Gamma|^{-1/2} \, e^{ik\cdot x}$ the
planewave with wavevector $k$. The family $(e_k)_{k \in \cR^\ast}$
forms an orthonormal basis of 
$$
L^2_\#(\Gamma,\C):=\left\{ u \in L^2_{\rm loc}(\R^3,\C) \; | \; u \mbox{
    $\cR$-periodic} \right\},
$$
and for all $u \in L^2_\#(\Gamma,\C)$,
$$
u(x) = \sum_{k \in \cR^\ast} \widehat u_k \, e_k(x) \qquad \mbox{with} \qquad
\widehat u_k=(e_k,u)_{L^2_\#} = |\Gamma|^{-1/2} \int_\Gamma u(x)
e^{-ik\cdot x} \, dx.
$$
In our analysis, we will mainly consider real valued functions. We
therefore introduce the Sobolev spaces of real valued $\cR$-periodic functions 
$$
H^s_\#(\Gamma) :=\left\{ u(x) =  \sum_{k \in \cR^\ast} \widehat u_k \,
  e_k(x) \; | \; \sum_{k \in \cR^\ast} (1+|k|^2)^s |\widehat
  u_k|^2 < \infty \mbox{ and } \forall k,  \; \widehat u_{-k}=\widehat u_k^\ast  \right\},
$$
$s \in \R$ (here and in the sequel $a^\ast$ denotes the complex conjugate of the complex number $a$), endowed with the inner products
$$
(u,v)_{H^s_\#} =  \sum_{k \in \cR^\ast} (1+|k|^2)^s \,
\widehat u_k^\ast \, \widehat v_k.
$$ 
For $N_c \in \N$, we denote by
\begin{equation} 
\label{eq_1}
V_{N_c} = \left\{ \sum_{k \in \cR^\ast \, | \, |k| \le \frac{2\pi}L N_c} c_k
  e_k \; | \; \forall k, \; c_{-k}=c_k^\ast \right\}
\end{equation} 
(the constraints $c_{-k}=c_k^\ast$ imply that the functions of $V_{N_c}$
are real valued). For all $s \in \R$, and each $v
\in H^s_\#(\Gamma)$, the best approximation of $v$ in $V_{N_c}$ for {\it any}
$H^r_\#$-norm, $r \le s$, is
$$
\Pi_{N_c} v = \sum_{\k\in \cR^\ast \, | \, |\k| \le \frac{2\pi}L N_c}
\widehat v_k e_k. 
$$
The more regular $v$ (the regularity being measured in terms of the
Sobolev norms $H^r$), the faster the convergence of this truncated
series to $v$: for all real numbers $r$ and $s$ with $r \le s$, we
have for each $v\in H^s_\#(\Gamma)$,
\begin{eqnarray}
\|v - \Pi_{N_c}v\|_{H^r_\#} = \min_{v_{N_c}
  \in V_{N_c}} \|v-v_{N_c}\|_{H^r_\#} & \le & 
\left( \frac{L}{2\pi} \right)^{s-r} \, N_c^{-(s-r)} \|v-
\Pi_{N_c}v\|_{H^s_\#} 
\nonumber \\
&\le & 
\left( \frac{L}{2\pi} \right)^{s-r} \, N_c^{-(s-r)} \|v\|_{H^s_\#}.
\label{eq:app-Fourier}
\end{eqnarray}
For $N_g \in \N \setminus \left\{0\right\}$, we denote by 
$\widehat{\phi}^{{\rm FFT},N_g}$ the discrete Fourier transform 
on the carterisan grid $\cG_{N_g}:=\frac{L}{N_g} \, \Z^3$ of the
function $\phi \in C^0_\#(\Gamma,\C)$, where
$$
 C^0_\#(\Gamma,\C) :=\left\{ u \in C^0(\R^3,\C) \; | \; u \mbox{
    $\cR$-periodic} \right\}.
$$
Recall that
if $\phi = \sum_{k \in \cR^\ast} \widehat \phi_k \, e_k \in C^0_\#(\Gamma,\C)$,
the discrete Fourier transform of $\phi$ is the $N_g\cR^\ast$-periodic
sequence $\widehat{\phi}^{{\rm FFT},N_g}=(\widehat{\phi}^{{\rm FFT},N_g}_{k})_{k
  \in \cR^\ast}$ where  
$$
\widehat{\phi}^{{\rm FFT},N_g}_{k}
= \frac{1}{N_g^3} \sum_{x \in \cG_{N_g} \cap \Gamma} \phi(x) e^{-ik
  \cdot x} = |\Gamma|^{-1/2} \sum_{K \in \cR^\ast} \widehat \phi_{k+N_gK}.
$$
We now introduce the subspaces 
$$
W_{N_g}^{\rm 1D}  \; = \; \left| \begin{array}{lll}
\dps \mbox{Span} \left\{ e^{ily} \; | \; l \in \frac{2\pi}L \Z, \; |l| \le 
\frac{2\pi}L \left( \frac{N_g-1}2 \right) \right\}
 & \quad  (N_g \mbox{ odd}), \\
 \dps \mbox{Span} \left\{ e^{ily} \; | \; l \in \frac{2\pi}L \Z, \; |l| \le 
\frac{2\pi}L \left( \frac{N_g}2 \right) \right\} \oplus \C
(e^{i \pi N_gy/L}+e^{-i \pi N_gy/L}) 
& \quad  (N_g \mbox{ even}),
\end{array} \right.
$$
($W_{N_g}^{\rm 1D} \in C^\infty_\#([0,L),\C)$ and $\mbox{dim}(W_{N_g}^{\rm
  1D}) = N_g$), and 
$W_{N_g}^{\rm 3D} = W_{N_g}^{\rm 1D} \otimes W_{N_g}^{\rm 1D} \otimes
W_{N_g}^{\rm 1D}$.
Note that $W_{N_g}^{\rm 3D}$ is a subspace of $H^s_\#(\Gamma,\C)$ of dimension
$N_g^3$, for all $s \in \R$, and that if $N_g$ is odd,
$$
W_{N_g}^{\rm 3D} =  \mbox{Span} \left\{ e_k \; | \; k \in \cR^\ast =
  \frac{2\pi}L \Z^3, \;
  |k|_\infty \le  \frac{2\pi}L \left( \frac{N_g-1}2 \right) \right\}
\qquad  \qquad  (N_g \mbox{ odd}).
$$
It is then possible to define
the interpolation projector $\cI_{N_g}$ from $C^0_\#(\Gamma,\C)$ onto
$W_{N_g}^{\rm 3D}$ by $[\cI_{N_g}(\phi)](x) = \phi(x)$ for all 
$x \in \cG_{N_g}$. It holds
\begin{equation} \label{eq:integration_formula_INg}
\forall \phi \in C^0_\#(\Gamma,\C), \quad 
\int_\Gamma \cI_{N_g}(\phi) =  \sum_{x \in \cG_{N_g} \cap \Gamma} \left(
  \frac{L}{N_g} \right)^3 \phi(x).
\end{equation}
The coefficients of the expansion of
$\cI_{N_g}(\phi)$ in the canonical 
basis of $W_{N_g}^{\rm 3D}$ is given by the discrete Fourier transform of
$\phi$. In particular, when $N_g$ is odd, we have the simple relation
$$
\cI_{N_g}(\phi)  =  
|\Gamma|^{1/2} \dps \sum_{k \in \cR^\ast \, | \, 
|k|_\infty \le \frac{2\pi}L \left( \frac{N_g-1}2 \right)} \widehat{\phi}^{{\rm
      FFT},N_g}_{k} \, e_k  \qquad \qquad  (N_g \mbox{ odd}).
$$
It is easy to check that if $\phi$ is real-valued, then so is
$\cI_{N_g}(\phi)$.

We will assume in the sequel that $N_g \ge 4N_c+1$. We will then have for
all $v_{4N_c} \in V_{4N_c}$, 
\begin{equation} \label{eq:exact_integration}
\int_\Gamma v_{4N_c} = \sum_{x \in \cG_{N_g} \cap \Gamma}
\left( \frac{L}{N_g} \right)^3 v_{4N_c}(x) = \int_\Gamma \cI_{N_g} (v_{4N_c}).
\end{equation}

The following lemma gathers some technical results which will be useful
for the numerical analysis of the planewave discretization of
orbital-free and Kohn-Sham models.

\begin{lemma} \label{lem:technical}
Let $N_c \in \N^\ast$ and $N_g \in \N^\ast$ such that $N_g \ge 4N_c+1$.
  \begin{enumerate}
  \item Let $V$ be a function of $C^0_\#(\Gamma,\C)$ and $v_{N_c}$
    and $w_{N_c}$ be two functions of $V_{N_c}$. Then 
\begin{eqnarray}
\int_\Gamma \cI_{N_g}(V v_{N_c}w_{N_c}) &=& 
\int_\Gamma \cI_{N_g}(V) v_{N_c} w_{N_c} ;  \label{eq:ineg_INg_2} \\
\left| \int_\Gamma \cI_{N_g}(V |v_{N_c}|^2) \right| &\le & \|V\|_{L^\infty} 
\|v_{N_c}\|_{L^2_\#}^2.  \label{eq:bound_INg1}
\end{eqnarray}

  \item Let $s > 3/2$, $0 \le r \le s$, and $V$ a function of
    $H^s_\#(\Gamma)$. Then,
\begin{eqnarray}
\left\| (1-\cI_{N_g})(V) \right\|_{H^r_\#} & \le &
C_{r,s} N_g^{-(s-r)} \|V\|_{H^s_\#};  \label{eq:ineg_INg_0} \\
\left\| \Pi_{2N_c}(\cI_{N_g}(V)) \right\|_{L^2_\#} & \le &
\left( \int_\Gamma \cI_{N_g}(|V|^2) \right)^{1/2};  \label{eq:ineg_INg_4} \\
\left\| \Pi_{2N_c}(\cI_{N_g}(V)) \right\|_{H^s_\#} & \le &
(1+C_{s,s}) \|V\|_{H^s_\#},  \label{eq:ineg_INg_5}
\end{eqnarray}
for constants $C_{r,s}$ independent of $V$. Besides if there exists $m
> 3$ and $C \in \R_+$ such that $|\widehat V_k| \le C |k|^{-m}$, then
there exists a constant $C_V$ independent of $N_c$ and $N_g$ such that
\begin{eqnarray}
\left\| \Pi_{2N_c}(1-\cI_{N_g})(V) \right\|_{H^r_\#} & \le & 
C_V N_c^{r+3/2} N_g^{-m}.  \label{eq:ineg_INg_6} 
\end{eqnarray}
 \item Let $\phi$ be a Borel function from
$\R_+$ to $\R$ such that there exists $C_\phi \in \R_+$ for which
$|\phi(t)| \le C_\phi (1+t^2)$ for all $t \in \R_+$. Then, for all $v_{N_c}
\in V_{N_c}$,
\begin{eqnarray}
\left| \int_\Gamma \cI_{N_g}(\phi(|v_{N_c}|^2)) \right| &\le & C_\phi \left(
  |\Gamma|+\|v_{N_c}\|_{L^4_\#}^4 \right). \label{eq:bound_INg2} 
\end{eqnarray}
  \end{enumerate}
\end{lemma}

\begin{proof} 
For $z_{2N_c} \in V_{2N_c}$, it holds
\begin{eqnarray}
\int_\Gamma \cI_{N_g}(V z_{2N_c}) &=& \sum_{x \in \cG_{N_g} \cap \Gamma}  \left( \frac{L}{N_g} \right)^3 V(x) z_{2N_c}(x)  \nonumber \\
& = & \sum_{x \in \cG_{N_g} \cap \Gamma} \left( \frac{L}{N_g} \right)^3 
(\cI_{N_g}(V))(x) z_{2N_c}(x)  \nonumber \\
& = & \int_\Gamma \cI_{N_g}(V) \,
z_{2N_c} \label{eq:equality_INg}
\end{eqnarray}
since $\cI_{N_g}(V) z_{2N_c} \in V_{N_g+2 N_c} \subset V_{2 N_g}$ is exactly integrated. The function
$v_{N_c}w_{N_c}$ being in $V_{2N_c}$, (\ref{eq:ineg_INg_2}) 
is proved. Moreover, as $|v_{N_c}|^2 \in V_{4N_c}$, it
follows from (\ref{eq:exact_integration}) that 
\begin{eqnarray*}
\left| \int_\Gamma \cI_{N_g}(V |v_{N_c}|^2 ) \right| & = &
\left|  \sum_{x \in \cG_{N_g} \cap \Gamma} \left( \frac{L}{N_g}
  \right)^3  V(x)  |v_{N_c}(x)|^2 \right| \\
& \le &  \|V\|_{L^\infty} \left|  \sum_{x \in \cG_{N_g}
    \cap \Gamma} \left( \frac{L}{N_g} 
  \right)^3    |v_{N_c}(x)|^2 \right| \\
& = &  \|V\|_{L^\infty} \int_\Gamma |v_{N_c}|^2.
\end{eqnarray*}
Hence (\ref{eq:bound_INg1}). 
The estimate (\ref{eq:ineg_INg_0}) is
proved in \cite{CHQZ}. To prove (\ref{eq:ineg_INg_4}), we
notice that
\begin{eqnarray*}
\|\Pi_{2N_c}(I_{N_g}(V))\|_{L^2_\#}^2 & \le  & \|I_{N_g}(V)\|_{L^2_\#}^2
\\
& = & \int_\Gamma (I_{N_g}(V))^\ast  (I_{N_g}(V)) \\
& = &  \sum_{x \in \cG_{N_g} \cap \Gamma} \left( \frac{L}{N_g} \right)^3 (I_{N_g}(V))(x)^\ast
(I_{N_g}(V))(x) \\
& = &  \sum_{x \in \cG_{N_g} \cap \Gamma} \left( \frac{L}{N_g} \right)^3 |V(x)|^2 \\
& = &  \int_\Gamma I_{N_g}(|V|^2).
\end{eqnarray*}
The bound (\ref{eq:ineg_INg_5}) is a straightforward
consequence of (\ref{eq:ineg_INg_0}):
\begin{eqnarray*}
\|\Pi_{2N_c}(I_{N_g}(V))\|_{H^s_\#} & \le & \|I_{N_g}(V)\|_{H^s_\#} \le 
\|V\|_{H^s_\#} + \|(1-I_{N_g})(V)\|_{H^s_\#}  \le 
(1+C_{s,s}) \|V\|_{H^s_\#}.
\end{eqnarray*}
Now, we notice that
\begin{eqnarray}
\Pi_{2N_c} (\cI_{N_g}(V)) &=& |\Gamma|^{1/2} 
\sum_{k \in \cR^\ast \, | \, |k| \le \frac{4\pi}{L}N_c} \widehat
V_k^{{\rm FFT},N_g} e_k 
\nonumber \\
& = &  \sum_{k \in \cR^\ast \, | \, |k| \le \frac{4\pi}{L}N_c} \left(
  \sum_{K \in \cR^\ast}  
\widehat V_{k+N_g K} \right)  e_k \label{eq:Pi2NINg}.
\end{eqnarray}
From (\ref{eq:Pi2NINg}), we obtain
\begin{eqnarray*}
\left\| \Pi_{2N_c}(1-\cI_{N_g})(V) \right\|_{H^s_\#}^2 & = &
 \sum_{k \in \cR^\ast \, | \, |k| \le \frac{4\pi}{L}N_c} (1+|k|^2)^s \left| 
\sum_{K \in \cR^\ast \setminus \left\{0\right\}} 
\widehat V_{k+N_g K} \right|^2 \\ & \le  & \left( \sum_{k \in \cR^\ast \, |
  \, |k| \le \frac{4\pi}{L}N_c} (1+ |k|^2)^s \right) 
 \max_{k \in \cR^\ast \, | \, |k| \le \frac{4\pi}{L}N_c}  \left| 
\sum_{K \in \cR^\ast \setminus \left\{0\right\}} 
\widehat V_{k+N_g K} \right|^2 .
\end{eqnarray*}
On the one hand,
$$
\sum_{k \in \cR^\ast \, | \, 
|k| \le \frac{4\pi}{L}N_c} (1+ |k|^2)^s \mathop{\sim}_{N_c \to
  \infty} \frac{32\pi}{2s+3} \left( \frac{4\pi}L \right)^{2s} \, N_c^{2s+3},  
$$
and on the other hand, we have for each $k \in \cR^\ast$ such that $|k|
\le \frac{4\pi}{L}N_c$, 
\begin{eqnarray*}
 \left| \sum_{K \in \cR^\ast \setminus \left\{0\right\}} 
\widehat V_{k+N_g K} \right| & \le & C 
 \sum_{K \in \cR^\ast \setminus \left\{0\right\}} \frac{1}{|k+N_gK|^m}
 \\
& \le & C \, C_0 \left( \frac{L}{2\pi} \right)^{m} N_g^{-m} 
\end{eqnarray*}
where
$$
C_0 = \max_{y \in \R^3 \, | \, |y| \le 1/2} \sum_{K \in \Z^3 \setminus
  \left\{0\right\}}  \frac{1}{|y-K|^m}.
$$
The estimate (\ref{eq:ineg_INg_6}) then easily follows. Let us finally prove
(\ref{eq:bound_INg2}). Using  (\ref{eq:integration_formula_INg}) and
(\ref{eq:exact_integration}), we have
\begin{eqnarray*}
\left| \int_\Gamma \cI_{N_g}(\phi(|v_{N_c}|^2)) \right| &=&
\left| \sum_{x \in \cG_{N_g} \cap \Gamma} \left( \frac{L}{N_g}
  \right)^3  \phi(|v_{N_c}(x)|^2) \right| \\
& \le & C_\phi
\left| \sum_{x \in  \cG_{N_g} \cap \Gamma} \left( \frac{L}{N_g}
  \right)^3  (1+|v_{N_c}(x)|^4) \right| \\
& = &  C_\phi \int_\Gamma (1+|v_{N_c}|^4)  = 
C_\phi \left( |\Gamma|+\|v_{N_c}\|_{L^4_\#}^4 \right).
\end{eqnarray*}
This completes the proof of Lemma~\ref{lem:technical}.
\end{proof}

\section{Planewave approximation of the TFW model}
\label{sec:TFW}

In the TFW model, as well as in any orbital-free model, the ground state
electronic density of the system is obtained by minimizing an explicit
functional of the density. Denoting by $\cN$ the number of electrons in
the simulation cell and by
$$
{\mathfrak R}_\cN = \left\{ \rho \ge 0 \; | \; \sqrt\rho \in
  H^1_\#(\Gamma), \; \int_\Gamma \rho = \cN \right\}
$$
the set of admissible densities, the TFW problem reads
\begin{equation} 
I^{\rm TFW}  =  \inf \left\{ {\mathcal E}^{\rm TFW}(\rho), \; \rho \in
  {\mathfrak R}_\cN \right\} \label{eq:minTFWrho}, 
\end{equation}
where 
$$\dps
{\mathcal E}^{\rm TFW}(\rho)  =  \frac{C_{\rm W}}2 \int_\Gamma |\nabla
\sqrt \rho|^2 +  
C_{\rm TF} \int_\Gamma \rho^{5/3} + \int_\Gamma \rho V^{\rm ion} + \frac
12 D_\Gamma(\rho,\rho).
$$

\noindent 
$C_{\rm W}$ is a positive real number ($C_{\rm W}=1$, $1/5$ or $1/9$
depending on the context \cite{DreizlerGross}), and $C_{\rm TF}$ is 
the Thomas-Fermi constant: $C_{\rm TF}=\frac{10}3(3\pi^2)^{2/3}$.
The last term of the TFW energy models the periodic Coulomb energy: for
$\rho$ and $\rho'$ in $H^{-1}_\#(\Gamma)$,
$$
D_\Gamma(\rho,\rho'):= 4 \pi \sum_{k \in \cR^\ast \setminus
  \left\{0\right\}} |k|^{-2} \widehat \rho_k^\ast \, 
\widehat \rho'_k.
$$
We finally make the assumption that $V^{\rm ion}$ is a $\cR$-periodic 
potential such that 
\begin{equation} \label{eq:hypothesis}
\exists m > 3, \; C \ge 0 \mbox{ s.t. }  \forall k \in \cR^\ast, \;
|\widehat V^{\rm ion}_k| \le C |k|^{-m}.
\end{equation}
Note that this implies that $V^{\rm ion}$ is in
$H^{m-3/2-\epsilon}(\Gamma)$ for all $\epsilon > 0$, hence in $C^0_\#(\Gamma)$ since $m-3/2-\epsilon > 3/2$ for $\epsilon$ small enough. 
It is convenient to reformulate the TFW model in terms of
$v=\sqrt{\rho}$. It can be easily seen that
\begin{equation} 
I^{\rm TFW}  =  \inf \left\{ E^{\rm TFW}(v), \; v \in H^1_\#(\Gamma),
  \; \int_\Gamma |v|^2 = \cN \right\}, \label{eq:minTFWu} 
\end{equation}
where 
$$
E^{\rm TFW}(v)  =  \frac{C_{\rm W}}2 \int_\Gamma |\nabla v|^2 + 
C_{\rm TF} \int_\Gamma |v|^{10/3} + \int_\Gamma  V^{\rm ion} |v|^2 + \frac
12 D_\Gamma(|v|^2,|v|^2).
$$

Let $F(t)=C_{\rm TF}t^{5/3}$ and $f(t) = F'(t) = \frac 53 C_{\rm
  TF}t^{2/3}$. The function $F$ is in $C^1([0,+\infty))\cap 
C^\infty((0,+\infty))$, is strictly convex on $[0,+\infty)$, and
for all $(t_1,t_2) \in \R_+ \times \R_+$,
\begin{equation} \label{eq:fff}
|f(t_2^2)t_2- f(t_1^2)t_2 - 2f'(t_1^2)t_1^2(t_2-t_1)| 
\le \frac{70}{27} C_{\rm TF} \max(t_1^{1/3},t_2^{1/3}) \, |t_2-t_1|^2.
\end{equation}

The first and second derivatives of
$E^{\rm TFW}$   are respectively given by
\begin{eqnarray*}
&& \langle {E^{\rm TFW}}'(v),w \rangle_{H^{-1}_\#,H^1_\#} 
= 2 \langle {\cal H}^{{\rm TFW}}_{|v|^2}v,w \rangle; \\
&& \langle {E^{\rm TFW}}''(v)w_1,w_2 \rangle_{H^{-1}_\#,H^1_\#} 
= 2 \langle {\cal H}^{{\rm TFW}}_{|v|^2}w_1,w_2 \rangle + 4D_\Gamma(vw_1,vw_2)+ 4 \int_\Gamma
f'(|v|^2) |v|^2 w_1w_2, 
\end{eqnarray*}
where we have denoted by  ${\cal H}^{{\rm TFW}}_{\rho}$ the TFW Hamiltonian associated with the density $\rho$
$$
{\cal H}^{{\rm TFW}}_{\rho} = - \frac{C_{\rm W}}2 \Delta + f(\rho) + V^{\rm ion} +
V_{\rho}^{\rm Coulomb},   
$$
 where 
$$
V_{\rho}^{\rm Coulomb}(x) = 4\pi
\sum_{k \in \cR^\ast \setminus \left\{0\right\}} |k|^{-2}
\widehat{\rho}_k e_k(x)
$$
is the $\cR$-periodic Coulomb potential generated by the $\cR$-periodic charge
distribution~$\rho$. Recall that $V^{\rm Coulomb}_\rho$ can also be
defined as the unique solution in $H^1_\#(\Gamma)$ to 
$$
\left\{ \begin{array}{l}
\dps - \Delta V_\rho^{\rm Coulomb} = 4\pi \left(
  \rho-|\Gamma|^{-1}\int_\Gamma\rho \right) \\
\dps \int_\Gamma V_\rho^{\rm Coulomb} = 0.
\end{array} \right.
$$

Let us recall (see  
\cite{Lieb} and the proof of Lemma~2 in \cite{CCM}) that 
\begin{itemize}
\item (\ref{eq:minTFWrho}) has a unique
minimizer $\rho^0$, and that the minimizers of 
  (\ref{eq:minTFWu}) are $u$ and $-u$, where
  $u=\sqrt{\rho^0}$;
\item  $u$ is in
$H^{m+1/2-\epsilon}_\#(\Gamma)$ for each $\epsilon
> 0$ (hence in
  $C^{2}_\#(\Gamma)$ since $m+1/2-\epsilon > 7/2$ for $\epsilon$ small
  enough);
\item $u > 0$ on $\R^3$;
\item $u$  satisfies the Euler equation 
$$
{\cal H}^{{\rm TFW}}_{|u|^2}(u) = - \frac{C_{\rm W}}2 \Delta u + \left(\frac 53 C_{\rm TF} u^{4/3} +
V^{\rm ion}  + V_{u^2}^{\rm Coulomb} \right) u = \lambda u 
$$
for some $\lambda \in \R$, (the
ground state eigenvalue of ${\cal H}^{{\rm TFW}}_{\rho^0}$, that is
non-degenerate).  
\end{itemize}

The planewave discretization of the TFW model is obtained by choosing 
\begin{enumerate}
\item an energy cut-off $\Ec > 0$ or, equivalently, a finite
  dimensional Fourier space $V_{N_c}$, the integer $N_c$ being related
  to $\Ec$ through the relation $N_c :=
[\sqrt{2\Ec} \, L/2\pi]$;
\item a cartesian grid $\cG_{N_g}$ with step size
$L/N_g$ where $N_g \in \N^\ast$ is such
that $N_g \ge 4N_c+1$, 
\end{enumerate}
and by considering the finite dimensional minimization problem 
\begin{equation} 
I^{\rm TFW}_{N_c,N_g}  =  \inf \left\{ E^{\rm TFW}_{N_g}(v_{N_c}),
  \; v_{N_c} \in V_{N_c},
  \; \int_\Gamma |v_{N_c}|^2 = \cN \right\}, \label{eq:minTFWuN} 
\end{equation}
where
\begin{eqnarray*} 
E^{\rm TFW}_{N_g}(v_{N_c}) & = & \frac{C_{\rm W}}2 \int_\Gamma
|\nabla v_{N_c}|^2 +  
C_{\rm TF} \int_\Gamma \cI_{N_g}(|v_{N_c}|^{10/3}) + \int_\Gamma
\cI_{N_g}(V^{\rm ion}) |v_{N_c}|^2 \\ && + \frac 12
D_\Gamma(|v_{N_c}|^2,|v_{N_c}|^2),
\end{eqnarray*} 
$\cI_{N_g}$ denoting the interpolation operator introduced in the 
previous section. The Euler equation associated with
(\ref{eq:minTFWuN}) can be written as a nonlinear eigenvalue problem
$$
\forall v_{N_c} \in V_{N_c}, \quad \langle (\widetilde
{\cal H}^{{\rm TFW}, N_g}_{|u_{N_c,N_g}|^2}  - \lambda_{N_c,N_g})
u_{N_c,N_g},v_{N_c} \rangle_{H^{-1}_\#,H^1_\#} = 0,
$$
where we have denoted by
$$
\widetilde {\cal H}^{{\rm TFW}, N_g}_{\rho} =  - \frac{C_{\rm W}}2 \Delta +
\cI_{N_g}\left(\frac 53 C_{\rm TF} \rho^{2/3} + V^{\rm ion}\right) +
V_{\rho}^{\rm Coulomb}    
$$
the pseudospectral TFW Hamiltonian associated with the density $\rho$,
and by $\lambda_{N_c,N_g}$ the Lagrange multiplier of the constraint
$\int_\Gamma |v_{N_c}|^2 = \cN$.
We therefore have
$$
- \frac{C_{\rm W}}2 \Delta u_{N_c,N_g}+ \Pi_{N_c} \left[
\left(\cI_{N_g}\left( \frac 53 C_{\rm TF}|u_{N_c,N_g}|^{4/3} + V^{\rm ion}\right) + V_{|u_{N_c,N_g}|^2}^{\rm Coulomb}\right)u_{N_c,N_g}
\right]  = \lambda_{N_c,N_g} u_{N_c,N_g}.
$$
Under the condition that $N_g \ge 4N_c+1$, we have for all $\phi \in C^0_\#(\Gamma)$,  
$$
\forall (k,l) \in \cR^\ast \times \cR^\ast \mbox{ s.t. } |k|
,|l| \le \frac{2\pi}L N_c, \quad \int_\Gamma \cI_{N_g}(\phi) \, 
e_k^\ast \, e_l  = \widehat{\phi}^{{\rm FFT}}_{k-l},
$$
so that,  $\widetilde {\cal H}^{{\rm TFW}}_{u_{N_c,N_g}}$ is defined on $V_{N_c}$  by
the Fourier matrix 
\begin{eqnarray*}
[\widehat {\cal H}^{{\rm TFW}, N_g}_{|u_{N_c,N_g}|^2}]_{kl} &=& \frac{C_{\rm W}}2 |k|^2 \delta_{kl} + 
\frac 53 C_{\rm TF} \widehat{(|u_{N_c,N_g}|^{4/3})}_{k-l}^{{\rm FFT},N_g}
+ \widehat{(V^{\rm ion})}_{k-l}^{{\rm FFT},N_g} \\ &&  
+4\pi \frac{\widehat{(|u_{N_c,N_g}|^2)}_{k-l}^{{\rm FFT},N_g}}{|k-l|^2}
\left( 1 - \delta_{kl} \right),
\end{eqnarray*}
where, by convention, the last term of the right hand side is equal to
zero for $k=l$.

\medskip

\noindent
We also introduce the variational approximation of (\ref{eq:minTFWu}) 
\begin{equation}
I^{\rm TFW}_{N_c}  =  \inf \left\{ E^{\rm TFW}(v_{N_c}),
  \; v_{N_c} \in V_{N_c},
  \; \int_\Gamma |v_{N_c}|^2 = \cN \right\}. \label{eq:minTFWuV} 
\end{equation}
Any minimizer $u_{N_c}$ to (\ref{eq:minTFWuV}) satisfies the
elliptic equation
\begin{equation} \label{eq:Euler_VNc}
- \frac{C_{\rm W}}2 \Delta u_{N_c} + \Pi_{N_c} \left[ 
\frac 53 C_{\rm TF} |u_{N_c}|^{4/3} u_{N_c}+ V^{\rm ion}u_{N_c} + V_{|u_{N_c}|^2}^{\rm Coulomb} u_{N_c} 
 \right] = \lambda_{N_c}u_{N_c},
\end{equation}
for some $\lambda_{N_c} \in \R$.

\medskip

The main result of this section is an extension of results previously obtained by A. Zhou~\cite{Zhou}.

\medskip

\begin{theorem} \label{Th:TFW} For each $N_c \in \N$, we denote by
  $u_{N_c}$ a minimizer to (\ref{eq:minTFWuV}) such that  
  $(u_{N_c},u)_{L^2_\#} \ge 0$ and, for each $N_c
  \in \N$ and $N_g \ge 4N_c+1$, we 
  denote by $u_{N_c,N_g}$ a minimizer to (\ref{eq:minTFWuN}) such that 
  $(u_{N_c,N_g},u)_{L^2_\#} \ge 0$. Then for $N_c$ large enough, $u_{N_c}$
  and $u_{N_c,N_g}$ are unique, and the following estimates hold true
\begin{eqnarray}
\|u_{N_c}-u\|_{H^s_\#} &\le& C_{s,\epsilon} N_c^{-(m-s+1/2-\epsilon)};
\label{eq:estim_Nc_u} \\
|\lambda_{N_c}-\lambda| &\le & C_\epsilon N_c^{-(2m-1-\epsilon)};
\label{eq:estim_Nc_lambda} \\
\gamma \|u_{N_c}-u\|_{H^1_\#}^2 \le I^{\rm TFW}_{N_c} -I^{\rm TFW}
&\le& C \|u_{N_c}-u\|_{H^1_\#}^2;
\label{eq:estim_Nc_I}
\\
\|u_{N_c,N_g}-u_{N_c}\|_{H^s_\#} &\le & C_s \, 
N_c^{3/2+(s-1)_+} N_g^{-m}; \label{eq:estim_NcNg_u}  \\ 
|\lambda_{N_c,N_g}-\lambda_{N_c}| &\le& C 
  N_c^{3/2}N_g^{-m};\label{eq:estim_NcNg_lambda} \\
|I^{\rm TFW}_{N_c,N_g} -I^{\rm TFW}_{N_c}| &\le & C
N_c^{3/2}N_g^{-m} , \label{eq:estim_NcNg_I}
\end{eqnarray}
for all $- m +3/2 < s < m+1/2$ and $\epsilon > 0$, and for some
constants $\gamma > 0$, $C_{s,\epsilon} \ge 0$, $C_\epsilon \ge 0$, $C 
\ge 0$ and $C_s \ge 0$ independent of $N_c$ and $N_g$.
\end{theorem}

\medskip

\begin{remark} \label{Th:OF}
More complex orbital-free models have been proposed in the recent years
\cite{orbitalfree}, which are used to perform
multimillion atom DFT calculations. Some of these models however are not
well posed  (the energy functional is not bounded from below 
\cite{BlancCances}), and 
the others are not well understood from a mathematical point of
view. For these reasons, we will not deal with those models in this
article.
\end{remark}

\subsection{A priori estimates for the variational approximation.}

In this section, we prove  the first part of Theorem~\ref{Th:TFW}, related to the variational approximation (\ref{eq:minTFWuV}).
The estimates (\ref{eq:estim_Nc_u}), (\ref{eq:estim_Nc_lambda}) and
(\ref{eq:estim_Nc_I}) originate from arguments already introduced in
\cite{CCM}. For brevity, we only recall the main steps of the
proof and leave the details to the reader. 

\medskip

\noindent
The difference between (\ref{eq:minTFWu}) and the problem dealt
with in \cite{CCM} is the presence of the Coulomb term
$D_\Gamma(|v|^2,|v|^2)$, for which the following estimates are
available: 
\begin{eqnarray}
0 \le D_\Gamma(\rho,\rho) & \le & C \|\rho\|_{L^2_\#}^2, \quad \mbox{
  for all } \rho \in L^2_\#(\Gamma),  \label{eq:estim_D1} \\
|D_\Gamma(uv,uw)| & \le  & C \|v\|_{L^2_\#} \|w\|_{L^2_\#},  \quad \mbox{
  for all } (v,w) \in (L^2_\#(\Gamma))^2, \label{eq:estim_D2} \\
|D_\Gamma(\rho,vw)| & \le & C \|\rho\|_{L^2_\#} \|v\|_{L^2_\#}
\|w\|_{L^2_\#}, \quad \mbox{
  for all } (\rho,v,w) \in (L^2_\#(\Gamma))^3, \qquad
\label{eq:estim_D3} \\
\|V_\rho^{\rm Coulomb}\|_{L^\infty} &\le& C \|\rho\|_{L^2_\#} , \quad \mbox{
  for all } \rho \in L^2_\#(\Gamma), \label{eq:estim_D4} \\
\|V_\rho^{\rm Coulomb}\|_{H^{s+2}_\#} &\le& C \|\rho\|_{H^s_\#}, \quad \mbox{
  for all } \rho \in H^s_\#(\Gamma).\label{eq:estim_D5}
\end{eqnarray}
Here and in the sequel, $C$ denotes a non-negative constant which may
depend on $\Gamma$, $V^{\rm ion}$ and $\cN$, but not on the
discretization parameters.

Using (\ref{eq:estim_D1}), (\ref{eq:estim_D2}) and the fact that $f' >
0$ on $(0,+\infty)$, we can then show (see the proof of Lemma~1 in
\cite{CCM}) that there exist $\beta > 0$, $\gamma > 0$ and $M \ge 0$
such that for all $v \in H^1_\#(\Gamma)$,
\begin{eqnarray}
&& 
0 \le \langle ({\cal H}^{{\rm TFW}}_{\rho^0}-\lambda)v,v \rangle_{H^{-1}_\#,H^1_\#}  \le
M \|v\|_{H^1_\#}^2, 
\\
&& \beta \|v\|_{H^1_\#}^2  \le 
\langle ({E^{\rm TFW}}''(u)-2\lambda)v,v \rangle_{H^{-1}_\#,H^1_\#}
\le M \|v\|_{H^1_\#}^2, \label{eq:NRJsecondContinue}
\end{eqnarray}
and for all $v \in H^1_\#(\Gamma)$ such that $\|v\|_{L^2_\#}=\cN^{1/2}$ and
$(v,u)_{L^2_\#} \ge 0$,
\begin{equation} \label{eq:borne_1}
\gamma \|v-u\|_{H^1_\#}^2 \le 
\langle ({\cal H}^{{\rm TFW}}_{\rho^0}-\lambda)(v-u),(v-u) \rangle_{H^{-1}_\#,H^1_\#}. 
\end{equation}
Remarking that
\begin{eqnarray} \!\!\!\!\!\!\!\!\!\!\!\!\!\!\!\!\!\!
E^{\rm TFW}(u_{N_c})- E^{\rm TFW}(u) \!\!\! & = & \!\!\! \langle
({\cal H}^{{\rm TFW}}_{\rho^0}-\lambda)(u_{N_c}-u),(u_{N_c}-u)\rangle_{H^{-1}_\#,H^1_\#}
\nonumber \\
\!\!\! && \!\!\! + \frac 12
D_\Gamma(|u_{N_c}|^2-|u|^2,|u_{N_c}|^2-|u|^2) \nonumber \\ \!\!\!
&& \!\!\! + \int_\Gamma F(|u_{N_c}|^2)-F(|u|^2)-f(|u|^2)(|u_{N_c}|^2-|u|^2)
\label{eq:erreur_TFW_Nc}
\end{eqnarray}
and using (\ref{eq:borne_1}), the positivity of the bilinear form
$D_\Gamma$, and the convexity of the function $F$, we obtain that 
$$
I^{\rm TFW}_{N_c}-I^{\rm TFW} = E^{\rm TFW}(u_{N_c})- E^{\rm TFW}(u) 
\ge \gamma \|u_{N_c}-u\|_{H^1_\#}^2. 
$$
For each $N_c \in \N$, 
$\widetilde u_{N_c}=\cN^{1/2}\Pi_{N_c}u/\|\Pi_{N_c}u\|_{L^2_\#}$ satisfies
$(\widetilde u_{N_c},u)_{L^2_\#} \ge 0$ and 
$\|\widetilde u_{N_c}\|_{L^2_\#} =\cN^{1/2}$, and the
sequence $(\widetilde u_{N_c})_{N_c \in \N}$ converges to $u$ in
$H^{m+1/2-\epsilon}_\#(\Gamma)$ for each $\epsilon > 0$. As the
functional $E^{\rm TFW}$ is continuous on $H^1_\#(\Gamma)$, we have 
$$
 \|u_{N_c}-u\|_{H^1_\#}^2 \le \gamma^{-1} \left(I^{\rm TFW}_{N_c}-I^{\rm
     TFW}\right) 
\le \gamma^{-1} \left(E^{\rm TFW}(\widetilde u_{N_c}) - E^{\rm
    TFW}(u)\right) 
\mathop{\longrightarrow}_{N_c \to \infty} 0.
$$
Hence, $(u_{N_c})_{N_c \in \N}$ converges to $u$ in
$H^1_\#(\Gamma)$, and we also have
\begin{eqnarray*}
\lambda_{N_c} & = & \cN^{-1} \bigg[ \frac {C_{\rm W}}2 \int_{\Gamma} |\nabla
u_{N_c}|^2
+ \int_\Gamma f(|u_{N_c}|^2) |u_{N_c}|^2
+ \int_\Gamma V^{\rm ion} |u_{N_c}|^2 + D_\Gamma(|u_{N_c}|^2,|u_{N_c}|^2)
 \bigg] \\
& \dps \mathop{\longrightarrow}_{N_c \to \infty} & \cN^{-1} \bigg[ 
\frac {C_{\rm W}}2 \int_{\Gamma} |\nabla u|^2 + \int_\Gamma f(|u|^2) |u|^2 
+ \int_\Gamma V^{\rm ion} |u|^2 + D_\Gamma(|u|^2,|u|^2)
\bigg] \\
& = & \lambda.
\end{eqnarray*}
As $f(|u_{N_c}|^2)u_{N_c}+ V^{\rm ion}u_{N_c} + V_{|u_{N_c}|^2}^{\rm Coulomb} u_{N_c} 
$ is bounded in $L^2_\#(\Gamma)$, uniformly in
$N_c$, we deduce from (\ref{eq:Euler_VNc}) that the sequence
$(u_{N_c})_{N_c \in \N}$ is bounded in $H^2_\#(\Gamma)$, hence in
$L^\infty(\Gamma)$. Now 
\begin{eqnarray*}
\Delta (u_{N_c}-u) & = & 2C_{\rm W}^{-1} \bigg[ \Pi_{N_c} \bigg( 
f(|u_{N_c}|^2)u_{N_c} - f(|u|^2)u + V^{\rm ion}(u_{N_c}-u) + 
 \\
&& \qquad \qquad \qquad V_{|u_{N_c}|^2}^{\rm Coulomb}
u_{N_c}-V_{|u|^2}^{\rm Coulomb}u 
 \bigg) \\
&  & \qquad  \quad + \left(1-\Pi_{N_c}\right) \left(f(|u|^2)u + V^{\rm ion} u + V_{|u|^2}^{\rm
    Coulomb}u \right) \\
&& \qquad  \quad 
- \lambda_{N_c}(u_{N_c}-u) - (\lambda_{N_c}-\lambda) u \bigg].
\end{eqnarray*}
Observing that the right-hand side goes to zero in $L^2_\#(\Gamma)$ when
$N_c$ goes to infinity, we obtain that $(u_{N_c})_{N_c \in \N}$
converges to $u$ in $H^2_\#(\Gamma)$, and therefore in
$C^{0,1/2}_\#(\Gamma)$. In addition, we know from Harnack inequality~\cite{GT} that $u > 0$ in
$\R^3$. Consequently, for $N_c$ large enough,
the function $u_{N_c}$ (which is continuous and $\cR$-periodic) is
bounded away from $0$, uniformly in $N_c$. As $f \in
C^\infty(0,+\infty)$, one can see by a simple bootstrap argument that
the convergence of $(u_{N_c})_{N_c \in \N}$
to $u$ also holds in $H^{m+1/2-\epsilon}_\#(\Gamma)$ for each
$\epsilon > 0$. 
The upper bound in~(\ref{eq:estim_Nc_I}) is obtained from
(\ref{eq:erreur_TFW_Nc}), remarking that
\begin{eqnarray*}
0 & \le & \int_\Gamma
F(|u_{N_c}|^2)-F(|u|^2)-f(|u|^2)(|u_{N_c}|^2-|u|^2) \\
& \le & \frac{35}9 C_{\rm TF} \int_\Gamma \max(|u_{N_c}|^{4/3},|u|^{4/3})
|u_{N_c}-u|^2 \\
& \le & \frac{35}9 C_{\rm TF} \left(\max_{N_c \in \N}
  \|u_{N_c}\|_{L^\infty}\right)^{4/3} \, \|u_{N_c}-u\|_{L^2_\#}^2,
\end{eqnarray*}
and that
\begin{eqnarray*}
0 & \le & D_\Gamma(|u_{N_c}|^2-|u|^2,|u_{N_c}|^2-|u|^2) \le  C
\||u_{N_c}|^2-|u|^2\|_{L^2_\#}^2 \\ 
& \le & 4 C  \, \left(\max_{N_c \in \N} \|u_{N_c}\|_{L^\infty}\right)^2
 \, \|u_{N_c}-u\|_{L^2_\#}^2.
\end{eqnarray*}

\medskip

\noindent
The uniqueness of $u_{N_c}$ for $N_c$ large enough can then be checked
as follows. First, $(u_{N_c},\lambda_{N_c})$ satisfies the variational equation
$$
\forall v_{N_c} \in V_{N_c}, \quad 
\langle
({\cal H}^{{\rm TFW}}_{|u_{N_c}|^2}-\lambda_{N_c})u_{N_c},v_{N_c}\rangle_{H^{-1}_\#,H^1_\#}
= 0.  
$$
Therefore $\lambda_{N_c}$ is the variational approximation  in $V_{N_c}$
of some eigenvalue of ${\cal H}^{{\rm TFW}}_{|u_{N_c}|^2}$. As $(u_{N_c})_{N_c \in
  \N}$ converges to $u$ in $L^\infty(\Gamma)$,
${\cal H}^{{\rm TFW}}_{|u_{N_c}|^2}-{\cal H}^{{\rm TFW}}_{\rho^0}$ converges to~$0$ in operator
norm. Consequently, the $n^{\rm th}$ eigenvalue of ${\cal H}^{{\rm TFW}}_{|u_{N_c}|^2}$
converges to the $n^{\rm th}$ eigenvalue of ${\cal H}^{{\rm TFW}}_{\rho^0}$ when $N_c$ goes
to infinity, the convergence being uniform in $n$. Together
with the fact that the sequence $(\lambda_{N_c})_{N_c \in \N}$ converges
to $\lambda$, the non-degenerate ground state eigenvalue of
${\cal H}^{{\rm TFW}}_{\rho^0}$, this implies that for $N_c$ large enough, $\lambda_{N_c}$
is the ground state eigenvalue of ${\cal H}^{{\rm TFW}}_{|u_{N_c}|^2}$ in $V_{N_c}$ and
for all $v_{N_c} \in V_{N_c}$ such that $\|v_{N_c}\|_{L^2_\#}=\cN^{1/2}$ and
$(v_{N_c},u_{N_c})_{L^2_\#} \ge 0$,
\begin{eqnarray}
E^{\rm TFW}(v_{N_c})- E^{\rm TFW}(u_{N_c}) & = & \langle
({\cal H}^{{\rm TFW}}_{|u_{N_c}|^2}-\lambda_{N_c})(v_{N_c}-u_{N_c}),(v_{N_c}-u_{N_c})\rangle_{H^{-1}_\#,H^1_\#}  \nonumber
\\
&& + \frac 12 D_\Gamma(|v_{N_c}|^2-|u_{N_c}|^2,|v_{N_c}|^2-|u_{N_c}|^2)\nonumber \\
&& + \int_\Gamma
F(|v_{N_c}|^2)-F(|u_{N_c}|^2)-f(|u_{N_c}|^2)(|v_{N_c}|^2-|u_{N_c}|^2)
\nonumber \\
& \ge &  \langle
({\cal H}^{{\rm TFW}}_{|u_{N_c}|^2}-\lambda_{N_c})(v_{N_c}-u_{N_c}),(v_{N_c}-u_{N_c})\rangle_{H^{-1}_\#,H^1_\#} \nonumber
\\
& \ge & \frac{\gamma}2 \|v_{N_c}-u_{N_c}\|_{H^1_\#}^2. \label{eq:pre_uniqueness}
\end{eqnarray}
It easily follows that for $N_c$ large enough, (\ref{eq:minTFWuV}) has a
unique minimizer $u_{N_c}$ such that $(u_{N_c},u)_{L^2_\#} \ge 0$.

\medskip

Let us now establish the rates of convergence of
$|\lambda_{N_c}-\lambda|$ and $\|u_{N_c}-u\|_{H^s_\#}$. First,
\begin{eqnarray}\label{eq:val_ppe}
\lambda_{N_c}-\lambda & = &
\cN^{-1} \bigg[
\langle ({\cal H}^{{\rm
    TFW}}_{|u|^2}-\lambda)(u_{N_c}-u),(u_{N_c}-u)\rangle_{H^{-1}_\#,H^1_\#} 
\nonumber \\
&& \qquad \qquad 
+ \int_\Gamma w_{N_c} (u_{N_c}-u) \bigg]
\end{eqnarray}
with
$$
w_{N_c}=\frac{f(|u_{N_c}|^2)-f(|u|^2)}{u_{N_c}-u}|u_{N_c}|^2
+ V^{\rm Coulomb}_{|u_{N_c}|^2}(u_{N_c}+u).
$$
As $u_{N_c}$ is bounded away
from $0$ and $f \in
C^\infty((0,+\infty))$, the function $w_{N_c}$ is uniformly bounded in 
$H^{m-3/2-\epsilon}_\#(\Gamma)$ (at least for $N_c$ large enough). We
therefore obtain that for all $0 \le r < m-3/2$, there exists a constant
$C_r \in \R_+$ such that for all $N_c$ large enough,
\begin{equation} \label{eq:pre_estim_lambda}
|\lambda_{N_c}-\lambda| \le C_r \left( \|u_{N_c}-u\|_{H^1_\#}^2
+ \| u_{N_c}-u \|_{H^{-r}_\#} \right).
\end{equation}
In order to evaluate the $H^1_\#$-norm of the error $(u_{N_c}-u)$, we first
notice that 
\begin{equation}\label{eq:4:1:H1}
\forall v_{N_c} \in V_{N_c}, \quad 
\| u_{N_c}-u \|_{H^1_\#} \le \| u_{N_c} - v_{N_c}  \|_{H^1_\#} 
+\| v_{N_c}  - u \|_{H^1_\#},
\end{equation}
and that 
 \begin{eqnarray} \!\!\!\!\!\!\!\!\!\!\!\!\!\!\!\!\!\!
 \| u_{N_c}  - v_{N_c} \|_{H^1_\#}^2 & \le & \beta^{-1} \,  
\langle ({E^{\rm TFW}}''(u) - 2\lambda)(u_{N_c}  - v_{N_c} ),
(u_{N_c}  - v_{N_c} ) \rangle_{H^{-1}_\#,H^1_\#} \nonumber \\
& = & \beta^{-1}  \bigg( \langle ({E^{\rm TFW}}''(u) - 2\lambda)(u_{N_c}  - u),
(u_{N_c}  - v_{N_c} ) \rangle_{H^{-1}_\#,H^1_\#} \nonumber \\ 
&  & \quad  +\langle ({E^{\rm TFW}}''(u) - 2\lambda)(u- v_{N_c} ),(u_{N_c}  -
  v_{N_c} ) \rangle_{H^{-1}_\#,H^1_\#} \bigg). 
 \label{eq:4:3:H1}
 \end{eqnarray}
For all $z_{N_c} \in V_{N_c}$, 
\begin{eqnarray}
&& \!\!\!\!\!\!\!\!\!\!\!\!\!\!\!\!
\langle ({E^{\rm TFW}}''(u)-2\lambda)(u_{N_c}-u),z_{N_c}
\rangle_{H^{-1}_\#,H^1_\#} \nonumber\\ & = & 
- 2 \int_\Gamma [f(|u_{N_c}|^2)u_{N_c}-f(|u|^2)u - 2 f'(|u|^2)|u|^2
(u_{N_c}-u)] z_{N_c} \nonumber \\
& & - 2D_\Gamma((u_{N_c}-u)(u_{N_c}+u),(u_{N_c}-u)z_{N_c})
- 2D_\Gamma((u_{N_c}-u)^2,uz_{N_c}) \nonumber \\ &&
+ 2 (\lambda_{N_c}-\lambda) \int_\Gamma u_{N_c}z_{N_c}. \label{eq:new36}
\end{eqnarray}
On the other hand, we have for all $v_{N_c} \in V_{N_c}$ such that
$\|v_{N_c}\|_{L^2_\#}={\mathcal N}^{1/2}$, 
$$
\int_\Gamma u_{N_c} (u_{N_c}-v_{N_c}) =\cN - \int_\Gamma  u_{N_c}
v_{N_c}
= \frac 12 \|u_{N_c}-v_{N_c}\|_{L^2_\#}^2.
$$
Using (\ref{eq:fff}), (\ref{eq:estim_D3}), (\ref{eq:pre_estim_lambda})
with $r=0$ and the above
equality, we therefore obtain for all $v_{N_c} \in V_{N_c}$ such
 that $\|v_{N_c}\|_{L^2_\#}=\cN^{1/2}$,
 \begin{eqnarray}
&& 
\!\!\!\!\!\!\!\!\!\!\!\!\!
\left| \langle ({E^{\rm TFW}}''(u) - 2\lambda)(u_{N_c}-u),(u_{N_c}-v_{N_c})
  \rangle_{H^{-1}_\#,H^1_\#}\right|
\nonumber \\
&& \le  C \bigg( \|u_{N_c}-u\|^2_{H^1_\#} \, \| u_{N_c}-v_{N_c}
\|_{H^1_\#}  \nonumber \\
&& \qquad  \quad + 
\left( \|u_{N_c}-u\|_{H^1_\#}^2 + \|u_{N_c}-u\|_{L^{2}_\#} \right) 
\| u_{N_c}-v_{N_c} \|_{L^2_\#}^2  \bigg). \label{eq:4:3:H2}
 \end{eqnarray}
Therefore, for $N_c$ large enough, we have for all
$v_{N_c} \in V_{N_c}$ such that $\|v_{N_c}\|_{L^2_\#}=\cN^{1/2}$,  
$$
 \| u_{N_c}  - v_{N_c} \|_{H^1_\#} \le   C 
 \left(   \|u_{N_c}-u\|_{H^1_\#}^2  + \|v_{N_c}-u \|_{H^1_\#} \right) .
$$
Together with (\ref{eq:4:1:H1}), this shows that
there exists $N \in \N$ and $C \in \R_+$ such that for all ${N_c} \ge N$, 
$$
\forall v_{N_c} \in V_{N_c} \mbox{ s.t. } \|v_{N_c}\|_{L^2_\#}=\cN^{1/2},
\quad \|u_{N_c} -u \|_{H^1_\#} \le C  \|v_{N_c}-u \|_{H^1_\#}.   
$$
By a classical argument (see e.g. the proof of Theorem~1 in \cite{CCM}),
we deduce from~(\ref{eq:app-Fourier}) and the above inequality that
\begin{equation} \label{eq:estim_H1}
\|u_{N_c} -u \|_{H^1_\#} \le C \min_{v_{N_c} \in V_{N_c}} \|v_{N_c}-u
\|_{H^1_\#} \le C_{1,\epsilon} N_c^{-(m-1/2-\epsilon)},  
\end{equation}
for some constant $C_{1,\epsilon}$ independent of $N_c$.
This completes the proof of the estimate in the $H^1_\#$--norm. We
proceed with the analysis of the $L^2_\#$--norm. 

\medskip

\noindent
For $w \in L^{2}_\#(\Gamma)$, we denote by $\psi_w$ the unique solution to the
adjoint problem 
\begin{equation} \label{eq:adjoint}
\left\{ \begin{array}{l}
\mbox{find } \psi_w \in u^\perp \mbox{ such that} \\
\forall v \in u^\perp, \quad \langle ({E^{\rm TFW}}''(u) -
2\lambda)\psi_w,v \rangle_{H^{-1}_\#,H^1_\#} = 
\langle w,v \rangle_{H^{-1}_\#,H^1_\#},
\end{array} \right.
\end{equation}
where
$$
u^\perp = \left\{ v \in H^1_\#(\Gamma) \; | \; \int_\Gamma uv=0 \right\}.
$$
The function $\psi_w$ is solution to the elliptic equation
\begin{eqnarray*} \!\!\!\!\!\!\!\!\!\!\!\!\!\!\!\!
&& - \frac{C_{\rm W}}2 \Delta \psi_w + \left( V^{\rm
    ion}+V_{u^2}^{\rm Coulomb}+f(u^2)+2f'(u^2)u^2-\lambda \right) \psi_w 
+ 2 V_{u\psi_w}^{\rm Coulomb} u \nonumber \\ & & \qquad\qquad =
2 \left(\int_\Gamma f'(u^2)u^3\psi_w+D_{\Gamma}(u^2,u\psi_w)\right) u +
\frac 12 \left( w - (w,u)_{L^2_\#} u \right), 
\end{eqnarray*}
from which we deduce that if $w \in H^r_\#(\Gamma)$ for some $0 \le r <
m-3/2$, then $\psi_w \in H^{r+2}_\#(\Gamma)$ and  
\begin{equation} \label{eq:borne_psiw}
\| \psi_w \|_{H^{r+2}_\#} \le C_r \| w \|_{H^{r}_\#}, 
\end{equation}
for some constant $C_r$ independent of $w$.
Let $u_{N_c}^*$ be the orthogonal projection, for the $L^2_\#$ inner
product, of  
$u_{N_c}$ on the affine space $\left\{v \in L^2_\#(\Gamma) \, | \,
  \int_\Gamma uv=\cN \right\}$. One has 
$$
u_{N_c} ^* \in H^1_{\#}(\Gamma), \qquad u_{N_c} ^*-u \in u^\perp,  \qquad 
u_{N_c}^*-u_{N_c}  = \frac 1{2\cN} \|u_{N_c}-u \|_{L^2_\#}^2 u, 
$$
from which we infer that
\begin{eqnarray*}
\|u_{N_c} -u\|_{L^2_\#}^2 & = & \int_\Gamma (u_{N_c}-u )(u_{N_c} ^*-u) + 
\int_\Gamma (u_{N_c}-u)(u_{N_c} - u_{N_c} ^*) \\
& = &  \int_\Gamma (u_{N_c}-u)(u_{N_c}^*-u) - \frac 1{2\cN}
\|u_{N_c}-u\|_{L^2_\#}^2 \int_\Gamma (u_{N_c}-u) u  \\
& = &  \int_\Gamma (u_{N_c}-u)(u_{N_c}^*-u) + \frac 1{2\cN}
\|u_{N_c}-u\|_{L^2_\#}^2 \left( \cN - \int_\Gamma u_{N_c} u \right) \\
& = &  \int_\Gamma (u_{N_c}-u)(u_{N_c}^*-u) + \frac 1{4\cN}
\|u_{N_c}-u\|_{L^2_\#}^4 \\ 
& = &   \langle u_{N_c}-u  ,u_{N_c}^*-u \rangle_{H^{-1}_\#,H^1_\#}  + \frac 1{4\cN}
\|u_{N_c}-u\|_{L^2_\#}^4  \\ 
& = &   \langle ({E^{\rm TFW}}''(u)-2\lambda) \psi_{u_{N_c}-u}, u_{N_c}^*-u
\rangle_{H^{-1}_\#,H^1_\#} + \frac 1{4\cN}
\|u_{N_c}-u\|_{L^2_\#}^4 \\
& = &   \langle  ({E^{\rm TFW}}''(u)-2\lambda)  (u_{N_c}-u), \psi_{u_{N_c}-u}
\rangle_{H^{-1}_\#,H^1_\#} +  \frac 1{4\cN}
\|u_{N_c}-u\|_{L^2_\#}^4 \\ 
& & + \frac 1{2\cN} \|u_{N_c}-u\|_{L^2_\#}^2 
 \langle  ({E^{\rm TFW}}''(u)-2\lambda) u , \psi_{u_{N_c}-u} \rangle_{H^{-1}_\#,H^1_\#} 
 \\ 
& = &   \langle  ({E^{\rm TFW}}''(u)-2\lambda)  (u_{N_c}-u), \psi_{u_{N_c}-u}
\rangle_{H^{-1}_\#,H^1_\#} + \frac 1{4\cN} \|u_{N_c}-u \|_{L^2_\#}^4 \\ 
& & +  \frac{2}{\cN} \|u_{N_c}-u\|_{L^2_\#}^2 \left[\int_\Gamma f'(u^2) u^3
  \psi_{u_{N_c}-u} + D_\Gamma(u^2,u \psi_{u_{N_c}-u}) \right] .
\end{eqnarray*}
For all $\psi_{N_c} \in V_{N_c}$, it therefore holds
\begin{eqnarray} \!\!\!\!\!\!\!\!\!\!\!\!\!\!\!\!\!
\|u_{N_c}-u\|_{L^2}^2 & = &   
\langle  ({E^{\rm TFW}}''(u)-2\lambda)  (u_{N_c}-u), \psi_{u_{N_c}-u}-\psi_{N_c}
\rangle_{H^{-1}_\#,H^1_\#} \nonumber \\
& & + \langle  ({E^{\rm TFW}}''(u)-2\lambda)  (u_{N_c}-u),
\psi_{N_c} \rangle_{H^{-1}_\#,H^1_\#} + \frac 1{4\cN} \|u_{N_c}-u
\|_{L^2_\#}^4  \nonumber  \\ 
& & +  \frac{2}{\cN} \|u_{N_c}-u\|_{L^2_\#}^2 \left[\int_\Gamma f'(u^2) u^3
  \psi_{u_{N_c}-u} + D_\Gamma(u^2,u \psi_{u_{N_c}-u}) \right] .
\label{eq:eeeee}
\end{eqnarray}
Using  (\ref{eq:fff}), (\ref{eq:estim_D3}), (\ref{eq:pre_estim_lambda})
with $r=0$ and (\ref{eq:new36}), we obtain that for all $\psi_{N_c} \in V_{N_c}
\cap u^\perp$,
\begin{eqnarray} \!\!\!\!\!\!\!\!\!\!\!\!\!\!\!\!\!\!   
\left| \langle ({E^{\rm TFW}}(u) -2 \lambda)(u_{N_c}-u),\psi_{N_c}
  \rangle_{H^{-1}_\#,H^1_\#}\right|
& \le & C \bigg(  \|u_{N_c}-u\|_{H^{1}_\#}^2 
\nonumber \\ & &
\!\!\!\!\!\!\!\!\!\!\!\!\!\!\!\!\!\!\!\!\!\!\!\!\!\!\!\!\!\!\!\!\!\!\!\!\!\!\!\!\!\!\!\!\!\!\!\!\!\!\!\!\!\!\!\!\!\!\!\!\!\!\!\!\!\!\!\!\!\!\!\!\!\!\!\!\!\!\!\!\!\!\!\!\!\!\!\!\!\!\!\!\!\!\!\!\!\!\!\!\!\!
+ \| u_{N_c}-u \|_{L^{2}_\#} 
\left( \|u_{N_c}-u\|_{H^1_\#}^2 + \|u_{N_c}-u\|_{L^{2}_\#} \right) 
 \bigg) \|\psi_{N_c}\|_{H^1_\#}.  \label{eq:new37}
 \end{eqnarray}
Let us denote by $\Pi^1_{V_{N_c} \cap u^\perp}$ the orthogonal projector
on $V_{N_c} \cap u^\perp$ for the $H^1_\#$ inner product and by 
$\psi_{N_c}^0 = \Pi^1_{V_{N_c} \cap u^\perp} \psi_{u_{N_c}-u}$.
Noticing that 
$$
\|\psi_{N_c}^0\|_{H^1_\#} \le  \|\psi_{u_{N_c}-u}\|_{H^1_\#}
 \le \beta^{-1} M\|u_{N_c}-u\|_{L^2_\#},
$$
we obtain from (\ref{eq:NRJsecondContinue}), (\ref{eq:eeeee}) and
(\ref{eq:new37}) that there exists $N \in \N$ and $C \in 
\R_+$ such that for all ${N_c} \ge N$,
$$
\|u_{N_c}-u\|_{L^2_\#}^2 \le  C \, \bigg(
  \|u_{N_c}-u\|_{L^2_\#} \|u_{N_c}-u\|_{H^{1}_\#}^2 
+   \|u_{N_c}-u\|_{H^1_\#}
\|\psi_{u_{N_c}-u}-\psi_{N_c}^0\|_{H^1_\#} \bigg). 
$$
Lastly, for all $v \in u^\perp$ and all $N_c \in \N^\ast$ 
\begin{equation} \label{eq:minuperp}
\|v - \Pi^1_{V_{N_c} \cap u^\perp} v \|_{H^1_\#}
\le \left(1+\frac{\cN^{1/2}}{2\pi L^{1/2} N_c \int_\Gamma u} \right)  
\|v- \Pi_{N_c} v \|_{H^1_\#},
\end{equation}
so that, in view of (\ref{eq:app-Fourier}) and (\ref{eq:borne_psiw}) 
\begin{eqnarray*}
\|\psi_{u_{N_c}-u}-\psi_{N_c}^0\|_{H^1_\#} &\le& C  \|\psi_{u_{N_c}-u}-
\Pi_{N_c} \psi_{u_{N_c}-u} \|_{H^1_\#} \\ &\le& C N_c^{-1}
\|\psi_{u_{N_c}-u} \|_{H^2_\#} \\ &\le &C N_c^{-1} \|u_{N_c}-u\|_{L^2_\#}.
\end{eqnarray*}
Therefore, 
\begin{eqnarray*}
\|u_{N_c}-u\|_{L^2_\#} & \le &  C \,   \bigg(
 \|u_{N_c}-u\|_{H^{1}_\#}^2  +  N_c^{-1} \|u_{N_c}-u\|_{H^1_\#}  \bigg)
 \\
& \le & C_{0,\epsilon} N_c^{-(m+1/2-\epsilon)}.
\end{eqnarray*}
By means of the inverse inequality
\begin{equation} \label{eq:inv_ineq}
\forall v_{N_c} \in V_{N_c}, \quad 
\|v_{N_c}\|_{H^r_\#} \le \left( \frac{2\pi}L \right)^{(r-s)}  {N_c}^{r-s} \|v_{N_c}\|_{H^s_\#}, 
\end{equation}
which holds true for all $s \le r$ and all ${N_c} \ge 1$,
we obtain that
\begin{equation}  \label{eq:error_H1_F}
\|u_{N_c}-u\|_{H^s_\#}  \le  C_{s,\epsilon} N_c^{-(m-s+1/2-\epsilon)} \qquad 
\mbox{for all } 0 \le s < m+1/2.
\end{equation}
To complete the first part of the proof of Theorem~\ref{Th:TFW}, we
still have to compute the $H^{-r}_\#$-norm of the error $(u_{N_c}-u)$
for $0 < r < m-3/2$. Let $w \in H^{r}_\#(\Gamma)$. Proceeding as above
we obtain   
\begin{eqnarray}
 \int_\Gamma w (u_{N_c}-u) & = &
\langle  ({E^{\rm TFW}}''(u)-2\lambda)(u_{N_c}-u), \Pi^1_{V_{N_c} \cap
  u^\perp} \psi_{w} 
\rangle_{H^{-1}_\#,H^1_\#}  \nonumber \\
& & + \langle  ({E^{\rm
    TFW}}''(u)-2\lambda)(u_{N_c}-u),\psi_{w}-\Pi^1_{V_{N_c} 
  \cap u^\perp}\psi_{w} 
\rangle_{H^{-1}_\#,H^1_\#} \nonumber \\ 
& &+ \frac{2}{\cN} \|u_{N_c}-u\|_{L^2_\#}^2\left[ \int_\Gamma f'(u^2) u^3
  \psi_{w}+ D_\Gamma(u^2,u\psi_w) \right] 
\nonumber\\&&\qquad- \frac 1{2\cN} \|u_{N_c}-u \|_{L^2_\#}^2 \int_\Gamma uw.
\label{eq:intom} 
\end{eqnarray}
Combining
 (\ref{eq:NRJsecondContinue}), (\ref{eq:borne_psiw}), (\ref{eq:new37}), 
(\ref{eq:minuperp}), (\ref{eq:error_H1_F}) and (\ref{eq:intom}), we
obtain that there exists a constant $C \in \R_+$ such that for all
${N_c}$ large enough and all $w \in H^{r}_\#( \Gamma)$,
\begin{eqnarray*}
\int_ \Gamma w (u_{N_c}-u) & \le & C' \left( \|u_{N_c}-u\|_{H^1_\#}^2 
+ {N_c}^{-(r+1)} \|u_{N_c}-u\|_{H^1_\#} \right) \|w\|_{H^r_\#} \\
 & \le & C_{-r,\epsilon} \, N_c^{-(m+r+1/2-\epsilon)} \|w\|_{H^{r}_\#} .
\end{eqnarray*}

Therefore
\begin{equation} \label{eq:Hm1boundFourier}
\|u_{N_c}-u\|_{H^{-r}_\#} = \sup_{w \in H^{r}_\#( \Gamma) \setminus
  \left\{0\right\}} \frac{\dps \int_ \Gamma w (u_{N_c}-u)}{\|w\|_{H^{r}_\#}}
\le C_{-r,\epsilon}  \, N_c^{-(m+r+1/2-\epsilon)} ,
\end{equation}
for some constant $C_{-r,\epsilon}  \in \R_+$ independent of ${N_c}$.
Using (\ref{eq:pre_estim_lambda}), (\ref{eq:estim_H1}) and (\ref{eq:Hm1boundFourier}), we end up with 
$$
|\lambda_{N_c}-\lambda| \le  C_{\epsilon}  N_c^{-(2m-1-\epsilon)} .
$$

\medskip
\subsection{A priori estimates for the full discretization.}

Let us now turn to the pseudospectral approximation (\ref{eq:minTFWuN})
of (\ref{eq:minTFWu}). First, we notice that
\begin{eqnarray*}
\frac{C_{\rm W}}2 \|\nabla u_{N_c,N_g}\|_{L^2_\#}^2 - \|V^{\rm
  ion}\|_{L^\infty} {\mathcal N} &\le & E^{\rm TFW}_{N_g}(u_{N_c,N_g})
\\ & \le & E^{\rm TFW}_{N_g}({\mathcal N}^{1/2}|\Gamma|^{-1/2}) \\
& \le &C_{\rm TF} {\mathcal N}^{5/3} |\Gamma|^{-2/3} + \|V^{\rm
  ion}\|_{L^\infty} {\mathcal N},
\end{eqnarray*}
from which we infer that $u_{N,N_g}$ is uniformly bounded in
$H^1_\#(\Gamma)$. We then see that
\begin{eqnarray*}
\lambda_{N_c,N_g} &=& \cN^{-1} \bigg[
\frac{C_{\rm W}}2 \int_\Gamma |\nabla u_{N_c,N_g}|^2 + 
\int_\Gamma \cI_{N_g}( V^{\rm
  ion}|u_{N_c,N_g}|^2+f(|u_{N_c,N_g}|^2)|u_{N_c,N_g}|^2)  \\ && \qquad\qquad
+D_\Gamma(|u_{N_c,N_g}|^2,|u_{N_c,N_g}|^2) \bigg].
\end{eqnarray*}
Using (\ref{eq:bound_INg1}), (\ref{eq:bound_INg2}) and
(\ref{eq:estim_D1}), we obtain that $\lambda_{N_c,N_g}$ 
also is uniformly bounded. Now,
\begin{eqnarray}
\Delta u_{N_c,N_g} & = & 2 C_{\rm W}^{-1} \Pi_{N_c}
\left(\cI_{N_g}\left(f(|u_{N_c,N_g}|^2)u_{N_c,N_g}\right) \right) + 
2 C_{\rm W}^{-1} \Pi_{N_c} \left( \cI_{N_g}\left(
  V^{\rm ion}u_{N_c,N_g} \right) \right) 
\nonumber \\ && + 2 C_{\rm W}^{-1}
\Pi_{N_c} \left(  V^{\rm Coulomb}_{|u_{N_c,N_g}|^2}  u_{N_c,N_g} \right) - 2 C_{\rm W}^{-1} 
\lambda_{N_c,N_g} u_{N_c,N_g}, \label{eq:EDP_uNcNg}
\end{eqnarray}
and we deduce from (\ref{eq:exact_integration}), (\ref{eq:bound_INg1})
and (\ref{eq:ineg_INg_4}) that 
\begin{eqnarray*}
\left\| \Pi_{N_c} \left( \cI_{N_g}\left( f(|u_{N_c,N_g}|^2)u_{N_c,N_g}
    \right) \right) \right\|_{L^2_\#}
& \le & \left( \int_\Gamma
  \left(\cI_{N_g}(f(|u_{N_c,N_g}|^2))\right)^2 |u_{N_c,N_g}|^2 \right)^{1/2} \\
& = & \left( \sum_{x \in \cG_{N_g} \cap \Gamma}
 \left( \frac{L}{N_g} \right)^3
  f(|u_{N_c,N_g}(x)|^2)^2|u_{N_c,N_g}(x)|^2\right)^{1/2} \\
& \le & \frac 53 C_{\rm TF} \|u_{N_c,N_g}\|_{L^\infty}^{1/3} 
   \left( \sum_{x \in \cG_{N_g}  \cap \Gamma}
\left( \frac{L}{N_g} \right)^3 |u_{N_c,N_g}(x)|^4 \right)^{1/2} \\
& = & \frac 53 C_{\rm TF} \|u_{N_c,N_g}\|_{L^\infty}^{1/3}
\|u_{N_c,N_g}\|_{L^4_\#}^2, 
\end{eqnarray*}
and that
\begin{eqnarray*}
\|\Pi_{N_c} \left( \cI_{N_g}\left(
  V^{\rm ion}u_{N_c,N_g} \right) \right) \|_{L^2_\#} & \le & 
\|\Pi_{2N_c} \left( \cI_{N_g}\left(
  V^{\rm ion}u_{N_c,N_g} \right) \right) \|_{L^2_\#} \\
& \le & \left( \int_\Gamma \cI_{N_g}(|V^{\rm ion}|^2|u_{N_c,N_g}|^2)
\right)^{1/2} \\
& \le & \|V^{\rm ion}\|_{L^\infty} \cN^{1/2}.
\end{eqnarray*}
Besides, using (\ref{eq:estim_D4}), 
\begin{eqnarray*}
\|\Pi_{N_c} \left( V^{\rm Coulomb}_{|u_{N_c,N_g}|^2} u_{N_c,N_g} \right)\|_{L^2_\#} & \le &
\|V^{\rm Coulomb}_{|u_{N_c,N_g}|^2} u_{N_c,N_g} \|_{L^2_\#} \\
 & \le & \cN^{1/2} \|V^{\rm Coulomb}_{|u_{N_c,N_g}|^2} \|_{L^\infty} \\
& \le & \cN^{1/2} \|u_{N_c,N_g}\|_{L^4_\#}^2.
\end{eqnarray*}
As $u_{N_c,N_g}$ is uniformly bounded in $H^1_\#(\Gamma)$, and therefore
in $L^4_\#(\Gamma)$, we get 
\begin{eqnarray*}
\|u_{N_c,N_g}\|_{H^2_\#} & = & \left( \|u_{N_c,N_g}\|_{L^2_\#}^2 +
  \|\Delta u_{N_c,N_g}\|_{L^2_\#}^2  \right)^{1/2} \\
& \le & C \left( 1 +   \|u_{N_c,N_g}\|_{L^\infty}^{1/3} \right) \\
& \le & C \left( 1 +   \|u_{N_c,N_g}\|_{H^2_\#}^{1/3} \right).
\end{eqnarray*}
Therefore $u_{N_c,N_g}$ is uniformly bounded in $H^2_\#(\Gamma)$, hence
in $L^\infty(\R^3)$.

Returning to (\ref{eq:EDP_uNcNg}) and using
(\ref{eq:ineg_INg_5}), (\ref{eq:hypothesis}), and a bootstrap argument, we conclude that
$u_{N_c,N_g}$ is in fact uniformly bounded in $H^{7/2+\epsilon}_\#(\Gamma)$. 

Next, using (\ref{eq:pre_uniqueness}),
\begin{eqnarray*}
\frac\gamma 2 \|u_{N_c,N_g}-u_{N_c}\|_{H^1_\#}^2 & \le & E^{\rm TFW}(u_{N_c,N_g})-E^{\rm TFW}(u_{N_c}) \\
& = & E^{\rm TFW}_{N_g}(u_{N_c,N_g})-E^{\rm TFW}_{N_g}(u_{N_c}) \\ && 
+ \int_\Gamma ((1-\cI_{N_g})(V))(|u_{N_c,N_g}|^2-|u_{N_c}|^2) \\ && 
+ \int_\Gamma (1-\cI_{N_g})(F(|u_{N_c,N_g}|^2)-F(|u_{N_c}|^2)) \\
& \le & \int_\Gamma ((1-\cI_{N_g})(V))(|u_{N_c,N_g}|^2-|u_{N_c}|^2) \\ && 
+ \int_\Gamma (1-\cI_{N_g})(F(|u_{N_c,N_g}|^2)-F(|u_{N_c}|^2)). 
\end{eqnarray*}
Let $g(t,t')=\frac{F(t'^2)-F(t^2)}{t'-t}$. For $N_c$ large enough,
$u_{N_c}$ is uniformly bounded away from zero; besides,
both $u_{N_c}$ and $u_{N_c,N_g}$ are uniformly bounded in
$H^{7/2+\epsilon}_\#(\Gamma)$. Therefore, $g(u_{N_c},u_{N_c,N_g})$ is uniformly bounded in
$H^{7/2+\epsilon}_\#(\Gamma)$. This implies that the Fourier coefficients of
$g(u_{N_c},u_{N_c,N_g})$ go to zero faster that $|k|^{-7/2}$, which
in turn implies, using (\ref{eq:ineg_INg_2}) and (\ref{eq:ineg_INg_6}), that
\begin{eqnarray}
& & \left| \int_\Gamma (1-\cI_{N_g})(F(|u_{N_c,N_g}|^2)-F(|u_{N_c}|^2))
\right| \nonumber \\
& & \qquad = 
\left|\int_\Gamma
  (1-\cI_{N_g})\left(g(u_{N_c},u_{N_c,N_g})\right) \,
  (u_{N_c,N_g}-u_{N_c}) 
\right| \nonumber \\
& & \qquad \le \left\|\Pi_{N_c} \left(
    (1-\cI_{N_g})\left(g(u_{N_c},u_{N_c,N_g})\right) \right)
\right\|_{L^2_\#} \|u_{N_c,N_g}-u_{N_c}\|_{L^2_\#}
 \nonumber \\
& & \qquad \le C N_c^{3/2} N_g^{-7/2} \|u_{N_c,N_g}-u_{N_c}\|_{L^2_\#}.
\label{eq:bound_CC1}
\end{eqnarray}
On the other hand,
\begin{eqnarray*}
&& \left| \int_\Gamma ((1-\cI_{N_g})(V))(|u_{N_c,N_g}|^2-|u_{N_c}|^2)
\right| \\ && \qquad \le 
\|\Pi_{2N_c} ((1-\cI_{N_g})(V))\|_{L^2_\#}
\|u_{N_c,N_g}+u_{N_c}\|_{L^\infty} \|u_{N_c,N_g}-u_{N_c}\|_{L^2_\#} \\
&& \qquad \le C N_c^{3/2} N_g^{-m} \|u_{N_c,N_g}-u_{N_c}\|_{L^2_\#}.
\end{eqnarray*}
Therefore, 
\begin{eqnarray}
 \|u_{N_c,N_g}-u_{N_c}\|_{H^1_\#} & \le & C N_c^{3/2} N_g^{-7/2}.
\label{eq:bound_CC2}
\end{eqnarray}
We then deduce from (\ref{eq:bound_CC2}) and the inverse inequality
(\ref{eq:inv_ineq}) that $(u_{N_c,N_g})_{N_c,N_g \ge 4N_c+1}$ converges to $u$ in
$H^2_\#(\Gamma)$, and therefore in $L^\infty(\R^3)$. It follows that for
$N_c$ large enough, $u_{N_c,N_g}$ is bounded away from zero, which,
together with (\ref{eq:EDP_uNcNg}), implies that $(u_{N_c,N_g})_{N_c,N_g
  \ge 4N_c+1}$  is bounded in $H^{m+1/2-\epsilon}_\#(\Gamma)$. The
estimates (\ref{eq:bound_CC1}) and (\ref{eq:bound_CC2}) can therefore be
improved, yielding 
$$
\left| \int_\Gamma (1-\cI_{N_g})(F(|u_{N_c,N_g}|^2)-F(|u_{N_c}|^2))
\right|  \le C N_c^{3/2} N_g^{-(m+1/2-\epsilon)} \|u_{N_c,N_g}-u_{N_c}\|_{L^2_\#}.
$$
and
$$
 \|u_{N_c,N_g}-u_{N_c}\|_{H^1_\#}  \le  C N_c^{3/2} N_g^{-m}.
$$
We deduce (\ref{eq:estim_NcNg_u}) from the inverse inequality
(\ref{eq:inv_ineq}). For $N_c$ large enough, $u_{N_c,N_g}$ is bounded
away from zero, so that $f(|u_{N_c,N_g}|^2)$ is uniformly
bounded in $H^{m+1/2-\epsilon}_\#(\Gamma)$. Therefore, the $k^{\rm th}$ Fourier
coefficient of $(V^{\rm ion}+f(|u_{N_c,N_g}|^2))$ is bounded by 
$C|k|^{-m}$ where the constant $C$ does not depend on $N_c$ and $N_g$. 
Using the equality
\begin{eqnarray*}
\lambda_{N_c,N_g}-\lambda_{N_c} & = & \cN^{-1} \bigg[ \langle
({\cal H}^{{\rm TFW}}_{|u_{N_c}|^2}-\lambda_{N_c})(u_{N_c,N_g}-u_{N_c}),(u_{N_c,N_g}-u_{N_c})
\rangle_{H^{-1}_\#,H^1_\#} \\
&& \qquad - \int_\Gamma (1-\cI_{N_g})(V^{\rm ion}+
f(|u_{N_c,N_g}|^2))|u_{N_c,N_g}|^2 \\
&& \qquad + D_\Gamma(|u_{N_c,N_g}|^2,|u_{N_c,N_g}|^2-|u_{N_c}|^2) \\
&& \qquad 
+ \int_\Gamma (f(|u_{N_c,N_g}|^2)-f(|u_{N_c}|^2)) |u_{N_c,N_g}|^2 \bigg],
\end{eqnarray*}
(\ref{eq:estim_NcNg_u}) and (\ref{eq:estim_D3}),
 we obtain  (\ref{eq:estim_NcNg_lambda}). A similar calculation leads to 
(\ref{eq:estim_NcNg_I}).

Lastly, we have for all $v_{N_c} \in V_{N_c}$, 
\begin{eqnarray}
&& \!\!\!\!\!\!\!\!\!\!\!\!\!\!\!\!
E^{\rm TFW}_{N_g}(v_{N_c})- E^{\rm TFW}_{N_g}(u_{N_c,N_g}) \\
& = & \langle
(\widetilde
{\cal H}^{{\rm TFW}}_{u_{N_c,N_g}}-\lambda_{N_c,N_g})(v_{N_c}-u_{N_c,N_g}),(v_{N_c}-u_{N_c,N_g})\rangle_{H^{-1}_\#,H^1_\#}
\nonumber 
\\
&& + \frac 12 D_\Gamma(|v_{N_c}|^2-|u_{N_c,N_g}|^2,|v_{N_c}|^2-|u_{N_c,N_g}|^2)\nonumber \\
&& + \sum_{x \in \cG_{N_g} \cap \Gamma}
\left( \frac{L}{N_g} \right)^3 \left(
F(|v_{N_c}(x)|^2)-F(|u_{N_c}(x)|^2)-f(|u_{N_c}(x)|^2)(|v_{N_c}(x)|^2-|u_{N_c}(x)|^2) \right)
\nonumber \\
& \ge &   \langle
(\widetilde
{\cal H}^{{\rm TFW}}_{u_{N_c,N_g}}-\lambda_{N_c,N_g})(v_{N_c}-u_{N_c,N_g}),(v_{N_c}-u_{N_c,N_g})\rangle_{H^{-1}_\#,H^1_\#}.
\end{eqnarray}
As $u_{N_c,N_g}$ converges to $u$ in $H^2_\#(\Gamma)$, the operator $\widetilde
{\cal H}^{{\rm TFW}, N_g}_{|u_{N_c,N_g}|^2}-{\cal H}^{{\rm TFW}}_{\rho^0}$ converges to zero in operator
norm. Reasoning as in the proof of the uniqueness of $u_{N_c}$, we
obtain that for $N_c$ large enough and $N_g \ge 4N_c+1$, we have for all
$v_{N_c} \in V_{N_c}$ such that $\|v_{N_c}\|_{L^2_\#}=\cN^{1/2}$ and 
$(v_{N_c},u_{N_c})_{L^2_\#} \ge 0$,
$$
 \langle
(\widetilde
{\cal H}^{{\rm TFW}}_{u_{N_c,N_g}}-\lambda_{N_c,N_g})(v_{N_c}-u_{N_c,N_g}),(v_{N_c}-u_{N_c,N_g})\rangle_{H^{-1}_\#,H^1_\#}
\ge \frac \gamma 2 \|v_{N_c}-u_{N_c,N_g}\|_{H^1_\#}^2.
$$
Thus the uniqueness of $u_{N_c,N_g}$ for $N_c$ large enough.

\section{Planewave approximation of the Kohn-Sham LDA model}
\label{sec:KS}

The periodic Kohn-Sham LDA model with norm-conserving pseudopotentials~\cite{pseudo} leads to the constrained optimization problem 
\begin{equation} \label{eq:min_KS}
I^{\rm KS} = \inf \left\{ E^{\rm KS}(\Phi), \; \Phi \in {\cal M} \right\}
\end{equation}
where
$$
{\cal M} = \left\{ \Phi=(\phi_1,\cdots,\phi_\cN)^T \in
  (H^1_\#(\Gamma))^\cN \; | \; \int_\Gamma \phi_i \phi_j =
  \delta_{ij}  \right\}, 
$$
 $\cN$ being the number of valence electron pairs in the simulation cell,
and where
\begin{equation} \label{eq:energy_KS}
E^{\rm KS}(\Phi) = \sum_{i=1}^\cN \int_\Gamma |\nabla \phi_i|^2 + 
\int_\Gamma \rho_\Phi V_{\rm local} + 2 \sum_{i=1}^\cN \langle 
\phi_i|V_{\rm nl}|\phi_i\rangle +  J(\rho_\Phi) + E_{\rm xc}^{\rm
  LDA}(\rho_\Phi).  
\end{equation}
The density $\rho_\Phi$ associated with $\Phi$, the Coulomb
energy $J(\rho_\Phi)$ and the LDA exchange-correlation energy $E_{\rm
  xc}^{\rm LDA}(\rho_\Phi)$ are respectively defined as
\begin{eqnarray*}
\rho_\Phi(x) &=& 2 \sum_{i=1}^\cN |\phi_i(x)|^2, \label{eqrhodef}\\
 J(\rho_\Phi) &=& \frac 12 D_\Gamma(\rho_\Phi,\rho_\Phi) = 2\pi
\sum_{k \in \cR^\ast \setminus \left\{0\right\}} |k|^{-2} |\widehat
  {(\rho_\Phi)}_k|^2, \\
E_{\rm xc}^{\rm LDA}(\rho_\Phi) &=& \int_\Gamma
e_{\rm xc}^{\rm LDA}(\rho_{\rm c}(x)+\rho_\Phi(x)) \, dx,   
\end{eqnarray*} 
where $\rho_{\rm c} \ge 0$ is the nonlinear core correction and 
where $e_{\rm xc}^{\rm LDA}(\rho)$ is an approximation of
the exchange-correlation energy per unit volume in a uniform electron gas
with charge density $\rho$ \cite{DreizlerGross}. 

The local and nonlocal contributions to the pseudopotential 
model the interactions between valence electrons on the one hand, and
nuclei and core electrons on the other hand. Troullier-Martins
pseudopotentials \cite{pseudo} constitute a popular class of
pseudopotentials for 
which the Fourier coefficients $\widehat{(V_{\rm local})}_k$ decay as
$|k|^{-m}$ with $m=5$. The nonlocal contribution is defined by  
$$
V_{\rm nl} \phi = \sum_{j=1}^M (\chi_j ,\phi)_{L^2_\#} \, \chi_j,
$$
where the functions $\chi_j$ are regular enough functions of $L^2_\#(\Gamma)$. 
In all what follows, we will assume that 
\begin{equation} \label{eq:decay_V_local}
\exists m > 3, \; C \ge 0 \mbox{ s.t. } \forall k \in \cR^\ast, \; 
| \widehat{(V_{\rm local})}_k | \le C |k|^{-m}
\end{equation}
and that 
\begin{equation} \label{eq:hyp_Vnl}
\forall 1 \le j \le M, \quad \forall \epsilon > 0, \quad  \chi_j\in
H^{m-3/2-\epsilon}_\#(\Gamma). 
\end{equation}
The function $\rho \mapsto e_{\rm
  xc}^{\rm LDA}(\rho)$ 
does not have a simple analytical expression. Although this function is of
class $C^\infty$ on the open set $(0,+\infty)$, DFT simulation softwares
make use of approximate functions which are $C^\infty$ on $(0,\rho_*)
\cup (\rho_*,+\infty)$ but only $C^{1,1}$ in the neighborhood
of the density $\rho_* := 3/(4\pi)$ (atomic units) \cite{DreizlerGross}. In  
order not to deteriorate the convergence rate of the pseudospectral
approximation, it is better to ressort to more regular approximations of
the function $e^{\rm LDA}$ (see~\cite{CCEM}). We will assume 
here that 
\begin{eqnarray} &&
\mbox{the function } \rho \mapsto e_{\rm xc}^{\rm LDA}(\rho)
\mbox{ is in $C^1([0,+\infty)) \cap C^{[m]}((0,+\infty))$},\label{eq:hyp_epsilon_xc} \\
&& e_{\rm xc}^{\rm LDA}(0) = 0, \quad \frac{de_{\rm xc}^{\rm LDA}}{d\rho}(0) = 0,
\label{eq:hyp_epsilon_xc02}
\end{eqnarray}
(where $[m]$ denotes the
  integer part of $[m]$) and that there exists $0 < \alpha \le 1$ and $C
  \in \R_+$ such that
\begin{equation}\label{eq:hyp_epsilon_xc_2}
\forall \rho \in \R_+ \setminus \left\{0\right\}, \quad 
\left| \frac{d^2e_{\rm xc}^{\rm LDA}}{d\rho^2}(\rho) \right|
+ \left| \rho\frac{d^3e_{\rm xc}^{\rm LDA}}{d\rho^3}(\rho) \right|
\le C (1+\rho^{\alpha-1}).
\end{equation} 
Note that the X$\alpha$ exchange-correlation functional ($e_{\rm
  xc}^{{\rm X}\alpha}(\rho)=-C_{\rm X} \rho^{4/3}$, where $C_{\rm X} > 0$ is a
given constant) satisfies the assumptions
(\ref{eq:hyp_epsilon_xc})-(\ref{eq:hyp_epsilon_xc_2}) with $\alpha =
1/3$. Let us also
remark that (\ref{eq:hyp_epsilon_xc}) and (\ref{eq:hyp_epsilon_xc_2})
imply that 
\begin{equation} \label{eq:exc-regularity}
e_{\rm xc}^{\rm LDA} \in C^{1,\alpha}([0,L]) \quad \mbox{for each } L > 0,
\end{equation}
 a property we
will make use of below. Lastly, we assume for simplicity that
\begin{equation} \label{eq:rhoc}
  \rho_{\rm c} \in C^\infty_\#(\Gamma).
\end{equation}

It is easy to prove that under assumptions
(\ref{eq:decay_V_local})-(\ref{eq:rhoc}),  
(\ref{eq:min_KS}) has a minimizer
$\Phi^0=(\phi_1^0,\cdots,\phi_\cN^0)^T$ with density $\rho^0=
\rho_{\Phi^0}$.  The regularity assumptions on
$V_{\rm local}$, on $e_{\rm xc}^{\rm LDA}$ and on the functions
$\chi_j$ allow to state that the minimizer $\Phi^0$ is in $[H^{3}_\#(\Gamma)]^\cN$, and even in $[H^{m+1/2-\epsilon}_\#(\Gamma)]^\cN$ for any $\epsilon > 0$, if at least one of the following conditions is satisfied: $e_{\rm xc}^{\rm LDA} \in C^{[m]}([0,+\infty))$ or $\rho_{\rm c}+\rho^0 > 0$ in $\R^3$. The former condition is not satisfied for usual LDA exchange-correlation functionals. On the other hand, it is satisfied for the Hartree (also called reduced Hartree-Fock) model, for which $e_{\rm xc}^{\rm LDA} = 0$. The latter condition seems to be satisfied in practice, but we were not able to establish it rigourously.

Let us introduce the Kohn-Sham Hamiltonian 

\begin{eqnarray*}
{\cal H}^{{\rm KS}}_{\rho^0} &=& -\frac 12 \Delta  + \left( V_{\rm local}+V^{\rm
    Coulomb}_{\rho^0}+ \frac{de_{\rm
    xc}^{\rm LDA}}{d\rho} (\rho_{\rm c}+\rho^0)\right)    +
V_{\rm nl}.
\\
&=& h + {\cal V}_{\rho^0}
\end{eqnarray*}
where 
\begin{equation}\label{def:h}
h = -\frac 12 \Delta  +  V_{\rm local}    +
V_{\rm nl},
\end{equation}
and
\begin{equation}\label{def:K}
{\cal V}_{\rho^0} =  V^{\rm
    Coulomb}_{\rho^0}+\frac{de_{\rm
    xc}^{\rm LDA}}{d\rho} (\rho_{\rm c}+\rho^0) .
    \end{equation}
    
We notice that ${E^{\rm KS}}'(\Phi^0) = 4 {\cal H}^{{\rm KS}}_{\rho^0}
\Phi^0$ in $(H^{-1}_\#(\Gamma))^\cN$ and
thus the Euler equations associated with the
minimization problem (\ref{eq:min_KS}) read 
$$
\forall 1\le i \le \cN, \quad 
{\cal H}^{{\rm KS}}_{\rho^0} \phi_i^0 = \sum_{j=1}^\cN \lambda_{ij}^0 \phi_j^0, 
$$
where the $\cN \times \cN$ matrix $\Lambda_\cN^0=(\lambda_{ij}^0)$,
which is the Lagrange multiplier of the matrix constraint $\int_\Gamma
\phi_i\phi_j= \delta_{ij}$, is symmetric. 

In fact, (\ref{eq:min_KS}) has an infinity of
minimizers since any unitary transform of the Kohn-Sham orbitals
$\Phi^0$ is also a minimizer of the Kohn-Sham energy. This is a
consequence of the following invariance property:
\begin{equation} \label{eq:invariance_property}
\forall \Phi \in {\cal M}, \quad \forall U \in \cU(\cN), \quad 
U\Phi \in {\cal M} \mbox{ and } E^{\rm KS}(U\Phi) = E^{\rm KS}(\Phi),
\end{equation}
where $\cU(\cN)$ is the group of the real unitary matrices:
$$
\cU(\cN) = \left\{U \in \R^{\cN\times\cN} \; | \; U^TU=1_\cN \right\},
$$
$1_\cN$ denoting the identity matrix of rank $\cN$.
This invariance
can be exploited to diagonalize the matrix of the Lagrange multipliers
of the orthonormality constraints (see e.g. \cite{DreizlerGross}),
yielding the existence of a minimizer (still denoted by~$\Phi^0$) with
same density $\rho^0$, such that
 \begin{equation}\label{supeq64}
{\cal H}^{{\rm KS}}_{\rho^0} \phi_i^0  = \epsilon_i^0 \phi_i^0,
 \end{equation}
for some $\epsilon_1^0 \le \epsilon_2^0 \le \cdots \le
\epsilon_\cN^0$. 

\begin{remark} The Kohn-Sham Hamiltonian ${\cal H}^{{\rm KS}}_{\rho^0}$
  is an unbounded self-adjoint operator on $L^2_\#(\Gamma)$, bounded
  below, with compact resolvent. Its spectrum therefore is purely
  discrete. More precisely, it is composed of a
  increasing sequence of eigenvalues going to infinity, each of these
  eigenvalues being of finite multiplicity. It is not known whether
  $\epsilon_1^0$, ..., $\epsilon_\cN^0$ are the lowest eigenvalues (counted with
  their multiplicities) of ${\cal H}^{{\rm KS}}_{\rho^0}$ ({\it Aufbau}
  principle). However, it seems to be most often (though not always) the
  case in practice. On the other hand, the {\it Aufbau} principle is always
  satisfied for the {\em extended} Kohn-Sham model, for which the first
  order optimality conditions read
$$
\left\{ \begin{array}{l}
{\cal H}^{{\rm KS}}_{\rho^0} \phi_i^0  = \epsilon_i^0 \phi_i^0 \\ \\
\dps \rho^0(x) = 2 \sum_{i=1}^{+\infty} n_i |\phi_i^0(x)|^2 \\
\dps \int_{\Gamma} \phi_i^0\phi_j^0 = \delta_{ij}, \quad 1 \le i,j < +\infty \\
n_i = 1 \mbox{ if } \epsilon_i^0 < \epsilon_{\rm F}, \quad 
n_i = 0 \mbox{ if } \epsilon_i^0 > \epsilon_{\rm F}, \quad 
\dps 0 \le n_i \le 1 \mbox{ if } \epsilon_i^0 = \epsilon_{\rm F}, \quad
\sum_{i=1}^{+\infty} n_i = \cN,
\end{array}
\right.
$$
where $\epsilon_{\rm F}$ is the Fermi level
(see~\cite{CDKLBM} for details). In this article, we focus on the
standard Kohn-Sham model with integer occupation numbers. We do not need to
assume that the Aufbau principle is satisfied, but our analysis requires
some coercivity assumption on the second order condition at $\Phi^0$
(see (\ref{eq:coerc})). 
\end{remark}

For each $\Phi = (\phi_1,\cdots,\phi_\cN)^T \in {\cal M}$, we denote by
$$
T_{\Phi} {\cal M} = \left\{ (\psi_1, \cdots, \psi_\cN)^T \in
  (H^1_\#(\Gamma))^\cN \; | \; \int_\Gamma \phi_i\psi_j+\psi_i\phi_j=0 \right\}
$$ 
the tangent space to ${\cal M}$ at $\Phi$, and by 
$$
\Phi^{\perp\!\!\!\perp} = \left\{ \Psi = (\psi_1,\cdots,\psi_\cN)^T \in
  (H^1_\#(\Gamma))^\cN \; | \;  \int_\Gamma \phi_i \psi_j = 0 \right\}.
$$
Let us recall (see e.g. Lemma 4 in \cite{MT}) that 
$$
T_{\Phi} {\cal M} = {\cal A} \Phi \oplus \Phi^{\perp\!\!\!\perp},
$$
where ${\cal A} = \left\{A \in \R^{\cN \times \cN} \; | \;
  A^T=-A\right\}$ is the space of the $\cN \times \cN$ antisymmetric
real matrices.

Since the problem we are considering is a minimization problem, the
second order condition further states 
$$
\forall W \in T_{\Phi^0} {\cal M}, \quad a_{\Phi^0}(W,W) \ge 0,
$$
where 
\begin{eqnarray}\label{eq:defaKS}
a_{\Phi^0}(\Psi,\Upsilon) &=& \frac{1}{4} {E^{\rm KS}}''(\Phi^0)(\Psi,\Upsilon) - \sum_{i=1}^N \epsilon_i^0
\int_\Gamma \psi_i\upsilon_i \\ &=&  \sum_{i=1}^\cN \langle ({\cal
  H}^{{\rm KS}}_{\rho^0}-\epsilon_i^0) \psi_i,
\upsilon_i\rangle_{H^{-1}_\#,H^1_\#} +  4 \sum_{i,j=1}^\cN
D_{\Gamma}(\phi_i^0\psi_i,\phi_j^0\upsilon_j) \nonumber
\\ & &+  4 \sum_{i,j=1}^\cN \int_\Gamma \frac{d^2e_{\rm
    xc}^{\rm LDA}}{d\rho^2}(\rho_{\rm c}+\rho^0)  \phi_i^0\psi_i \phi_j^0\upsilon_j.\label{eq:defaKS2}
\end{eqnarray}

It follows from the invariance property (\ref{eq:invariance_property}) that
$$
a_{\Phi^0}(\Psi,\Psi) = 0 \quad \mbox{for all } \Psi \in {\cal A} \Phi^0.  
$$
This leads us, as in \cite{MT}, to make the assumption that $a_{\Phi^0}$
is positive definite on
$\Phi^{0,\perp\!\!\!\perp}$ , so that, as in
Proposition~1 in \cite{MT}, $a_{\Phi^0}$ is coercive on
$\Phi^{0,\perp\!\!\!\perp}$ (for the $H^1_\#$ norm). Thus, in all what follows, we assume that
there exists a positive constant $c_{\Phi^0}$ such that 
\begin{equation} \label{eq:coerc}
\forall\Psi\in\Phi^{0,\perp\!\!\!\perp}, \quad a_{\Phi^0}(\Psi,\Psi) \ge
c_{\Phi^0} \|\Psi\|^2_{H^1_\#}. 
\end{equation}
In the linear framework ($J=0$ and $E_{\rm xc}^{\rm LDA}=0$ in
(\ref{eq:energy_KS})), this condition amounts to assuming that there is
a gap between the lowest $\cN^{\rm th}$ and $(\cN+1)^{\rm st}$ eigenvalues
of the linear self-adjoint operator $h=-\frac 12\Delta + V_{\rm local} +
V_{\rm nl}$.

The planewave approximation of (\ref{eq:min_KS}) 
reads
\begin{equation} \label{eq:min_KS_PW}
I^{\rm KS}_{N_c,N_g} = \inf \left\{ E^{\rm KS}_{N_g}(\Phi_{N_c}), \;
  \Phi_{N_c} \in  V_{N_c}^\cN \cap {\cal M} \right\}, 
\end{equation}
where 
\begin{eqnarray} 
E^{\rm KS}_{N_g}(\Phi) &=& \sum_{i=1}^\cN \int_\Gamma |\nabla
\phi_{i}|^2 +  
\int_\Gamma \rho_{\Phi} V_{\rm local} 
+ 2\sum_{i=1}^\cN \langle \phi_{i}|V_{\rm nl}|\phi_{i}\rangle \nonumber \\ & & 
+
J(\rho_{\Phi}) + \int_\Gamma \cI_{N_g}(e_{\rm xc}^{\rm
  LDA}(\rho_{\rm c}+\Pi_{2N_c}\rho_{\Phi})) . \label{eq:KS_energy}
\end{eqnarray} 
Here $N_c$ is a given positive integer, equal to $[\sqrt{2E_{\rm c}} \,
L/2\pi]$, 
$E_{\rm c}$ denoting the so-called cut-off energy, and $N_g \ge 4N_c+1$ is the
number of integration points per direction used to evaluate the
exchange-correlation contribution. The energy $E^{\rm KS}_{N_g}(\Phi)$ is defined for
each $\Phi \in {\cal M}$. For $\Phi \in V_{N_c}^\cN \cap {\cal
  M}$, $\Pi_{2N_c}\rho_{\Phi}=\rho_{\Phi}$, so that on this set, 
$E^{\rm KS}_{N_g}$ differs from $E^{\rm KS}$ only by the presence of the
Fourier interpolation operator $\cI_{N_g}$ in the exchange-correlation
functional. Let us mention that in practice, the terms involving the
local and nonlocal components of the pseudopotential are also computed by some
interpolation procedure. However, these terms are calculated using
spherical harmonics and a very fine one dimensional radial grid, so that
the resulting integration error is usually much smaller than the
interpolation error on the exchange-correlation term. Note that, in addition,
the pseudopotential gives rise to linear contributions that can be
computed very accurately once and for all (and not at each
iteration of the self-consistent algorithm). We postpone the analysis of
(\ref{eq:min_KS_PW}) to a forthcoming article~\cite{CCEM}, and focus
here on the variational approximation
\begin{equation} \label{eq:min_KS_Nc}
I^{\rm KS}_{N_c} = \inf \left\{ E^{\rm KS}(\Phi_{N_c}), \; 
\Phi_{N_c} \in  V_{N_c}^\cN \cap {\cal M} \right\} 
\end{equation}
of (\ref{eq:min_KS}). The unitary invariance of
the Kohn-Sham model must 
be taken into account in the derivation of optimal {\it a priori} error
estimates. One way to take this invariance into account is to work with
density matrices (see e.g. \cite{CDKLBM}). An alternative is to define
for each $\Phi \in {\cal M}$ the set
$$
{\cal M}^\Phi := \left\{ \Psi \in {\cal M} \; | \; 
\|\Psi-\Phi\|_{L^2_\#} =  \min_{U \in \cU(\cN)} \|U \Psi - \Phi\|_{L^2_\#}
\right\}, 
$$
and to use the fact that all the local minimizers of
(\ref{eq:min_KS_Nc}) are obtained by unitary
transforms from the local minimizers of
\begin{equation} \label{eq:min_KS_Nc_U}
I^{\rm KS}_{N_c} = \inf \left\{ E^{\rm KS}(\Phi_{N_c}), \; 
\Phi_{N_c} \in  V_{N_c}^\cN \cap {\cal M}^{\Phi^0} \right\}. 
\end{equation}

\medskip

The main result of this section is the following.

\medskip

\begin{theorem} \label{Th:KS-LDA} Assume that
  (\ref{eq:decay_V_local})-(\ref{eq:rhoc}) hold.  Let $\Phi^0$ be a
local minimizer of (\ref{eq:min_KS}) satisfying (\ref{eq:coerc}). Then there exists $r^0 > 0$ and $N_c^0$ such that for $N_c \ge N_c^0$, (\ref{eq:min_KS_Nc_U}) has a unique local minimizer $\Phi_{N_c}^0$ in the set 
$$
\left\{\Phi_{N_c} \in V_{N_c}^\cN \cap {\cal M}^{\Phi^0} \; | \;  \|\Phi_{N_c}-\Phi^0\|_{H^1_\#} \le r^0 \right\}.
$$ 
If we assume either that $e_{\rm xc}^{\rm LDA} \in C^{[m]}([0,+\infty))$ or that $\rho_{\rm c}+\rho^0 > 0$ on $\Gamma$, then we have the following estimates:
\begin{eqnarray}
\| \Phi_{N_c}^0-\Phi^0\|_{H^s_\#} & \le & C_{s,\epsilon}
N_c^{-(m-s+1/2-\epsilon)}, \label{KSHS} \\
|\epsilon_{i,N_c}^0-\epsilon_i^0| & \le & C_\epsilon N_c^{-(2m-1-\epsilon)}, \label{KSHS'} \\
\gamma  \| \Phi_{N_c}^0-\Phi^0\|_{H^1_\#}^2 \le I^{\rm KS}_{N_c}-I^{\rm KS} & \le & C \| \Phi_{N_c}^0-\Phi^0\|_{H^1_\#}^2,\label{KSHS''} 
\end{eqnarray} 
for all  $-m+3/2 < s < m+1/2$ and $\epsilon > 0$, and for some constants
$\gamma > 0$, $C_{s,\epsilon} \ge 0$, $C_\epsilon \ge 0$ and $C \ge 0$, 
where the $\epsilon_{i,N_c}^0$'s are the eigenvalues of the symmetric 
matrix $\Lambda_{N_c}^0$, the Lagrange multiplier of the matrix constraint
$\int_\Gamma \phi_{i,N_c}\phi_{j,N_c}=\delta_{ij}$.
\end{theorem}

\subsection{Some technical lemmas}

For $\Phi=
(\phi_1,\cdots,\phi_\cN)^T \in (H^1_\#(\Gamma))^\cN$ and $\Psi=
(\psi_1,\cdots,\psi_\cN)^T \in (H^1_\#(\Gamma))^\cN$, we denote by
$M_{\Phi,\Psi}$ the $\cN \times \cN$ matrix with entries 
$$
[M_{\Psi,\Phi}]_{ij} = \int_\Gamma \psi_i\phi_j. 
$$
The following lemma is useful for the analysis of
(\ref{eq:min_KS_Nc_U}). We recall that if $A$ and $B$ are symmetric $N \times N$ real matrices, the notation $A \le B$ means that $x^TAx \le x^TBx$ for all $x \in \R^N$.

\medskip

\begin{lemma} \label{Lem:unitary_invariance} $\;$
  \begin{enumerate}
  \item Let $\Phi \in {\cal M}$ and $\Psi \in {\cal M}$. If
    $M_{\Psi,\Phi}$ is invertible, then
    $U_{\Psi,\Phi}=M_{\Psi,\Phi}^T(M_{\Psi,\Phi}M_{\Psi,\Phi}^T)^{-1/2}$
    is the unique minimizer to the problem  
    $\min_{U \in \cU(\cN)} \|U \Psi - \Phi\|_{L^2_\#}$.
    .
  \item Let $\Phi \in {\cal M}$. Then
    $$
    {\cal M}^\Phi = \left\{ (1_\cN-M_{W,W})^{1/2} \Phi + W \; | \; W \in
      \Phi^{\perp\!\!\!\perp}, \; 0 \le M_{W,W} \le 1_\cN \right\}
    $$ 
    where $1_\cN$ denotes the identity matrix of rank $\cN$. 
  \item Let $\Phi = (\phi_1,\cdots,\phi_\cN)^T \in {\cal M}$. If $N_c
    \in \N$ is such that
    $$
    {\rm dim}({\rm span}(\Pi_{N_c}\phi_1,\cdots,\Pi_{N_c}\phi_\cN)) = \cN,
    $$ 
    then the unique minimizer of the problem $\min_{\Phi_{N_c} \in
      V_{N_c}^\cN \cap {\cal M}} \|\Phi_{N_c} - \Phi\|_{L^2_\#}$ is
    \begin{equation}\label{eq:2K1100}
      \pi_{N_c}^{\cal M}\Phi = \left( M_{\Pi_{N_c}\Phi,\Pi_{N_c}\Phi}
      \right)^{-1/2} \Pi_{N_c}\Phi.
    \end{equation}
In addition, $\pi_{N_c}^{\cal M}\Phi \in V_{N_c}^\cN \cap {\cal M}^{\Phi}$,
    \begin{equation}\label{eq:2K11} 
\| \pi_{N_c}^{\cal M}\Phi - \Phi\|_{L^2_\#} \le \sqrt 2 \| \Pi_{N_c}\Phi
- \Phi\|_{L^2_\#},
\end{equation}
and for all $N_c$ large enough, 
\begin{equation}
\| \pi_{N_c}^{\cal M}\Phi - \Phi\|_{H^1_\#} \le \|\Phi\|_{H^1_\#}
\| \Pi_{N_c}\Phi - \Phi\|_{L^2_\#}^2+
 \| \Pi_{N_c}\Phi - \Phi\|_{H^1_\#} .\label{eq:2K1101} 
\end{equation}

  \item Let $N_c$ such that $\mbox{dim}(V_{N_c}) \ge \cN$ and $\Phi_{N_c} \in V_{N_c}^\cN \cap {\cal M}$. Then 
    $$
    V_{N_c}^\cN \cap {\cal M}^{\Phi_{N_c}} = \left\{ (1_\cN-M_{W_{N_c},W_{N_c}})^{1/2}
      \Phi_{N_c} + W_{N_c} \; | \; W_{N_c} \in 
      V_{N_c}^\cN \cap \Phi_{N_c}^{\perp\!\!\!\perp}, \; 0 \le M_{W_{N_c},W_{N_c}}
      \le 1_\cN \right\}. 
    $$ 
  \end{enumerate}
\end{lemma}

\medskip

\begin{proof} In order to
  simplify the notation, we set $M=M_{\Psi,\Phi}$. For each $U \in
  \cU(\cN)$, 
$$
\|U\Psi-\Phi\|_{L^2_\#}^2 = 2\cN - 2 \tr(MU).
$$
Any critical point $U$ of the problem 
\begin{equation}
\max_{U \in \R^{\cN\times\cN} \, | \, U^TU=1_\cN} \tr(MU) \label{eq:minUUU}
\end{equation}
satisfies an Euler equation of the form $\Lambda U^T=M$ for some symmetric
matrix~$\Lambda$. Besides, $\tr(MU) = \tr(\Lambda)$ and $\Lambda^2 =
MM^T$. Any maximizer $U$ of (\ref{eq:minUUU}) therefore
satisfies $M=(MM^T)^{1/2}U^T$. Consequently, if $M$ is invertible, the maximizer of (\ref{eq:minUUU}) is unique and reads $U_{\Psi,\Phi} = M^T(MM^T)^{-1/2}$. It also follows from the
definition of the matrix $M$ that $\Psi = M \Phi + W$ with $W \in
\Phi^{\perp\!\!\!\perp}$. Thus,
$$
U_{\Psi,\Phi} \Psi = M^T(MM^T)^{-1/2}M \Phi + \widetilde W,
$$
with $\widetilde W = U_{\Psi,\Phi} W \in \Phi^{\perp\!\!\!\perp}$.

\medskip

Let us now prove the second statement. Each $\Psi \in
(H^1_\#(\Gamma))^\cN$ can be written as $\Psi = M \Phi + W$ for some
matrix $M \in \R^{\cN \times \cN}$ and some $W  \in
\Phi^{\perp\!\!\!\perp}$. A simple calculation leads to
$$
\int_\Gamma \psi_i\psi_j = [MM^T]_{ij} + [M_{W,W}]_{ij}.
$$
Hence $\Psi = M\Phi + W \in {\cal M}$ if and only if $MM^T+M_{W,W} =
1_\cN$. In addition, $\Psi \in {\cal M}^\Phi$ if and only if $\Psi \in {\cal
  M}$ and $U_{\Psi,\Phi} = M^T(MM^T)^{-1/2} = 1_\cN$, that is to say if and
only if $M$ is symmetric, $0 \le M_{W,W} \le 1_\cN$ and $M=(1_\cN-M_{W,W})^{1/2}$. 

\medskip

Let $(\chi_\mu)_{1 \le \mu \le {\rm dim}(V_{N_c})}$ be an orthonormal
basis of $V_{N_c}$ (for the $L^2_\#$ inner product) and let $\widetilde
C \in \R^{{\rm dim}(V_{N_c}) \times \cN}$ be the matrix with entries
$$
\widetilde C_{\mu,i} = \int_\Gamma \phi_i\chi_\mu.
$$
Note that
\begin{equation} \label{eq:dec_Phi}
\Pi_{N_c} \phi_i = 
\sum_{\mu=1}^{{\rm dim}(V_{N_c})} \widetilde C_{\mu,i} \chi_\mu,
\end{equation}
For all $\Phi_{N_c}
=(\phi_{N_c,1}, \cdots , \phi_{N_c,\cN})^T \in V_{N_c}^\cN \cap {\cal M}$,
each $\phi_{N_c,i}$ can be expanded as 
\begin{equation} \label{eq:dec_Phi_2}
\phi_{N_c,i} = \sum_{\mu=1}^{{\rm dim}(V_{N_c})} C_{\mu i} \chi_\mu,
\end{equation}
where the matrix $C=[C_{\mu i}]  \in \R^{{\rm dim}(V_{N_c})} 
  \times \cN$ satisfies the constraint $C^TC=1_\cN$. The expansions
  (\ref{eq:dec_Phi}) and (\ref{eq:dec_Phi_2}) can be recast into the more
  compact forms
$$ 
\Pi_{N_c}\Phi = \widetilde C^T {\cal X} \qquad \mbox{and} \qquad 
\Phi_{N_c} = C^T{\cal X},
$$
where we have denoted
by ${\cal X} = (\chi_1,\cdots,\chi_{\mbox{dim}(V_{N_c})})^T$.
A simple calculation then leads to
\begin{equation} \label{eq:2K10}
\|\Phi_{N_c} - \Phi\|_{L^2_\#}^2 = 2N - 2 \tr(\widetilde C^T C).
\end{equation}
Reasoning as above, we obtain that the unique solution to the problem
$$
\max_{C \in \R^{{\rm dim}(V_{N_c}) \times \cN} \, | \, C^TC=1_\cN}
\tr(\widetilde C^TC) 
$$
is $C = \widetilde C(\widetilde C^T\widetilde C)^{-1/2}$. Note that the rank of the matrix $\widetilde C$ is $\cN$
provided that ${\rm dim}(V_{N_c})$ is large enough so that the matrix
$\widetilde C^T\widetilde C$ is invertible provided that  ${\rm
  dim}(V_{N_c})$ is large 
enough. As a consequence, the unique solution to the problem
$\min_{\Phi_{N_c} \in V_{N_c}^\cN \cap {\cal M}} \|\Phi_{N_c} -
\Phi\|_{L^2_\#}$ is $\pi_{N_c}^{\cal M}\Phi = (\widetilde C^T\widetilde
C)^{-1/2} \widetilde C^T{\cal X} = (\widetilde C^T\widetilde
C)^{-1/2} \Pi_{N_c}\Phi$. It is then easy to check that 
$\widetilde C^T\widetilde C = M_{\Pi_{N_c}\Phi,\Pi_{N_c}\Phi}$. Hence
(\ref{eq:2K1100}). Then, for all $U \in \R^{\cN\times\cN}$ such that
$U^TU=1_\cN$, 
$$
\|U \pi_{N_c}^{\cal M}\Phi - \Phi\|_{L^2_\#}^2 
= 2 (1 - \tr(UM_{\Pi_{N_c}\Phi,\Pi_{N_c}\Phi}^{1/2})),
$$
and the same argument as above leads to the result that this
quantity is minimized for
$U=M_{\Pi_{N_c}\Phi,\Pi_{N_c}\Phi}^{1/2}(M_{\Pi_{N_c}\Phi,\Pi_{N_c}\Phi}^{1/2}M_{\Pi_{N_c}\Phi,\Pi_{N_c}\Phi}^{1/2})^{-1/2}
= 1_\cN$. Therefore, $\pi_{N_c}^{\cal M}\Phi \in {\cal M}^\Phi$.

We also infer from (\ref{eq:2K10}) that
$$
\|\pi_{N_c}^{\cal M}\Phi - \Phi\|_{L^2_\#}^2 = 2N - 2
\tr\left((\widetilde C^T\widetilde C)^{1/2}\right) 
= 2 \tr\left(1_\cN-(\widetilde C^T\widetilde C)^{1/2}\right).
$$
Besides, an easy calculation leads to
$$
\|\Pi_{N_c}\Phi - \Phi\|_{L^2_\#}^2 = \tr\left(1_\cN-\widetilde
  C^T\widetilde C\right). 
$$
Using the fact that 
$$
0 \le \left(1_\cN-(\widetilde C^T\widetilde C)^{1/2}\right) \le
\left(1_\cN-(\widetilde C^T\widetilde C)^{1/2}\right)\left(1_\cN+(\widetilde C^T\widetilde C)^{1/2}\right) = 1_\cN -
\widetilde C^T\widetilde C, 
$$
we obtain
$$
\|\pi_{N_c}^{\cal M}\Phi - \Phi\|_{L^2_\#}^2 = 2
\tr\left(1_\cN-(\widetilde C^T\widetilde C)^{1/2}\right) 
\le 2  \tr\left(1_\cN-\widetilde C^T\widetilde C\right) = 2
\|\Pi_{N_c}\Phi - \Phi\|_{L^2_\#}^2. 
$$
Hence (\ref{eq:2K11}). We also have
\begin{eqnarray*}
\|\pi_{N_c}^{\cal M}\Phi - \Phi\|_{H^1_\#} &\le& 
\|\pi_{N_c}^{\cal M}\Phi - \Pi_{N_c}\Phi\|_{H^1_\#} + 
\|\Pi_{N_c}\Phi - \Phi\|_{H^1_\#} \\
&=& \|((M_{\Pi_{N_c}\Phi,\Pi_{N_c}\Phi})^{-1/2} - 1_\cN) \Pi_{N_c}\Phi\|_{H^1_\#} + 
\|\Pi_{N_c}\Phi - \Phi\|_{H^1_\#} \\
& \le & \|(M_{\Pi_{N_c}\Phi,\Pi_{N_c}\Phi})^{-1/2} - 1_\cN\|_{\rm F}
\|\Pi_{N_c}\Phi\|_{H^1_\#} + 
\|\Pi_{N_c}\Phi - \Phi\|_{H^1_\#} \\
& \le & \|(M_{\Pi_{N_c}\Phi,\Pi_{N_c}\Phi})^{-1/2} - 1_\cN\|_{\rm F}
\|\Phi\|_{H^1_\#} + 
\|\Pi_{N_c}\Phi - \Phi\|_{H^1_\#},
\end{eqnarray*}
where $\|\cdot\|_{\rm F}$ denotes the Frobenius norm.
We then notice that 
$$
M_{\Pi_{N_c}\Phi,\Pi_{N_c}\Phi} = 1_\cN -
M_{\Pi_{N_c}\Phi-\Phi,\Pi_{N_c}\Phi-\Phi}.
$$
Consequently, for $N_c$ large enough,
$$
\|(M_{\Pi_{N_c}\Phi,\Pi_{N_c}\Phi})^{-1/2}- 1_\cN\|_{\rm F} \le
\|M_{\Pi_{N_c}\Phi-\Phi,\Pi_{N_c}\Phi-\Phi}\|_{\rm F} \le  
\|\Pi_{N_c}\Phi-\Phi\|_{L^2_\#}^2.
$$
Therefore (\ref{eq:2K1101}) is proved.

\medskip

Lastly, the fourth assertion easily follows from the second one.
\end{proof}

\medskip

\begin{lemma} \label{lem:properties_of_S} Let 
$$
K = \left\{ W \in (L^2_\#(\Gamma)^\cN \; | \; 0 \le M_{W,W} \le 1_\cN \right\},
$$ 
and $\cS \, : \, K \rightarrow \R^{\cN \times \cN}_{\rm S}$ (the space of the symmetric $\cN \times \cN$ real matrices) defined by
$$
\cS(W) = (1_\cN-M_{W,W})^{1/2}-1_\cN.
$$
The function $\cS$ is continuous on $K$ and differentiable on
the interior $\dps\mathop{{K}}^\circ$ of $K$. In addition,  
\begin{equation} \label{eq:bound_SWW}
\forall W \in  K, \quad \|\cS(W) \|_F \le \|W\|_{L^2_\#}^2,
\end{equation}
and for all $(W_1,W_2,Z) \in K \times K \times (L^2_\#(\Gamma))^\cN$ such that $\|W_1\|_{L^2_\#}\le \frac 12$ and $\|W_2\|_{L^2_\#}\le \frac 12$, 
\begin{eqnarray}\label{eq:lip_S}
\|\cS(W_1)-\cS(W_2)\|_{\rm F} &\le & 2 (\|W_1\|_{L^2_\#}+\|W_2\|_{L^2_\#}) 
\|W_1-W_2\|_{L^2_\#}, \\
\|(\cS'(W_1)-\cS'(W_2))\cdot Z\|_{\rm F} &\le&  3
\|W_1-W_2\|_{L^2_\#} \|Z\|_{L^2_\#},  \label{eq:deriv_S} \\
\|(\cS''(W_1)(Z,Z)\|_{\rm F} &\le&  3 \|Z\|_{L^2_\#}^2 . \label{eq:deriv2_S}
\end{eqnarray}
\end{lemma}

\medskip

\begin{proof}
Diagonalizing $M_{W,W}$ and using the
properties of the function $t \mapsto (1-t)^{1/2}-1$, we see that $\cS$ is
continuous on $K$ and differentiable on $\dps \mathop{K}^\circ$, and that 
$$
\|\cS(W) \|_{\rm F} \le \|M_{W,W}\|_{\rm F} \le \|W\|_{L^2_\#}^2.
$$
Hence (\ref{eq:bound_SWW}). As 
$$
\cS(W) + \frac 12 \cS(W)^2 = - \frac 12 M_{W,W},
$$
we have for all $W \in \dps \mathop{K}^\circ$,
\begin{eqnarray*}
\!\!\!\!\!\!\!\!\!\!
&& \cS'(W) \cdot Z + \frac 12 \left[ 
\cS(W) (\cS'(W)\cdot Z)) + (\cS'(W) \cdot Z))
\cS(W) \right] \\ \!\!\!\!\!\!\!\!\!\!
&& \qquad = 
-\frac 12 \left[ M_{W,Z}+M_{Z,W} \right]. 
\end{eqnarray*}
Denoting by $A=\cS'(W)\cdot Z$, we deduce from the above equality that
$$
\|A\|_{\rm F}^2 + \tr(A^2\cS(W)) \le \|A\|_{\rm F} \|M_{W,Z}\|_{\rm F} \le \|A\|_{\rm F}  \|W\|_{L^2_\#} \|Z\|_{L^2_\#} .
$$
As $|\tr(A^2\cS(W))| \le \|A\|_{\rm F}^2  \|\cS(W)\|_{2} \le   \|A\|_{\rm F}^2 \|\cS(W)\|_{\rm F} \le  \|A\|_{\rm F}^2 \|W\|_{L^2_\#}^2$, we finally obtain the inequality  
\begin{equation}\label{eq:bound_A}
\|A\|_{\rm F} (1 - \|W\|_{L^2_\#}^2) \le \|W\|_{L^2_\#} \|Z\|_{L^2_\#},
\end{equation}
which straightforwardly leads to (\ref{eq:lip_S}) under the conditions $\|W_1\|_{L^2_\#} \le \frac 12$ and  $\|W_2\|_{L^2_\#} \le \frac 12$.  
Lastly, 
\begin{eqnarray*}
\!\!\!\!\!\!\!\!\!\!
&& (\cS'(W_2)-\cS'(W_1)) \cdot Z + \frac 12 \left[ 
\cS(W_2) ((\cS'(W_2)-\cS'(W_1)) \cdot Z) + ((\cS'(W_2)-\cS'(W_1)) \cdot Z)
\cS(W_2) \right] \\ \!\!\!\!\!\!\!\!\!\!
&& \qquad + \frac 12 \left[ (\cS'(W_1)\cdot Z) (\cS(W_2)-\cS(W_1))
+   (\cS(W_2)-\cS(W_1)) (\cS'(W_1)\cdot Z) \right] = 
-\frac 12 \left[ M_{W_2-W_1,Z}+M_{Z,W_2-W_1} \right], 
\end{eqnarray*}
so that still under the conditions $\|W_1\|_{L^2_\#} \le \frac 12$ and  $\|W_2\|_{L^2_\#} \le \frac 12$,
$$
\frac 34 \|(\cS'(W_2)-\cS'(W_1)) \cdot Z\|_{\rm F} \le 
\frac{17}{9} \|W_2-W_1\|_{L^2_\#} \|Z\|_{L^2_\#}.
$$
Hence (\ref{eq:deriv_S}). Lastly, taking $W_2 = W_1+tZ$ in (\ref{eq:deriv_S}) and letting $t$ go to zero, we obtain (\ref{eq:deriv2_S}).
\end{proof}

\medskip

\begin{lemma} \label{lem:coercdisc}
Let $\Phi^0$ be a local minimizer of (\ref{eq:min_KS}) satisfying
(\ref{eq:coerc}). Then $a_{\Phi^0}$ defines a continuous bilinear form
on $(H^1_\#(\Gamma))^\cN \times (H^1_\#(\Gamma))^\cN$,
and there exists $N_c^*$ such that for all $N_c \ge
N_c^*$, 
\begin{equation}\label{eq:PProjec}
\|\pi^{\cal M}_{N_c}\Phi^0-\Phi^0\|_{H^1_\#} \le 1,
\end{equation}
\begin{equation}\label{eq:elip}
a_{\Phi^0}(\pi^{\cal M}_{N_c}\Phi^0-\Phi^0,\pi^{\cal
  M}_{N_c}\Phi^0-\Phi^0) \ge \frac{c_{\Phi^0}}{2} \|\pi^{\cal
  M}_{N_c}\Phi^0-\Phi^0\|^2_{H^1_\#}, 
\end{equation}
\begin{equation} \label{eq:coercdisc3}
\forall W\in [\pi_{N_c}^{\cal M}\Phi^{0}]^{\perp\!\!\!\perp}, 
\quad a_{\Phi^0}(W,W) \ge \frac{c_{\Phi^0}}{2} \|W\|^2_{H^1_\#}.
\end{equation}
\end{lemma}

\medskip

In the sequel, we denote by $C_{\Phi^0}$ the continuity constant of
$a_{\Phi^0}$, i.e.
\begin{equation}  \label{eq:continuity_aPhi0}
\forall (\Psi,\Psi') \in 
((H^1_\#(\Gamma))^\cN)^2, \quad |a_{\Phi^0}(\Psi,\Psi')| \le C_{\Phi^0}
\|\Psi\|_{H^1_\#} \|\Psi'\|_{H^1_\#}. 
\end{equation}

\medskip

\begin{proof}
Estimate (\ref{eq:PProjec}) immediately results from the closeness 
of $\pi^{\cal M}_{N_c}  \Phi^0$ to $\Phi^0$. Using the fact that
$\pi_{N_c}^{\cal M}\Phi^0 \in  {\cal M}^{\Phi^0}$ (see Lemma \ref{Lem:unitary_invariance}, point 3), we get
\begin{equation} \label{eq:decomp000}
\pi_{N_c}^{\cal M}\Phi^0- \Phi^0 = \cS(W) \Phi^0 + W
\end{equation}
with $W\in [\Phi^{0}]^{\perp\!\!\!\perp}$, from which we derive, using (\ref{eq:bound_SWW}),  that
\begin{eqnarray*} \!\!\!\!\!\!\!\!\!\!
a_{\Phi^0}(\pi^{\cal M}_{N_c}\Phi^0-\Phi^0,\pi^{\cal M}_{N_c}\Phi^0-\Phi^0) &=& a_{\Phi^0}(W,W) + 2 a_{\Phi^0}(W,\cS(W) \Phi^0) +a_{\Phi^0}(\cS(W) \Phi^0, \cS(W)\Phi^0) \\
& \ge & c_{\Phi^0} \|W\|_{H^1_\#}^2 - 2 C_{\Phi^0} \|W\|_{H^1_\#} \|\Phi^0\|_{H^1_\#}
\|W\|_{L^2_\#}^2 - C_{\Phi^0} \|W\|_{L^2_\#}^4 \|\Phi^0\|_{H^1_\#}^2 \\
& \ge & \left( c_{\Phi^0} - 2 C_{\Phi^0} \|W\|_{L^2_\#} \|\Phi^0\|_{H^1_\#}
- C_{\Phi^0} \|W\|_{L^2_\#}^2 \|\Phi^0\|_{H^1_\#}^2\right)  \|W\|_{H^1_\#}^2.
\end{eqnarray*}
As by (\ref{eq:2K11}), $\|\pi_{N_c}^{\cal M}\Phi^0- \Phi^0\|_{L^2_\#}$ goes to zero when $N_c$ goes to infinity, so does $\|W\|_{L^2_\#}$. Using again (\ref{eq:bound_SWW}), we deduce from (\ref{eq:decomp000}) that $\dps \|W\|_{H^1_\#} \mathop{\sim}_{N_c \to \infty} \|\pi^{\cal M}_{N_c}\Phi^0-\Phi^0\|_{H^1_\#}$. Hence (\ref{eq:elip}).

\medskip

\noindent
Finally, for each $W\in [\pi_{N_c}^{\cal M}\Phi^{0}]^{\perp\!\!\!\perp}$,
$W^* = W - M_{W,\Phi^0} \Phi^0$ belongs to $[\Phi^{0}]^{\perp\!\!\!\perp}$. Remarking that $M_{W,\Phi^0} = M_{W,\Phi^0-\pi_{N_c}^{\cal M}\Phi^{0}}$, we derive 
$$
\|M_{W,\Phi^0}\|_{\rm F} \le \|M_{W,\Phi^0-\pi_{N_c}^{\cal M}\Phi^{0}}\|_{\rm F} \le \varepsilon (N_c) \|W\|_{L^2_\#}
$$ 
where $ \varepsilon (N_c) =  \| \Phi^0-\pi_{N_c}^{\cal M}\Phi^{0}\|_{L^2_\#} \rightarrow 0$ when $N_c$ goes to infinity. Therefore,
$$
\|W-W^*\|_{H^1_\#} \le \varepsilon (N_c)  \|\Phi^0\|_{H^1_\#} \|W\|_{H^1_\#}.
$$
As 
$$a_{\Phi^0}(W,W) = a_{\Phi^0}(W^*,W^*) + 2 a_{\Phi^0}(W,W-W^*)+ a_{\Phi^0}(W-W^*,W-W^*),$$
we obtain
\begin{eqnarray*}
a_{\Phi^0}(W,W)& \ge &   c_{\Phi^0}  \|W^*\|^2_{H^1_\#} - 2 C_{\Phi^0}  \|W\|_{H^1_\#}  \|W-W^*\|_{H^1_\#} - C_{\Phi^0}    \|W-W^*\|^2_{H^1_\#} \\
& \ge &  \left( c_{\Phi^0} - 2  C_{\Phi^0}  \varepsilon (N_c) \|\Phi^0\|_{H^1_\#}  - C_{\Phi^0}  \varepsilon (N_c) ^2  \|\Phi^0\|_{H^1_\#} \right) \|W\|^2_{H^1_\#}.
\end{eqnarray*}
Hence (\ref{eq:coercdisc3}) for $N_c$ large enough.
\end{proof}

\medskip

\begin{lemma}\label{lem:estimR}
There exists $C \ge 0$ such that 
\begin{enumerate}
\item for all $(\Upsilon_1,\Upsilon_2,\Upsilon_3) \in \left((H^1_\#(\Gamma))^\cN\right)^3$,
$$
\left| \left({E^{\rm KS}}''(\Phi^0+\Upsilon_1)-{E^{\rm KS}}''(\Phi^0) \right) (\Upsilon_2,\Upsilon_3) \right| \le  C \left( \|\Upsilon_1\|_{H^1_\#}^{\alpha}+\|\Upsilon_1\|_{H^1_\#}^2 \right) \, \|\Upsilon_2\|_{H^1_\#}\, \|\Upsilon_3\|_{H^1_\#}.
$$
\item for all $\Upsilon_1 \in (H^1_\#(\Gamma)\cap L^\infty_\#(\Gamma))^\cN$ and $(\Upsilon_2,\Upsilon_3) \in \left((H^1_\#(\Gamma))^\cN\right)^2$,
$$
\left| \left({E^{\rm KS}}''(\Phi^0+\Upsilon_1)-{E^{\rm KS}}''(\Phi^0) \right) (\Upsilon_2,\Upsilon_3) \right| \le  C \left( 1 + \|\Upsilon_1\|_{L^\infty}^{2-\alpha} \right)\|\Upsilon_1\|_{L^2_\#}^{\alpha} \, \|\Upsilon_2\|_{H^1_\#} \, \|\Upsilon_3\|_{H^1_\#}.
$$
\end{enumerate}
\end{lemma}

\medskip

\begin{proof}
Let us denote by
\begin{eqnarray*}
r_{\Phi^0}(\Upsilon_1,\Upsilon_2,\Upsilon_3) &=& \left({E^{\rm KS}}''(\Phi^0+\Upsilon_1)-{E^{\rm KS}}''(\Phi^0) \right) (\Upsilon_2,\Upsilon_3)
\end{eqnarray*}
Splitting $r_{\Phi^0}(\Upsilon_1,\Upsilon_2,\Upsilon_3)$ in its Coulomb and exchange-correlation contributions, we obtain
$$
r_{\Phi^0}(\Upsilon_1,\Upsilon_2,\Upsilon_3) = r^{\rm Coulomb}_{\Phi^0}(\Upsilon_1,\Upsilon_2,\Upsilon_3) + r_{\Phi^0}^{\rm xc}(\Upsilon_1,\Upsilon_2,\Upsilon_3),
$$
with
\begin{eqnarray*}
r^{\rm Coulomb}_{\Phi^0}(\Upsilon_1,\Upsilon_2,\Upsilon_3)
 &=&  16   \sum_{i,j=1}^{\cN} \left( 
D_\Gamma(\phi_i^0 \upsilon_{1,i},\upsilon_{2,j}\upsilon_{3,j})   + D_\Gamma(\phi_i^0\upsilon_{2,i},\upsilon_{1,j}\upsilon_{3,j}) + D_\Gamma(\phi_i^0\upsilon_{3,i},\upsilon_{1,j}\upsilon_{2,j})  \right) \\
&& + 16   \sum_{i,j=1}^{\cN} D_\Gamma(\upsilon_{1,i}\upsilon_{2,i},\upsilon_{1,j}\upsilon_{3,j}) + 8 \sum_{i,j=1}^{\cN} D_\Gamma(\upsilon_{1,i}^2,\upsilon_{2,j}\upsilon_{3,i}), 
\end{eqnarray*}
and
\begin{eqnarray*}  
r^{\rm xc}_{\Phi^0}(\Upsilon_1,\Upsilon_2,\Upsilon_3)
= r^{\rm xc,1}_{\Phi^0}(\Upsilon_1,\Upsilon_2,\Upsilon_3)+ 
r^{\rm xc,2}_{\Phi^0}(\Upsilon_1,\Upsilon_2,\Upsilon_3),
\end{eqnarray*}
where
\begin{eqnarray*}
 r^{\rm xc,1}_{\Phi^0}(\Upsilon_1,\Upsilon_2,\Upsilon_3) &=& 4 \int_\Gamma  \left( \frac{d e_{\rm xc}^{\rm LDA}}{d\rho}
(\rho_{\rm c}+ \rho_{\Phi^0+\Upsilon_1}) - \frac{d e_{\rm xc}^{\rm LDA}}{d\rho}
(\rho_{\rm c}+ \rho_{\Phi^0} )\right)  \left( \sum_{i=1}^\cN \upsilon_{2,i} \upsilon_{3,i} \right), \\
 r^{\rm xc,2}_{\Phi^0}(\Upsilon_1,\Upsilon_2) &=&  16 \int_\Gamma \bigg[  \frac{d^2e_{\rm xc}^{\rm LDA}}{d\rho^2}
(\rho_{\rm c}+ \rho_{\Phi^0+\Upsilon_1}) \left( \sum_{i=1}^\cN (\phi_i^0+\upsilon_{1,i}) \upsilon_{2,i} \right) \left( \sum_{i=1}^\cN (\phi_i^0+\upsilon_{1,i}) \upsilon_{3,i} \right)  \\ && \qquad \quad - 
\frac{d^2e_{\rm xc}^{\rm LDA}}{d\rho^2}
(\rho_{\rm c}+ \rho_{\Phi^0}) \left( \sum_{i=1}^\cN \phi_i^0 \upsilon_{2,i} \right)  \left( \sum_{i=1}^\cN \phi_i^0 \upsilon_{3,i} \right)\bigg].
\end{eqnarray*}
Using (\ref{eq:estim_D3}), we obtain that there exists a constant $C \ge
0$, such that for all $(\Upsilon_1,\Upsilon_2,\Upsilon_3) \in \left((H^1_\#(\Gamma))^\cN\right)^3$,
\begin{equation} \label{eq:estim_RC}
|r^{\rm Coulomb}_{\Phi^0}(\Upsilon_1,\Upsilon_2,\Upsilon_3)|
\le C \left( \|\Upsilon_1\|_{L^2_\#}+\|\Upsilon_1\|_{L^2_\#}^2\right) \, \|\Upsilon_2\|_{H^1_\#}  \, \|\Upsilon_3\|_{H^1_\#}.
\end{equation}

Using (\ref{eq:hyp_epsilon_xc_2}), we get
\begin{eqnarray*}
\left|  \frac{d e_{\rm xc}^{\rm LDA}}{d\rho}
(\rho_{\rm c}+ \rho_{\Phi^0+\Upsilon_1}) - \frac{d e_{\rm xc}^{\rm LDA}}{d\rho}
(\rho_{\rm c}+ \rho_{\Phi^0}) \right| & \le & C \left( |\rho_{\Phi^0+\Upsilon_1}-\rho_{\Phi^0}| +  \alpha^{-1}  |\rho_{\Phi^0+\Upsilon_1}-\rho_{\Phi^0}|^\alpha \right) \\ &\le &
  C \bigg[ \rho_{\Upsilon_1}^{\alpha/2} + \rho_{\Upsilon_1} \bigg],
\end{eqnarray*}
from which we infer
\begin{equation} \label{eq:estim_rxc1}
| r^{\rm xc,1}_{\Phi^0}(\Upsilon_1,\Upsilon_2,\Upsilon_3)| \le C \int_\Gamma 
\left( \rho_{\Upsilon_1}^{\alpha/2} + \rho_{\Upsilon_1} \right) \, \rho_{\Upsilon_2}^{1/2}  \, \rho_{\Upsilon_3}^{1/2}  .
\end{equation}
Introducing the function 
$$
\Phi(t) = \Phi^0 + t \Upsilon_1,
$$
we can rewrite $ r^{\rm xc,2}_{\Phi^0}(\Upsilon_1,\Upsilon_2,\Upsilon_3)$ as
\begin{eqnarray*} \!\!\!\!\!\!\!\!\!\!\!\!\!\!\!\!\!\!\!\!\!\!\!\!\!\!\!
 r^{\rm xc,2}_{\Phi^0}(\Upsilon_1,\Upsilon_2,\Upsilon_3) & = & 16 \int_\Gamma 
 \bigg[  \frac{d^2e_{\rm xc}^{\rm LDA}}{d\rho^2}
(\rho_{\rm c}+ \rho_{\Phi(1)}) \left( \sum_{i=1}^\cN \phi_i(1) \upsilon_{2,i} \right) \left( \sum_{i=1}^\cN \phi_i(1) \upsilon_{3,i} \right)  \\ && \qquad \quad - 
\frac{d^2e_{\rm xc}^{\rm LDA}}{d\rho^2}
(\rho_{\rm c}+ \rho_{\Phi(0)}) \left( \sum_{i=1}^\cN \phi_i(0) \upsilon_{2,i} \right)  \left( \sum_{i=1}^\cN \phi_i(0) \upsilon_{3,i} \right) \bigg] \\
& = &  16 \int_\Gamma \int_0^1
 \bigg[  \frac{d^2e_{\rm xc}^{\rm LDA}}{d\rho^2}
(\rho_{\rm c}+ \rho_{\Phi(t)}) \left( \sum_{i=1}^\cN \phi_{i}(t) \upsilon_{2,i} \right)  \left( \sum_{i=1}^\cN \upsilon_{1,i} \upsilon_{3,i} \right) \\ && \qquad \qquad + \frac{d^2e_{\rm xc}^{\rm LDA}}{d\rho^2}
(\rho_{\rm c}+ \rho_{\Phi(t)})  \left( \sum_{i=1}^\cN \upsilon_{1,i} \upsilon_{2,i} \right) \left( \sum_{i=1}^\cN \phi_{i}(t) \upsilon_{3,i} \right)  \\ && \qquad \qquad  + 2
\frac{d^3e_{\rm xc}^{\rm LDA}}{d\rho^3}
(\rho_{\rm c}+ \rho_{\Phi(t)})\left( \sum_{i=1}^\cN \phi_i(t) \upsilon_{1,i} \right)   \left( \sum_{i=1}^\cN \phi_i(t) \upsilon_{2,i} \right)
 \left( \sum_{i=1}^\cN \phi_i(t) \upsilon_{3,i} \right)  \bigg] \, dt.
\end{eqnarray*}
Thus, using again (\ref{eq:hyp_epsilon_xc_2}), we obtain
\begin{eqnarray*}
|r^{\rm xc,2}_{\Phi^0}(\Upsilon_1,\Upsilon_2,\Upsilon_3)| & \le  & C \int_\Gamma \int_0^1
  (1+(\rho_{\rm c}+\rho_{\Phi(t)})^{\alpha-1}) \rho_{\Phi(t)}^{1/2} \rho_{\Upsilon_1}^{1/2} \rho_{\Upsilon_2}^{1/2} \rho_{\Upsilon_3}^{1/2}   \, dt \\ & \le & C \int_\Gamma \int_0^1
  (1+\rho_{\Phi(t)}^{\alpha-1}) \rho_{\Phi(t)}^{1/2} \rho_{\Upsilon_1}^{1/2} \rho_{\Upsilon_2}^{1/2}  \rho_{\Upsilon_3}^{1/2}  \, dt.
\end{eqnarray*}
Now,
\begin{eqnarray*}
&& \!\!\!\!\!\!\!\!\!\!\!\! 
\int_0^1 \rho_{\Phi(t)}^{\alpha-1/2} dt  =  2^{\alpha-1/2}\int_0^1   
\left( \sum_{i=1}^\cN {\phi_i^0}^2 + 2 t \sum_{i=1}^\cN \phi_i^0 \upsilon_{1,i}+
  t^2 \sum_{i=1}^\cN \upsilon_{1,i}^2 \right)^{\alpha-1/2} \, dt \\
 &= &  2^{\alpha-1/2} \int_0^1  \left( \sum_{i=1}^\cN {\phi_i^0}^2 - 
\frac{\left( \sum_{i=1}^\cN
    \phi_i^0\upsilon_{1,i} \right)^2 }{\sum_{i=1}^\cN \upsilon_{1,i}^2} + 
\left( t +  \frac{\sum_{i=1}^\cN
    \phi_i^0\upsilon_{1,i} }{\sum_{i=1}^\cN
    \upsilon_{1,i}^2} \right)^2
\left( \sum_{i=1}^\cN \upsilon_{1,i}^2 \right) \right)^{\alpha-1/2}  \,
dt \\
 & \le &  2^{\alpha-1/2} \int_0^1  \left| t +  \frac{\sum_{i=1}^\cN
    \phi_i^0\upsilon_{1,i} }{\sum_{i=1}^\cN
    \upsilon_{1,i}^2} \right|^{2\alpha-1}
\left( \sum_{i=1}^\cN \upsilon_{1,i}^2 \right)^{\alpha-1/2}  \,
dt \le \frac{1}{\alpha 2^{\alpha+1/2}} \rho_{\Upsilon_{1}}^{\alpha-1/2}.  
\end{eqnarray*}
Therefore,
\begin{eqnarray}
|r^{\rm xc,2}_{\Phi^0}(\Upsilon_1,\Upsilon_2,\Upsilon_3)| & \le  &
C \int_\Gamma \left(\rho_{\Upsilon_1}^{\min(\alpha,1/2)}+\rho_{\Upsilon_1}\right) \, \rho_{\Upsilon_2}^{1/2} \, \rho_{\Upsilon_3}^{1/2} . \label{eq:estim_rxc2}
\end{eqnarray}
Gathering (\ref{eq:estim_RC}), (\ref{eq:estim_rxc1}) and (\ref{eq:estim_rxc2}), we obtain the desired estimates.
\end{proof}

\medskip

\begin{lemma} \label{lem:estim_2}
Let $\Phi^0$ be a local minimizer of (\ref{eq:min_KS}) satisfying
(\ref{eq:coerc}). Then there exists $C \ge 0$ such that
for all $\Psi\in {\cal M}$, 
\begin{equation} \label{eq:coercdisc}
E^{\rm KS}(\Psi) = E^{\rm KS}(\Phi^0) + 2
a_{\Phi^0}(\Psi-\Phi^0,\Psi-\Phi^0) + R(\Psi-\Phi^0),
\end{equation}
with
\begin{equation} \label{eq:coercdisc_2}
|R(\Psi-\Phi^0)| \le C \left(\|\Psi-\Phi^0\|^{2+\alpha}_{H^1_\#}
+ \|\Psi-\Phi^0\|^{4}_{H^1_\#} \right).
\end{equation}
\end{lemma}

\medskip

\begin{proof}
Using the fact that  the first order optimality condition (\ref{supeq64})
also reads   
$[{E^{\rm KS}}'(\Phi^0)]_i = 4{\cal H}^{\rm KS}_{\rho^0}\phi_i^0 =
4\epsilon_i^0\phi_i^0$ in $H^{-1}_\#(\Gamma)$, we have for all $\Psi \in
{\cal M}$,  
\begin{eqnarray*}
E^{\rm KS}(\Psi) &=& E^{\rm KS}(\Phi^0) + 
\langle {E^{\rm KS}}'(\Phi^0),\Psi-\Phi^0\rangle_{H^{-1}_\#,H^1_\#} 
+ \frac 12 {E^{\rm KS}}''(\Phi^0)(\Psi-\Phi^0,\Psi-\Phi^0) \\
&+& \int_0^1 ({E^{\rm KS}}''(\Phi^0+s(\Psi-\Phi^0))-{E^{\rm KS}}''(\Phi^0))
(\Psi-\Phi^0,\Psi-\Phi^0) \, (1-s) \, ds \\
&=& E^{\rm KS}(\Phi^0) + 4 \sum_{i=1}^\cN
 \epsilon_i^0 \int_\Gamma  \phi_i^0 (\psi_i-\phi_i^0) 
+ \frac 12 {E^{\rm KS}}''(\Phi^0)(\Psi-\Phi^0,\Psi-\Phi^0) \\
&+& \int_0^1 ({E^{\rm KS}}''(\Phi^0+s(\Psi-\Phi^0))-{E^{\rm KS}}''(\Phi^0))
(\Psi-\Phi^0,\Psi-\Phi^0) \, (1-s) \, ds \\ 
&=& E^{\rm KS}(\Phi^0) - 2 \sum_{i=1}^\cN
\epsilon_i^0 \int_\Gamma (\psi_i-\phi_i^0)^2 
+ \frac 12 {E^{\rm KS}}''(\Phi^0)(\Psi-\Phi^0,\Psi-\Phi^0) \\
&+& \int_0^1 ({E^{\rm KS}}''(\Phi^0+s(\Psi-\Phi^0))-{E^{\rm KS}}''(\Phi^0))
(\Psi-\Phi^0,\Psi-\Phi^0) \, (1-s) \, ds \\ 
&=& E^{\rm KS}(\Phi^0) + 2 a_{\Phi^0}(\Psi-\Phi^0,\Psi-\Phi^0) + R(\Psi-\Phi^0),
\end{eqnarray*}
where
$$
R(\Upsilon)= \int_0^1 ({E^{\rm KS}}''(\Phi^0+s\Upsilon)-{E^{\rm KS}}''(\Phi^0))
(\Upsilon,\Upsilon) \, (1-s) \, ds.
$$ 
The estimate (\ref{eq:coercdisc_2}) then straightforwardly follows from Lemma~\ref{lem:estimR}.
\end{proof}

\subsection{Existence of a discrete solution}

In this subsection, we derive, for 
$N_c$ large enough, the existence of a unique local minimum of the
discretized problem (\ref{eq:min_KS_Nc_U}) in the neighborhood of $\pi^{\cal M}_{N_c} \Phi^0$. 

\medskip

Let 
$$
{\cal B}_{N_c} = \left\{ W^{N_c} \in V_{N_c}^\cN \cap
[\pi^{\cal M}_{N_c}\Phi^{0}]^{\perp\!\!\!\perp} \; | \; 0 \le
M_{W^{N_c},W^{N_c}} \le 1 \right\},
$$
and ${\cal E}_{N_c}$ be the energy functional defined on ${\cal B}_{N_c}$ by
\begin{equation}
{\cal E}_{N_c}(W^{N_c}) = E^{\rm KS}\left( 
\pi^{\cal M}_{N_c}\Phi^0 + \cS(W^{N_c}) \pi^{\cal M}_{N_c}\Phi^0 + W^{N_c} \right).
\label{eq:A2}
\end{equation}
According to the fourth assertion of Lemma~\ref{Lem:unitary_invariance}, the application 
\begin{eqnarray*}
{\cal C} \; : \; {\mathop{{\cal B}}}_{N_c} &\rightarrow& V_{N_c}^\cN \cap {\cal M}^{\pi^{\cal M}_{N_c}\Phi^0} \\ 
  W^{N_c} &\mapsto&
\pi^{\cal M}_{N_c}\Phi^0 + \cS(W^{N_c}) \pi^{\cal M}_{N_c}\Phi^0 + W^{N_c}
\end{eqnarray*}
defines a global map of $V_{N_c}^\cN \cap {\cal M}^{\pi^{\cal M}_{N_c}\Phi^0}$ such that ${\cal C}(0) =\pi^{\cal M}_{N_c}\Phi^0$. Therefore the minimizers of 
\begin{equation} \label{eq:min_KS_Pi}
\inf \left\{ E^{\rm KS}(\Phi^{N_c}), \; \Phi^{N_c} \in V_{N_c}^\cN \cap {\cal M}^{\pi^{\cal M}_{N_c}\Phi^0} \right\}
\end{equation}
are in one-to-one correspondence with those of the 
minimization problem
\begin{equation} \label{eq:min_cE}
\inf \left\{ {\cal E}_{N_c}(W^{N_c}), \; W^{N_c} \in {\cal B}_{N_c}\right\}.
\end{equation}
In a first stage, we prove that for $N_c$ large enough, (\ref{eq:min_cE}) has a unique solution in some neighborhood of $0$. As a consequence (\ref{eq:min_KS_Pi}) has a unique solution in the vicinity of $\pi^{\cal M}_{N_c}\Phi^0$ (for $N_c$ large enough). In a second stage, we make use of the unitary invariance (\ref{eq:invariance_property}) to prove that for $N_c$ large enough, (\ref{eq:min_KS_Nc_U}) has a unique solution in the vicinity of $\Phi^0$. 

\medskip

\begin{lemma} \label{lem:eu_cE}
There exists $r > 0$ and $N_c^0$ such that for all $N_c \ge N_c^0$, the functional ${\cal E}_{N_c}$ has a unique critical point $W^{N_c}_0$ in the ball
$$
\left\{ W^{N_c} \in V_{N_c}^\cN \cap
[\pi^{\cal M}_{N_c}\Phi^{0}]^{\perp\!\!\!\perp} \; | \; \|W^{N_c}\|_{H^1_\#} \le r \right\}. 
$$
Besides, $W^{N_c}_0$ is a local minimizer of (\ref{eq:min_cE}) and we have the estimate
\begin{equation}
\|W^{N_c}_0\|_{H^1_\#} \le \frac{32C_{\Phi^0}^3}{c_{\Phi_0}^3} 
\| \pi^{\cal M}_{N_c}\Phi^0-\Phi^0\|_{H^1_\#}.
\end{equation} 
\end{lemma}

\medskip

\begin{proof}
We infer from Lemma~\ref{lem:estim_2} that 
\begin{eqnarray*}
\!\!\!\!\!\!\!\!\!\!\!\!\!\!\!\!\!\!\!\!
{\cal E}_{N_c}(W^{N_c}) &=&  
E^{\rm KS}\left( \Phi^0+
(\pi^{\cal M}_{N_c}\Phi^0-\Phi^0) + \cS(W^{N_c}) \pi^{\cal M}_{N_c}\Phi^0
+ W^{N_c} \right) \\
\!\!\!\!\!\!\!\!\!\!
& = & E^{\rm KS}\left( \Phi^0 \right) \\
\!\!\!\!\!\!\!\!\!\! 
&&
+ 2a_{\Phi^0} \left( (\pi^{\cal M}_{N_c}\Phi^0-\Phi^0) + \cS(W^{N_c})
  \pi^{\cal M}_{N_c}\Phi^0 + W^{N_c}, (\pi^{\cal M}_{N_c}\Phi^0-\Phi^0)
  + \cS(W^{N_c}) 
  \pi^{\cal M}_{N_c}\Phi^0 + W^{N_c} \right) \\
&& + R\left( (\pi^{\cal M}_{N_c}\Phi^0-\Phi^0) + \cS(W^{N_c})
  \pi^{\cal M}_{N_c}\Phi^0 + W^{N_c}\right) \\
&=& E^{\rm KS}\left( \Phi^0 \right) + 2a_{\Phi^0}(W^{N_c},W^{N_c}) 
+ 4 a_{\Phi^0} \left(  W^{N_c}, (\pi^{\cal M}_{N_c}\Phi^0-\Phi^0)
\right) \\ &&
+  2a_{\Phi^0}(\pi^{\cal M}_{N_c}\Phi^0-\Phi^0,\pi^{\cal
  M}_{N_c}\Phi^0-\Phi^0)
+ {\cal R}_{N_c}(W^{N_c})
\end{eqnarray*}
where
\begin{eqnarray*}
{\cal R}_{N_c}(W_{N_c}) & = & 
2a_{\Phi^0}(\cS(W^{N_c})\pi^{\cal M}_{N_c}\Phi^0,\cS(W^{N_c})\pi^{\cal M}_{N_c}\Phi^0)  \\ && 
+4a_{\Phi^0}(\cS(W^{N_c})\pi^{\cal M}_{N_c}\Phi^0,(\pi^{\cal
  M}_{N_c}\Phi^0-\Phi^0)+W^{N_c}) \\ &&
+ R\left( (\pi^{\cal M}_{N_c}\Phi^0-\Phi^0) + \cS(W^{N_c})
  \pi^{\cal M}_{N_c}\Phi^0 + W^{N_c}\right).
\end{eqnarray*}
Thus,
\begin{eqnarray}
\forall W^{N_c} \in {\cal B}_{N_c}, 
\quad {\cal E}_{N_c}(W^{N_c}) & = & {\cal E}_{N_c}(0)
+ 2a_{\Phi^0}(W^{N_c},W^{N_c}) 
+ 4 a_{\Phi^0} \left(  W^{N_c}, (\pi^{\cal M}_{N_c}\Phi^0-\Phi^0)
\right)\nonumber \\ && + {\cal R}_{N_c}(W^{N_c})-{\cal R}_{N_c}(0).\label{eq:A3}
\end{eqnarray}
It follows from Lemma~\ref{lem:estim_2}, (\ref{eq:bound_SWW}) and the
continuity of $a_{\Phi^0}$ on $(H^1_\#(\Gamma))^\cN$ that
\begin{eqnarray*}
\forall W^{N_c} \in {\cal B}_{N_c}, \quad
|{\cal R}_{N_c}(W^{N_c})| &\le& C_{\cal R} \bigg( \|W^{N_c}\|_{H^1_\#}^{2+\alpha}
+ \|W^{N_c}\|_{H^1_\#}^{8} + 
\|\pi^{\cal M}_{N_c}\Phi^0-\Phi^0\|_{H^1_\#}^{2+\alpha} \\ && 
+ \|\pi^{\cal M}_{N_c}\Phi^0-\Phi^0\|_{H^1_\#}^4 +
\|\pi^{\cal M}_{N_c}\Phi^0-\Phi^0\|_{H^1_\#} \|W^{N_c}\|_{H^1_\#}^{2} \bigg),
\end{eqnarray*}
for a constant $C_{\cal R} \ge 0$ independent of $N_c$. 
Let us introduce for $N_c \ge 0$ and $r > 0$ the ball
$$
B_{N_c}(r) = \left\{  W^{N_c} \in  V_{N_c}^\cN \cap
[\pi^{\cal M}_{N_c}\Phi^{0}]^{\perp\!\!\!\perp} \; | \; 
a_{\Phi^0}(W^{N_c},W^{N_c}) <  r^2 
 a_{\Phi^0} (\pi^{\cal M}_{N_c}\Phi^0- \Phi^0, \pi^{\cal M}_{N_c}\Phi^0-
 \Phi^0) \right\}.
$$
We deduce from Lemma~\ref{lem:coercdisc}, that for all $r > 0$ and all
$N_c \ge N_c^*$, we have  
$$
\forall W^{N_c} \in \partial B_{N_c}(r), \quad 
\sqrt{\frac{c_{\Phi^0}}{2C_{\Phi^0}}} \, r \|\pi^{\cal M}_{N_c}\Phi^0-
\Phi^0\|_{H^1_\#} \le \|W^{N_c}\|_{H^1_\#} 
\le \sqrt{\frac{2C_{\Phi^0}}{c_{\Phi^0}}} \, r \|\pi^{\cal M}_{N_c}\Phi^0-
\Phi^0\|_{H^1_\#} .
$$
Let $r_0 = 2 (2C_{\Phi^0}/c_{\Phi_0})^{5/2}$. For all $r > r_0$, there
exists $N_{c,r} \ge N_c^*$ such that 
$$
\forall N_c \ge N_{c,r}, \quad \partial B_{N_c}(r) \subset {\cal
  B}_{N_c} \quad \mbox{and} \quad \forall W^{N_c} \in \partial B_{N_c}(r),
\quad \|W^{N_c}\|_{H^1_\#} \le 1.
$$
Therefore, for all $r > r_0$ and all $N_c \ge N_{c,r}$ we have $\partial
B_{N_c}(r) \subset {\cal B}_{N_c}$ and 
\begin{eqnarray*}
\forall W^{N_c} \in \partial B_{N_c}(r), \nonumber\\
{\cal E}_{N_c}(W^{N_c}) & \ge & {\cal E}_{N_c}(0)
+ c_{\Phi^0} \|W^{N_c}\|_{H^1_\#}^2 - 4 C_{\Phi^0}  \|W^{N_c}\|_{H^1_\#}
 \|\pi^{\cal M}_{N_c}\Phi^0-\Phi^0\|_{H^1_\#} 
 \\ && - C_{\cal R} \bigg( \|W^{N_c}\|_{H^1_\#}^{2+\alpha}
+ \|W^{N_c}\|_{H^1_\#}^{8} + 
2\|\pi^{\cal M}_{N_c}\Phi^0-\Phi^0\|_{H^1_\#}^{2+\alpha} \\ && 
+2\|\pi^{\cal M}_{N_c}\Phi^0-\Phi^0\|_{H^1_\#}^4 +
\|\pi^{\cal M}_{N_c}\Phi^0-\Phi^0\|_{H^1_\#} \|W^{N_c}\|_{H^1_\#}^{2}
\bigg) \\
&\ge & {\cal E}_{N_c}(0)
+ c_{\Phi^0} \|W^{N_c}\|_{H^1_\#}^2 - 4 C_{\Phi^0}  \|W^{N_c}\|_{H^1_\#}
 \|\pi^{\cal M}_{N_c}\Phi^0-\Phi^0\|_{H^1_\#} 
 \\ && - 5 C_{\cal R} \bigg( \|W^{N_c}\|_{H^1_\#}^{2+\alpha}
+ \|\pi^{\cal M}_{N_c}\Phi^0-\Phi^0\|_{H^1_\#}^{2+\alpha} \bigg) \\
&\ge & {\cal E}_{N_c}(0)
+ \frac{c_{\Phi^0}^2}{2C_{\Phi^0}} r(r-r_0)
 \|\pi^{\cal M}_{N_c}\Phi^0-\Phi^0\|_{H^1_\#}^2 \\ &&
- 5 C_{\cal R} \bigg( 1+ \bigg(\frac{2C_{\Phi^0}}{c_{\Phi^0}}\bigg)^{1+\alpha/2} \,
r^{2+\alpha} \bigg) \|\pi^{\cal
  M}_{N_c}\Phi^0-\Phi^0\|_{H^1_\#}^{2+\alpha}.  
\end{eqnarray*}
As $\| \pi^{\cal M}_{N_c}\Phi^0-\Phi^0\|_{H^1_\#}$ goes to zero when
$N_c$ goes to infinity, we finally obtain that for all $r > r_0$, there
exists some $N_{c,r}' \ge N_c^*$ such that for all $N_c \ge N_{c,r}'$,
$$
\partial B_{N_c}(r) \subset {\cal B}_{N_c} \quad \mbox{and} \quad 
\forall W^{N_c} \in \partial B_{N_c}(r), \quad 
{\cal E}_{N_c}(W^{N_c}) > {\cal E}_{N_c}(0).
$$
This proves that for each $N_c \ge N_{c,2r_0}'$, ${\cal E}_{N_c}$ has a
minimizer $W^{N_c}_0$ in the ball $B_{N_c}(2r_0)$. In particular, 
\begin{equation} \label{eq:majorationW0}
\|W^{N_c}_0\|_{H^1_\#} \le \frac{32C_{\Phi^0}^3}{c_{\Phi_0}^3} 
\| \pi^{\cal M}_{N_c}\Phi^0-\Phi^0\|_{H^1_\#}.
\end{equation}
Let $W^{N_c}_1$ be a critical point of ${\cal E}_{N_c}$ such that $\|W^{N_c}_1\|_{L^2_\#} \le \frac 12$. We denote by $\delta W^{N_c} = W^{N_c}_1-W^{N_c}_0$, 
\begin{eqnarray*}
\widetilde \Phi^0_{N_c} & = & \pi^{\cal M}_{N_c}\Phi^0 + \cS(W^{N_c}_0) \pi^{\cal M}_{N_c}\Phi^0 + W^{N_c}_0, \\
\widetilde \Phi^1_{N_c} & = & \pi^{\cal M}_{N_c}\Phi^0 + \cS(W^{N_c}_1) \pi^{\cal M}_{N_c}\Phi^0 + W^{N_c}_1.
\end{eqnarray*}
As both $W^{N_c}_0$ and $W^{N_c}_1$  are critical points of ${\cal E}_{N_c}$, we have
\begin{eqnarray*}
&& {\cal E}_{N_c}'(W^{N_c}_0) \cdot (W^{N_c}_1-W^{N_c}_0) = 0, \\
&& {\cal E}_{N_c}'(W^{N_c}_1) \cdot (W^{N_c}_0-W^{N_c}_1) = 0,
\end{eqnarray*}
so that 
$$
\left({\cal E}_{N_c}'(W^{N_c}_1)-{\cal E}_{N_c}'(W^{N_c}_0)\right) \cdot (W^{N_c}_1-W^{N_c}_0) = 0.
$$
Using the expression (\ref{eq:A3}) for ${\cal E}_{N_c}$,
we can rewrite this equality as
\begin{eqnarray*}
a_{\Phi^0}(\delta W^{N_c},\delta W^{N_c}) &=& b_{\Phi^0}^{N_c}(W^{N_c}_0,W^{N_c}_1,\delta W^{N_c}) + c_{\Phi^0}(\widetilde \Phi^0_{N_c},\widetilde \Phi^1_{N_c},W^{N_c}_0,W^{N_c}_1,\delta W^{N_c}),
\end{eqnarray*}
where 
\begin{eqnarray*}
 b_{\Phi^0}^{N_c}(W^{N_c}_0,W^{N_c}_1,\delta W^{N_c}) & = & -
a_{\Phi^0}((\cS(W^{N_c}_1)-\cS(W^{N_c}_0)) \pi^{\cal M}_{N_c}\Phi^0,
(\cS'(W^{N_c}_1) \cdot \delta W^{N_c}) \pi^{\cal M}_{N_c}\Phi^0 + \delta W^{N_c}) 
\\ && \!\!\!\!\!\!\!\!\!\!\!\!\!\!\!\!\!\!\!\!\!\!\!\!\!\!\!\!\!\!\!\!\!\!\!\!\!\!\!\!\!\!\!\!\!\!\!\!\!\!\!\!\!\!\!\!
- a_{\Phi^0}(((\cS'(W^{N_c}_1)-\cS'(W^{N_c}_0))\cdot  \delta W^{N_c}) \pi^{\cal M}_{N_c}\Phi^0, 
( \pi^{\cal M}_{N_c}\Phi^0-\Phi^0)+ \cS(W^{N_c}_0)  \pi^{\cal M}_{N_c}\Phi^0 + W^{N_c}_0) 
\\ && \!\!\!\!\!\!\!\!\!\!\!\!\!\!\!\!\!\!\!\!\!\!\!\!\!\!\!\!\!\!\!\!\!\!\!\!\!\!\!\!\!\!\!\!\!\!\!\!\!\!\!\!\!\!\!\!
- a_{\Phi^0} ((\cS'(W^{N_c}_1) \cdot \delta W^{N_c}) \pi^{\cal M}_{N_c}\Phi^0,\delta W^{N_c})
\end{eqnarray*}
and
\begin{eqnarray*}
c_{\Phi^0}(\widetilde \Phi^0_{N_c},\widetilde \Phi^1_{N_c},W^{N_c}_0,W^{N_c}_1,\delta W^{N_c}) & = & \frac 14 \bigg[ R'(\widetilde\Phi^0_{N_c}-\Phi^0) \cdot ((\cS'(W^{N_c}_0) \cdot \delta W^{N_c}) \pi^{\cal M}_{N_c}\Phi^0 + \delta W^{N_c})\\
&& - R'(\widetilde\Phi^1_{N_c}-\Phi^0) \cdot ((\cS'(W^{N_c}_1) \cdot \delta W^{N_c}) \pi^{\cal M}_{N_c}\Phi^0 + \delta W^{N_c}) \bigg]. 
\end{eqnarray*}
Using Lemma~\ref{lem:properties_of_S} and (\ref{eq:majorationW0}), we obtain that there exists $\widetilde C_{\Phi^0}$ (depending only on $\Phi^0$) and $\widetilde N_c$ such that for all $N_c \ge \widetilde N_c$, 
\begin{eqnarray*}
| b_{\Phi^0}^{N_c}(W^{N_c}_0,W^{N_c}_1,\delta W^{N_c})| &\le & 
\widetilde C_{\Phi^0} \left(  \|\pi^{\cal M}_{N_c}\Phi^0-\Phi^0\|_{H^1_\#} + 
\|W^{N_c}_1\|_{L^2_\#} \right) \|\delta W^{N_c}\|_{H^1_\#}^2.
\end{eqnarray*}
On the other hand, remarking that for all $\Psi \in {\cal M}$ and all $\delta\Psi \in T_\Psi{\cal M}$,
$$
R'(\Psi-\Phi^0)\cdot\delta\Psi = {E^{\rm KS}}'(\Psi)\cdot\delta\Psi - 4 a_{\Phi^0}(\Psi-\Phi^0,\delta\Psi),
$$
and introducing the path $(\Psi(t))_{t \in [0,1]}$, drawn on the manifold ${\cal M}$ and connecting $\widetilde \Phi^0_{N_c}$ and $\widetilde \Phi^1_{N_c}$, defined as 
$$
\Psi(t) = \Phi^0 + \cS(tW^{N_c}_1+(1-t)W^{N_c}_0) \pi^{\cal M}_{N_c}\Phi^0 + tW^{N_c}_1+(1-t)W^{N_c}_0,
$$
we obtain
\begin{eqnarray*}
c_{\Phi^0}(\widetilde \Phi^0_{N_c},\widetilde \Phi^1_{N_c},W^{N_c}_0,W^{N_c}_1,\delta W^{N_c}) & = & \frac 14 \bigg[ {E^{\rm KS}}'(\Psi(0)) \cdot \Psi'(0) - {E^{\rm KS}}'(\Psi(1)) \cdot \Psi'(1) \bigg]  \\
& & - a_{\Phi^0}(\Psi(0)-\Phi^0,\Psi'(0)) + a_{\Phi^0}(\Psi(1)-\Phi^0,\Psi'(1)) \\
 & = & - \int_0^1 \bigg[ \frac 14  {E^{\rm KS}}''(\Psi(t))(\Psi'(t),\Psi'(t))  + \frac 14 {E^{\rm KS}}'(\Psi(t)) \cdot \Psi''(t)  \\
&& \qquad\quad - a_{\Phi^0}(\Psi'(t),\Psi'(t))- a_{\Phi^0}(\Psi(t)-\Phi^0,\Psi''(t)) \bigg] \; dt. 
\end{eqnarray*}
As $\Psi(t)=(\psi_1(t),\cdots,\psi_\cN(t)) \in {\cal M}$ for all $t \in [0,1]$, we have for all $1 \le i \le \cN$ and all $t \in [0,1]$,
$$
\int_{\Gamma} \psi_i'(t,x)^2 \, dx = - \int_\Gamma \psi_i(t,x) \, \psi_i''(t,x) \, dx, 
$$
so that
\begin{eqnarray*} \!\!\!\!\!\!\!\!
\frac 14 {E^{\rm KS}}'(\Phi^0)\cdot \Psi''(t)- a_{\Phi^0}(\Psi'(t),\Psi'(t)) &=& \sum_{i=1}^\cN \epsilon_i^0 \int_\Gamma \phi_i^0 \psi_i''(t) \\
&& - \frac 14 {E^{\rm KS}}''(\Phi^0)(\Psi'(t),\Psi'(t)) +  \sum_{i=1}^\cN \epsilon_i^0 \int_\Gamma \psi'_i(t)^2 \\
& = & -  \sum_{i=1}^\cN \epsilon_i^0 \int_\Gamma (\psi_i(t)-\phi_i^0) \psi_i''(t) - \frac 14 {E^{\rm KS}}''(\Phi^0)(\Psi'(t),\Psi'(t)) .
\end{eqnarray*}
Consequently,
\begin{eqnarray*}
\!\!\!\!\!\!\!\!\!\!\!\!
 && \!\!\!\!\!\!\!\!\!\!\!\!c_{\Phi^0}(\widetilde \Phi^0_{N_c},\widetilde \Phi^1_{N_c},W^{N_c}_0,W^{N_c}_1,\delta W^{N_c})
   \; = \;  - \int_0^1 \bigg[ \frac 14 \left( {E^{\rm KS}}''(\Psi(t))-{E^{\rm KS}}''(\Phi^0) \right) (\Psi'(t),\Psi'(t)) \\
&&\!\!\!\!\!\!\!\!\!\!\!\! \quad + \frac 14 \left( {E^{\rm KS}}'(\Psi(t)) -{E^{\rm KS}}'(\Phi^0) \right) \cdot \Psi''(t) -  \sum_{i=1}^\cN \epsilon_i^0 \int_\Gamma (\psi_i(t)-\phi_i^0) \psi_i''(t) - a_{\Phi^0}(\Psi(t)-\Phi^0,\Psi''(t)) \bigg] \; dt.
\end{eqnarray*}
Using Lemma~\ref{lem:estimR}, we obtain
\begin{eqnarray*}
|c_{\Phi^0}(\widetilde \Phi^0_{N_c},\widetilde \Phi^1_{N_c},W^{N_c}_0,W^{N_c}_1,\delta W^{N_c})| &\le & C \int_0^1 \bigg[ \left( \|\Psi(t)-\Phi^0\|_{H^1_\#}^\alpha 
+ \|\Psi(t)-\Phi^0\|_{H^1_\#}^2 \right)  \|\Psi'(t)\|_{H^1_\#}^2 \\
&& \qquad\qquad  \|\Psi(t)-\Phi^0\|_{H^1_\#}\|\Psi''(t)\|_{H^1_\#} \bigg] \; dt.
\end{eqnarray*}
As
\begin{eqnarray*}
\Psi'(t) &=&  (\cS'(tW^{N_c}_1+(1-t)W^{N_c}_0)\cdot\delta W^{N_c}) \pi^{\cal M}_{N_c}\Phi^0 + \delta W^{N_c}, \\
\Psi''(t) &=&  (\cS''(tW^{N_c}_1+(1-t)W^{N_c}_0)(\delta W^{N_c},\delta W^{N_c})) \pi^{\cal M}_{N_c}\Phi^0,
\end{eqnarray*}
we obtain that there exists some constant $C \in \R_+$ such that for $N_c$ large enough,
$$
|c_{\Phi^0}(\widetilde \Phi^0_{N_c},\widetilde \Phi^1_{N_c},W^{N_c}_0,W^{N_c}_1,\delta W^{N_c})| \le C \left( \|\pi^{\cal M}_{N_c}\Phi^0-\Phi^0\|_{H^1_\#}^\alpha + \|W^{N_c}_1\|_{H^1_\#}^\alpha \right) \|\delta W^{N_c} \|_{H^1_\#}^2.
$$
Thus,
\begin{eqnarray*}
\frac{c_{\Phi^0}}{2}  \|\delta W^{N_c}\|_{H^1_\#} & \le & |a_{\Phi^0}(\delta W^{N_c},\delta W^{N_c})| \\
& = & |b_{\Phi^0}^{N_c}(W^{N_c}_0,W^{N_c}_1,\delta W^{N_c}) + c_{\Phi^0}(\widetilde \Phi^0_{N_c},\widetilde \Phi^1_{N_c},W^{N_c}_0,W^{N_c}_1,\delta W^{N_c})| \\
& \le & C \left(  \|\pi^{\cal M}_{N_c}\Phi^0-\Phi^0\|_{H^1_\#}^\alpha + \|W^{N_c}_1\|_{H^1_\#}^\alpha \right) \|\delta W^{N_c} \|_{H^1_\#}^2.
\end{eqnarray*}
This proves that there exists a constant $r > 0$ such that for all $N_c$ large enough, $\|W^{N_c}_1\|_{H^1_\#} \le r$ implies $\delta W^{N_c}=0$. Hence the result.  
\end{proof}

As the mapping $B_{N_c}(2r_0) \ni W^{N_c} \mapsto \pi^{\cal
  M}_{N_c}\Phi^0 + \cS(W^N_c)\pi^{\cal M}_{N_c}\Phi^0 + W^{N_c}$ defines
a local map of $V_{N_c}^\cN \cap {\cal M}^{\pi^{\cal M}_{N_c}\Phi^0}$ in
the neighborhood of $\pi^{\cal M}_{N_c}\Phi^0$, we obtain that 
$\widetilde \Phi^0_{N_c} = \pi^{\cal M}_{N_c}\Phi^0 + \cS(W^{N_c}_0)
\pi^{\cal M}_{N_c}\Phi^0 + W^{N_c}_0$ is the unique local minimizer of 
$$
\inf \left\{ E^{\rm KS}(\Phi_{N_c}), \; 
\Phi_{N_c} \in V_{N_c}^\cN \cap {\cal M}^{\pi^{\cal M}_{N_c}\Phi^0} \right\},
$$
in the vicinity of $\pi^{\cal M}_{N_c}\Phi^0$. Besides,
\begin{eqnarray*}
\| \widetilde \Phi_{N_c}^0 - \Phi^0 \|_{H^1_\#} & \le &
\| \widetilde \Phi_{N_c}^0 - \pi^{\cal M}_{N_c}\Phi^0 \|_{H^1_\#} + 
\| \pi^{\cal M}_{N_c}\Phi^0 - \Phi^0 \|_{H^1_\#} \\
& \le & \|S(W^{N_c}_0) \pi^{\cal M}_{N_c}\Phi^0 + W^{N_c}_0\|_{H^1_\#} + 
\| \pi^{\cal M}_{N_c}\Phi^0 - \Phi^0 \|_{H^1_\#} \\
&\le & C \| \Pi_{N_c} \Phi^0-\Phi^0\|_{H^1_\#},
\end{eqnarray*}
for a constant $C$ independent of $N_c$. We then have
$$
\|M_{\widetilde \Phi_{N_c}^0,\Phi^0}-1_\cN \|_{\rm F} \le \|\widetilde
\Phi_{N_c}^0 - \Phi^0 \|_{L^2_\#} \le C \| \Pi_{N_c}
\Phi^0-\Phi^0\|_{H^1_\#}.
$$
Let $\Phi_{N_c}^0 = U_{\widetilde \Phi_{N_c}^0,\Phi^0}\widetilde
\Phi_{N_c}^0$, where  
$U_{\widetilde \Phi_{N_c}^0,\Phi^0} = 
M_{\widetilde\Phi_{N_c}^0,\Phi^0}^T
(M_{\widetilde\Phi_{N_c}^0,\Phi^0}M_{\widetilde\Phi_{N_c}^0,\Phi^0}^T)^{-1/2}$.
Then for each $N_c \ge N_{c,2r_0}'$,  
$\Phi_{N_c}^0$ is the unique local minimizer of~(\ref{eq:min_KS_Nc_U}) in the set
$$
\left\{\Phi_{N_c} \in V_{N_c}^\cN \cap {\cal M}^{\Phi^0} \; | \;  \|\Phi_{N_c}-\Phi^0\|_{H^1_\#} \le r^0 \right\},
$$ 
for some constant $r^0 > 0$ independent of $N_c$,  and it satisfies
\begin{equation} \label{eq:presetimH1}
\| \Phi_{N_c}^0 - \Phi^0 \|_{H^1_\#} \le C \| \Pi_{N_c}
\Phi^0-\Phi^0\|_{H^1_\#},
\end{equation}
for some $C \in \R_+$ independent of $N_c$. 

As $\Phi_{N_c}^0 \in {\cal
  M}^{\Phi^0}$, we can decompose $\Phi_{N_c}^0$ as
\begin{equation}\label{eq:decompdisc}
\Phi_{N_c}^0 = \Phi^0 + S^0_{N_c} \Phi_0 + W_{N_c}^0
\end{equation}
where $S^0_{N_c}=\cS(W_{N_c}^0)$ and $W_{N_c}^0 \in \Phi^{0,\perp\!\!\!\perp}$ (note that $W_{N_c}^0 \notin V_{N_c}^\cN$ in general). As 
\begin{equation} \label{eq:estimS0Nc}
\|S^0_{N_c}\|_{\rm F} \le \| W^0_{N_c}\|_{L^2_\#}^2
\end{equation}
and $\|\Phi_{N_c}^0 - \Phi^0\|_{H^1_\#}$ goes to zero when $N_c$ goes to infinity, we have, for $N_c$ large enough,
\begin{eqnarray}
&& \frac 12 \|W^0_{N_c}\|_{L^2_\#} \le \|\Phi_{N_c}^0 - \Phi^0\|_{L^2_\#} \le 
2 \|W^0_{N_c}\|_{L^2_\#}, \label{eq:estimW0NcL2} \\
&& \frac 12 \|W^0_{N_c}\|_{H^1_\#} \le \|\Phi_{N_c}^0 - \Phi^0\|_{H^1_\#} \le 
2 \|W^0_{N_c}\|_{H^1_\#}.  \label{eq:estimW0NcH1}
\end{eqnarray}
The discrete solution $\Phi_{N_c}^0$ satisfies the
Euler equations 
$$
\forall \Psi_{N_c} \in V_{N_c}^\cN, \quad 
\langle {\cal H}^{{\rm KS}}_{\rho^0_{N_c}} \phi_{i,N_c}^0,\psi_i
\rangle_{H^{-1}_\#,H^1_\#}  = \sum_{j=1}^\cN [\lambda_{N_c}^0]_{ij}
(\phi_{j,N_c}^0,\psi_j)_{L^2_\#},  
$$
where $\rho^0_{N_c}=\rho_{\Phi^0_{N_c}}$ and where the $\cN \times \cN$
matrix $\Lambda_{N_c}^0$ is symmetric (but generally not diagonal). Of
course, it follows from the invariance property
(\ref{eq:invariance_property}) that (\ref{eq:min_KS_Nc_U}) has a local
minimizer of the form $U\Phi_{N_c}^0$ with $U \in \cU(\cN)$ for which
the Lagrange multiplier of the orthonormality constraints is a diagonal
matrix.

\subsection{A priori error estimates}

We are now in position to derive {\it a priori} estimates for $\|\Phi^0_{N_c}-\Phi^0\|_{H^s_\#}$ and $(\Lambda_{N_c}^0-\Lambda^0)$, where we recall that
$\Lambda_0=\mbox{diag}(\epsilon_1^0,\cdots,\epsilon_\cN^0)$. 

Using (\ref{eq:app-Fourier}), (\ref{eq:presetimH1}) and the inverse inequality (\ref{eq:inv_ineq}), we obtain for each $s \ge 1$ such that $\Phi^0 \in \left(H^s_\#(\Gamma)\right)^\cN$ and each $1 \le r \le s$, 
\begin{eqnarray} 
\|\Phi^0_{N_c}-\Phi^0\|_{H^r_\#} &\le & \|\Phi^0_{N_c}-\Pi_{N_c} \Phi^0\|_{H^r_\#}+\|\Pi_{N_c}\Phi^0-\Phi^0\|_{H^r_\#} \nonumber \\
&\le & C N_c^{r-1} \|\Phi^0_{N_c}-\Pi_{N_c} \Phi^0\|_{H^1_\#}+\| \Pi_{N_c}\Phi^0-\Phi^0\|_{H^r_\#} \nonumber \\
&\le & C N_c^{r-1} \left( \|\Phi^0_{N_c}- \Phi^0 \|_{H^1_\#} + \|\Phi^0-\Pi_{N_c} \Phi^0\|_{H^1_\#} \right) + \| \Pi_{N_c}\Phi^0-\Phi^0\|_{H^r_\#}\nonumber  \\
&\le & C N_c^{r-1} \|\Pi_{N_c} \Phi^0 - \Phi^0\|_{H^1_\#} + \| \Pi_{N_c}\Phi^0-\Phi^0\|_{H^r_\#} \nonumber \\ 
& \le &  C N_c^{-(s-r)} \|\Pi_{N_c} \Phi^0 -\Phi^0\|_{H^s_\#}. \label{eq:1estimKS}
\end{eqnarray}
In particular, for $s=3$ and $r=2$, we obtain that  $\Phi^0_{N_c}$ converges to $\Phi^0$ in $(H^2_\#(\Gamma))^\cN$, hence in  $(L^\infty_\#(\Gamma))^\cN$.

We then proceed as in (\ref{eq:val_ppe}) and remark that
\begin{eqnarray} \!\!\!\!\!\!\!\!\!\!\!\!\!\!\!\!\!\!
\lambda_{ij,N_c}^0-\lambda_{ij}^0 & = &
\langle {\cal H}^{{\rm KS}}_{\rho^0_{N_c}}\phi_{i,N_c}^0,\phi_{j,N_c}^0
\rangle_{H^{-1}_\#,H^1_\#} - \langle {\cal H}^{{\rm KS}}_{\rho^0}\phi_i^0,\phi_j^0 \rangle_{H^{-1}_\#,H^1_\#}
  \nonumber \\
&=&  \langle {\cal H}^{{\rm KS}}_{\rho^0} (\phi_{i,N_c}^0-\phi_i^0), (\phi_{i,N_c}^0-\phi_j^0)   \rangle_{H^{-1}_\#,H^1_\#} \nonumber\\
&&+ \epsilon_i^0\int_\Gamma\phi_i^0(\phi_{j,N_c}^0 - \phi_{j}^0) + \epsilon_j^0\int_\Gamma\phi_j^0(\phi_{i,N_c}^0 - \phi_{i}^0\nonumber)\\
&& + \int_\Gamma  V^{\rm Coulomb}_{\phi_{i,N_c}^0\phi_{j,N_c}^0} (\rho_{N_c}^0 - \rho^0 )     \nonumber \\ &&
+ \int_\Gamma \bigg( \frac{de_{\rm xc}^{\rm LDA}}{d\rho}(\rho_{\rm c}+\rho^0_{N_c}) -\frac{de_{\rm xc}^{\rm LDA}}{d\rho}(\rho_{\rm c}+\rho^0) \bigg)
\phi_{i,N_c}^0\phi_{j,N_c}^0. \label{2K10}
\end{eqnarray}
As, from (\ref{eq:decompdisc}),
$$
\epsilon_i^0\int_\Gamma\phi_i^0(\phi_{j,N_c}^0-\phi_j^0) +
\epsilon_j^0\int_\Gamma\phi_j^0(\phi_{i,N_c}^0-\phi_i^0) = 
(\epsilon_i^0+\epsilon_j^0)[S_{N_c}^0]_{ij},
$$
we easily obtain, using the convergence of  $\Phi^0_{N_c}$ to $\Phi^0$ in $(H^1_\#(\Gamma)\cap L^\infty_\#(\Gamma))^\cN$, 
\begin{equation} \label{eq:1estimLambda}
\| \Lambda^0_{N_c} - \Lambda^0\|_{\rm F} \mathop{\longrightarrow}_{N_c \to \infty} 0.
\end{equation}
For $W \in (L^2_\#(\Gamma))^\cN$, we introduce the adjoint problem
\begin{equation}\label{eq:adjoint_KS}
\left\{ \begin{array}{l}
\mbox{find $\Psi_W \in \Phi^{0,\perp\!\!\!\perp}$ such that} \\
\forall Z \in \Phi^{0,\perp\!\!\!\perp}, \; a_\Phi(\Psi_W,Z) = (W,Z)_{L^2_\#},
\end{array}\right.
\end{equation}
the solution of whom exists and is unique by the coercivity assumption (\ref{eq:coerc}). Clearly,
\begin{equation} \label{eq:estim_Psi}
\| \Psi_W \|_{H^1_\#} \le C \| W \|_{L^2_\#}.
\end{equation}
In addition, it follows from standard elliptic regularity arguments that
$$
\| \Psi_W \|_{H^2_\#} \le C \| W \|_{L^2_\#},
$$
yielding
\begin{eqnarray} 
&& \| \Psi_W - \Pi_{N_c}\Psi_W  \|_{L^2_\#} \le  CN_c^{-2}
\|W \|_{L^2_\#} \label{eq:estim_Psi_deltaL2} \\
&& \| \Psi_W - \Pi_{N_c}\Psi_W  \|_{H^1_\#} \le CN_c^{-1}
\|W \|_{L^2_\#}. \label{eq:estim_Psi_delta}
\end{eqnarray}
Denoting by $\Psi = \Psi_{\Phi^0_{N_c} - \Phi^0}$ and  using (\ref{eq:decompdisc}), we get
\begin{eqnarray}
\| \Phi^0_{N_c} - \Phi^0 \|_{L^2_\#}^2 & = & (\Phi^0_{N_c} - \Phi^0,\Phi^0_{N_c} -
\Phi^0)_{L^2_\#} \nonumber\\
& = & (\Phi^0_{N_c} - \Phi^0,S^0_{N_c} \Phi^0)_{L^2_\#} + (\Phi^0_{N_c} -
\Phi^0,W^0_{N_c})_{L^2_\#} \nonumber\\
& = & (\Phi^0_{N_c} - \Phi^0,S^0_{N_c} \Phi^0)_{L^2_\#} +
a_{\Phi^0}(\Psi,W^0_{N_c}) \nonumber\\
& = & (\Phi^0_{N_c} - \Phi^0,S^0_{N_c} \Phi^0)_{L^2_\#} - 
a_{\Phi^0}(\Psi,S^0_{N_c} \Phi^0) +
a_{\Phi^0}(\Psi,\Phi^0_{N_c}-\Phi^0)  \nonumber \\
 & = & (\Phi^0_{N_c} - \Phi^0,S^0_{N_c} \Phi^0)_{L^2_\#}  - 
a_{\Phi^0}(\Psi,S^0_{N_c} \Phi^0) +
 a_{\Phi^0}(\Psi - \Pi_{N_c}\Psi,\Phi^0_{N_c}-\Phi^0)   \nonumber \\
 && + a_{\Phi^0}(\Pi_{N_c}\Psi,\Phi^0_{N_c}-\Phi^0). \label{3K11}
\end{eqnarray}
From the definition (\ref{eq:defaKS}), the last term in the above expression reads
$$
a_{\Phi^0}(\Pi_{N_c}\Psi,\Phi^0_{N_c}-\Phi^0) = \frac{1}{4} {E^{\rm KS}}''(\Phi^0)(\Pi_{N_c}\Psi,\Phi^0_{N_c}-\Phi^0) - \sum_{i=1}^N  \sum_{j=1}^N \lambda_{ij}^0
\int_\Gamma  (\phi^0_{j,N_c}-\phi^0_{j})\Pi_{N_c}\psi_i, 
$$
so that from the definition of the continuous and discrete eigenvalue problems
\begin{eqnarray}
4a_{\Phi^0}(\Pi_{N_c}\Psi,\Phi^0_{N_c}-\Phi^0) &=&{E^{\rm KS}}''(\Phi^0)(\Pi_{N_c}\Psi,\Phi^0_{N_c}-\Phi^0) - {E^{\rm KS}}'(\Phi^0_{N_c})(\Pi_{N_c}\Psi) +{E^{\rm KS}}'(\Phi^0)(\Pi_{N_c}\Psi)\nonumber \\
&&+ 4 \sum_{i=1}^N  \sum_{j=1}^N( \lambda_{ij,N_c}^0 - \lambda_{ij}^0)  \int_\Gamma  \phi^0_{j,N_c}\Pi_{N_c}\psi_i.
\end{eqnarray}
The definition of $\Pi_{N_c}$ and the fact that $\Psi \in
\Phi^{0,\perp\!\!\!\perp}$ yields
$$ 
\int_\Gamma \phi^0_{j,N_c}\Pi_{N_c} \psi_i   = \int_\Gamma  (\phi^0_{j,N_c} -\phi^0_j) \psi_i, 
$$ 
which  finally provides the estimate
\begin{eqnarray}
\| \Phi^0_{N_c} - \Phi^0 \|_{L^2_\#}^2 & = &(\Phi^0_{N_c} - \Phi^0,S^0_{N_c} \Phi^0)_{L^2_\#} -
a_{\Phi^0}(\Psi,S^0_{N_c} \Phi^0)  +
a_{\Phi^0}(\Psi - \Pi_{N_c}\Psi,\Phi^0_{N_c}-\Phi^0)  \nonumber\\
&&-\frac 14 \left({E^{\rm KS}}'(\Phi^0_{N_c})(\Pi_{N_c}\Psi) - {E^{\rm KS}}'(\Phi^0)(\Pi_{N_c}\Psi) -{E^{\rm KS}}''(\Phi^0)(\Phi^0_{N_c}-\Phi^0,\Pi_{N_c}\Psi) \right)\nonumber \\
&&+ \sum_{i=1}^N  \sum_{j=1}^N( \lambda_{ij,N_c}^0 - \lambda_{ij}^0)  \int_\Gamma (\phi^0_{j,N_c} -\phi^0_j) \psi_i .
\end{eqnarray}
Using Lemma~\ref{lem:estimR}, (\ref{eq:presetimH1}), (\ref{eq:estimS0Nc}) and (\ref{eq:estim_Psi_delta}), we infer
\begin{eqnarray}
\| \Phi^0_{N_c} - \Phi^0 \|_{L^2_\#} & \le & C \bigg(
\| \Phi^0_{N_c} - \Phi^0 \|_{L^2_\#}^2 + N_{ c}^{-1} \| \Phi^0_{N_c} - \Phi^0
\|_{H^1_\#} + \| \Phi^0_{N_c} - \Phi^0 \|_{L^2_\#}^{1+\alpha}
\| \Phi^0_{N_c} - \Phi^0 \|_{H^1_\#}  \nonumber \\ & & \qquad
+   \|\Lambda^0_{N_c} - \Lambda^0\|_{\rm F} \| \Phi^0_{N_c} - \Phi^0 \|_{L^2_\#}  \bigg). 
\end{eqnarray}
We thus obtain, using (\ref{eq:1estimLambda}) and   the above estimate, 
that asymptotically, when $N_{c}$ goes to infinity,
$$
\| \Phi^0_{N_c} - \Phi^0 \|_{L^2_\#} \le C \, N_{c}^{-1} \|
\Pi_{N_c} \Phi^0 - \Phi^0 \|_{H^1_\#}.
$$
Reasoning as in (\ref{eq:1estimKS}), we obtain that for each $s \ge 1$ such that $\Phi^0 \in \left(H^s_\#(\Gamma)\right)^\cN$ and each $0 \le r \le s$, there exists a constant $C$ such that
\begin{equation} \label{eq:estimHpos}
\| \Phi^0_{N_c} - \Phi^0 \|_{H^r_\#} \le C \, N_{c}^{-(s-r)} \|
\Pi_{N_c} \Phi^0 - \Phi^0 \|_{H^s_\#}.
\end{equation}
To proceed further, we need to make an assumption on the regularity of the exchange-correlation potential. In the sequel, we assume that
\begin{itemize}
\item either the function $\rho \mapsto e_{\rm xc}^{\rm LDA}(\rho)$ is in $C^{[m]}([0,+\infty))$;
\item or the function $\rho_c + \rho^0$ is positive everywhere. As it is continuous on $\R^3$, this is equivalent to assuming that there exists a constant $\eta > 0$ such that for all $x \in \R^3$,  $\rho_c(x) + \rho^0(x) \ge \eta$. 
\end{itemize}
It follows by standard elliptic regularity arguments that $\Phi^0$ then is in $(H^{m+1/2-\epsilon}_\#(\Gamma))^\cN$ for any $\epsilon > 0$, and we deduce from (\ref{eq:estimHpos}) that (\ref{KSHS}) holds true for all $0 \le s < m+1/2$. 

Then, following the same lines as in the proof of~(\ref{eq:pre_estim_lambda}),
we obtain the estimates
\begin{equation*}
\left| \int_\Gamma  V^{\rm Coulomb}_{\phi_{i,N_c}^0\phi_{j,N_c}^0} ( \rho^0_{N_c} - \rho^0 ) \right|  \le C  \| \rho^0_{N_c}-\rho^0 \|_{H^{-r}_\#},
\end{equation*}
and
\begin{eqnarray*}
\left|\int_\Gamma \bigg( \frac{de_{\rm xc}^{\rm LDA}}{d\rho}(\rho_{\rm c}+\rho^0) -
\frac{de_{\rm xc}^{\rm LDA}}{d\rho}(\rho_{\rm c}+\rho^0_{N_c}) \bigg)
\phi_{i,N_c}^0\phi_{j,N_c}^0 \right|   \le c \|\rho^0_{N_c}- \rho^0 \|_{H^{-r}_\#},
\end{eqnarray*}
valid for all $0\le r < m-3/2$. Using these estimates in (\ref{2K10}), we are lead to 
$$
|\lambda_{ij,N_c}^0-\lambda_{ij}^0|  \le  C \left(
\| \Phi^0 - \Phi^0_{N_c} \|_{H^1_\#}^2 +   \| \rho^0_{N_c} -  \rho^0 \|_{H^{-r}_\#}  \right).
$$
Now,
$$ 
\| \rho^0_{N_c}-\rho^0 \|_{H^{-r}_\#} = \sup_{w\in H^r_\#(\Gamma)}\frac {\dps\int_\Gamma  (\rho^0_{N_c}-\rho^0 )  w}{\|w\|_{H^r_\#}}.
$$
Noticing that
$$
\rho^0_{N_c}-\rho^0 =\sum_{i=1}^\cN |\phi^0_{i,N_c}|^2 -  \sum_{i=1}^\cN |\phi_{i}^0|^2 = \sum_{i=1}^\cN (\phi_{i,N_c}^0 - \phi_{i}^0) (\phi_{i,N_c}^0 +\phi_{i}^0), 
$$
we deduce
\begin{equation}
\label{3K10}
\| \rho^0_{N_c}-\rho^0 \|_{H^{-r}_\#} \le C  \| \Phi_{N_c}^0 -  \Phi^0 \|_{H^{-r}_\#},
\end{equation}
since $\Phi^0_{N_c}$ converges, therefore is uniformly bounded in $H^r_\#(\Gamma)$. Thus
\begin{equation} \label{eq:estim_Lambda}
|\Lambda^0_{N_c}-\Lambda^0|  \le  C \left(
\| \Phi^0_{N_c} - \Phi^0 \|_{H^1_\#}^2 + C \| \Phi^0_{N_c} -  \Phi^0 \|_{H^{-r}_\#}  \right).
\end{equation}
The derivation of estimates for $\| \Phi^0_{N_c} -  \Phi^0 \|_{H^{-r}_\#}$ follows exactly the same lines as the derivation of the $L^2$ estimate: starting from the definition
$$
\| \Phi^0_{N_c} -  \Phi^0   \|_{H^{-r}_\#} = \sup_{W \in (H^r_\#(\Gamma))^\cN}
\frac{(W,\Phi^0_{N_c} -  \Phi^0)_{L^2_\#} }{\|W\|_{H^r_\#}},
$$
and remarking that the solution $\Psi_W$ to the adjoint problem (\ref{eq:adjoint_KS}) satisfies 
$$
\|\Psi_W\|_{H^{r+2}_\#} \le C \| W\|_{H^r_\#},
$$
we proceed as in (\ref{3K11}) to get
\begin{eqnarray}
(W,\Phi^0_{N_c} -  \Phi^0)_{L^2_\#} 
& = & (W,S^0_{N_c} \Phi^0)_{L^2_\#} + (W,W^0_{N_c})_{L^2_\#} \nonumber\\
& = & (W,S^0_{N_c} \Phi^0)_{L^2_\#} +
a_{\Phi^0}(\Psi_W,W^0_{N_c}) \nonumber\\
& = & (W,S^0_{N_c} \Phi^0)_{L^2_\#} -
a_{\Phi^0}(\Psi_W,S^0_{N_c} \Phi^0) +
a_{\Phi^0}(\Psi_W,\Phi^0_{N_c}-\Phi^0)  \nonumber\\
& = & (W,S^0_{N_c} \Phi^0)_{L^2_\#} -
a_{\Phi^0}(\Psi_W,S^0_{N_c} \Phi^0) +
a_{\Phi^0}(\Psi_W - \Pi_{N_c}\Psi_W,\Phi^0_{N_c}-\Phi^0)  \nonumber\\
&&+a_{\Phi^0}(\Pi_{N_c}\Psi_W,\Phi^0_{N_c}-\Phi^0),
\end{eqnarray}
that yields
\begin{eqnarray}\| \Phi^0_{N_c} - \Phi^0 \|_{H^{-r}_\#} &\le& C \bigg(
\| \Phi^0_{N_c} - \Phi^0 \|_{L^2_\#}^2 + N_{ c}^{-1-r} \| \Phi^0_{N_c} - \Phi^0
\|_{H^1_\#} + \| \Phi^0_{N_c} - \Phi^0 \|_{H^{-r}_\#}
\| \Phi^0_{N_c} - \Phi^0 \|_{H^1_\#} \nonumber \\ & & \qquad
+   \|\Lambda^0_{N_c} - \Lambda^0\|_{\rm F} \| \Phi^0_{N_c} - \Phi^0 \|_{H^{-r}_\#}  \bigg).
\end{eqnarray}
The proof of (\ref{KSHS}) follows and then we get easily from (\ref{eq:estim_Lambda}) that
\begin{equation} \label{eq:estim_Lambdafinale}
\|\Lambda^0_{N_c}-\Lambda^0\|_{\rm F}  \le  C_\epsilon N_c^{-(2m-1-\epsilon)}.
\end{equation}
Hence (\ref{KSHS'}). Finally, (\ref{KSHS''}) is a straightforward consequence of Lemma~\ref{lem:estim_2}, (\ref{eq:coerc}), (\ref{eq:continuity_aPhi0}), and (\ref{eq:presetimH1}).

\subsection{Numerical results} 
In order to evaluate the quality of the error bounds obtained in Theorem~\ref{Th:KS-LDA}, we have performed numerical tests using the Abinit software~\cite{Abinit1} (freely available online, cf. http://www.abinit.org),  
whose main program allows one to find the total energy, charge density and electronic structure of systems (molecules and periodic solids) within Density Functional Theory (DFT), using pseudopotentials and a planewave basis. \\

We have run simulation tests with the Hartree functional (i.e. with $e^{\rm LDA}_{\rm xc}=0$), for which there is no numerical integration error. In this particular case, the problems (\ref{eq:min_KS_PW}) (solved by Abinit) and (\ref{eq:min_KS_Nc}) (analyzed in Theorem~\ref{Th:KS-LDA}) are identical. 

For Troullier-Martins pseudopotentials, the parameter $m$ in Theorem~\ref{Th:KS-LDA} is equal to $5$. Therefore, we expect the following error bounds (as functions of the cut-off energy $E_{\rm c} = \frac 12 \left( \frac{2\pi N_c}{L} \right)^{2}$) 
\begin{eqnarray} \label{estimationResNum1}
\| \Phi_{N_c}^0-\Phi^0\|_{H^1_\#} & \le & C_{1,\epsilon}
E_{\rm c}^{-2.25+\epsilon}, \\
\| \Phi_{N_c}^0-\Phi^0\|_{L^2_\#} & \le & C_{2,\epsilon}E_{\rm c}^{-2.75+\epsilon}, \\
|\epsilon_{i,N_c}^0-\epsilon_i^0| & \le & C_{3,\epsilon} E_{\rm c}^{-4.5+\epsilon}, \\
0 \le I^{\rm KS}_{N_c}-I^{\rm KS} & \le & C_{4,\epsilon} E_{\rm c}^{-4.5+\epsilon}.
\label{estimationResNum2}
\end{eqnarray} 

The first tests were performed with the Hydrogen molecule (H$_2$). The nuclei were clamped at the points with cartesian coordinates $r_1 = ( -0.7 ; 0; 0)$ and $ r_2 = ( 0.7 ; 0; 0)$ (in Bohrs). The simulation cell was a cube of side length $L=10$~Bohrs. The so-obtained numerical errors are plotted in log-scales in Figures \ref{FigH2a} and \ref{FigH2b}. The second series of tests were performed with the Nitrogen molecule (N$_2$). The nuclei were clamped at positions  $r_1 = ( -0.55  ; 0; 0)$ and $ r_2 = ( 0.55 ; 0; 0)$ (in Angstroms), and the simulation cell was a cube of side length $L=6$~Angstroms. The numerical errors for N$_2$ are plotted in Figures~\ref{FigN2a}, \ref{FigN2b} and~\ref{FigN2}. The  reference values for $\Phi^0$, $\epsilon_i^0$ and $I^{\rm KS}$ for both H$_2$ and N$_2$ are those obtained for a cut-off energy equal to $500$~Hartrees.

\begin{figure}[h!] 
\begin{tabular}{c  c }
\includegraphics[width=7cm]{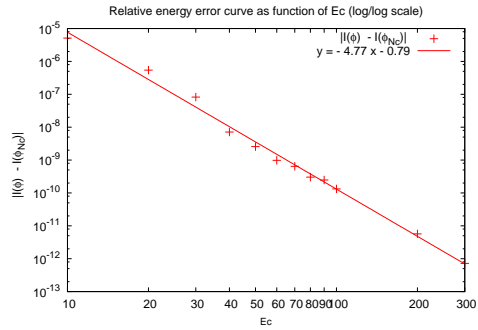}&
\end{tabular}	 
\caption{Error on the energy as a function of $E_{\rm c}$  for H$_2$}\label{FigH2a}
 \end{figure}
\begin{figure}[h!] 
\begin{tabular}{c  c }
\includegraphics[width=7cm]{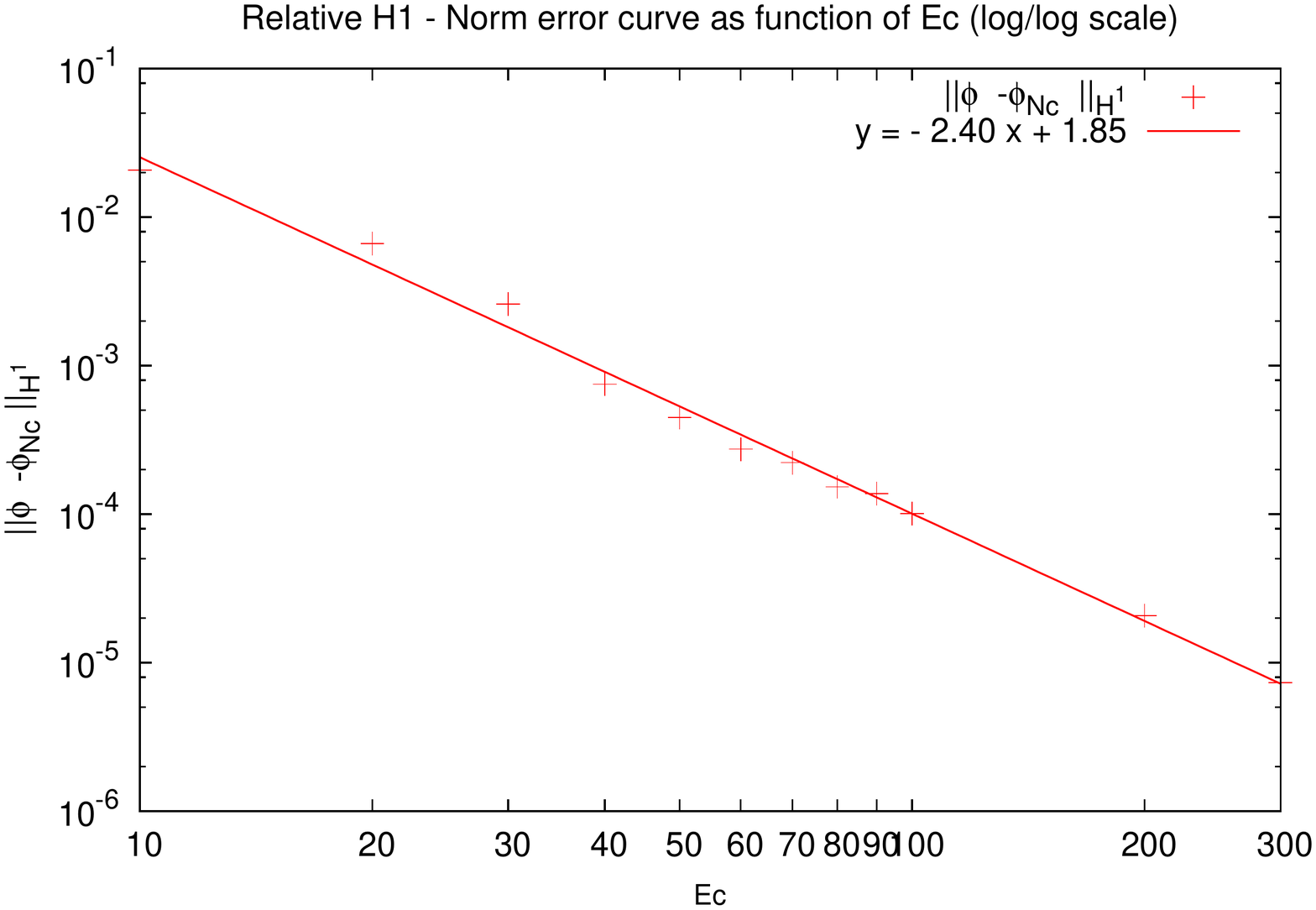}&
\includegraphics[width=7cm]{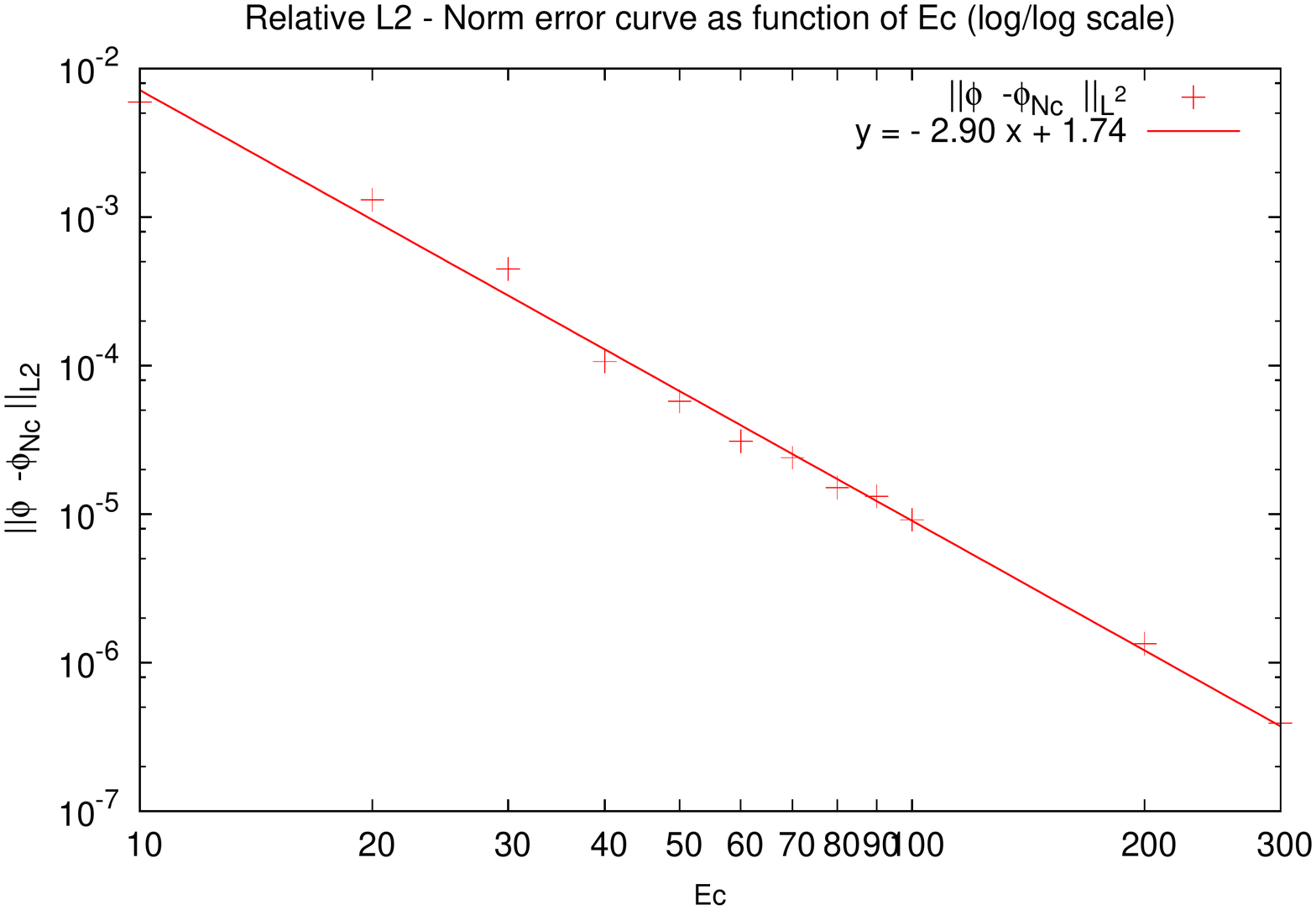}\\
\includegraphics[width=7cm]{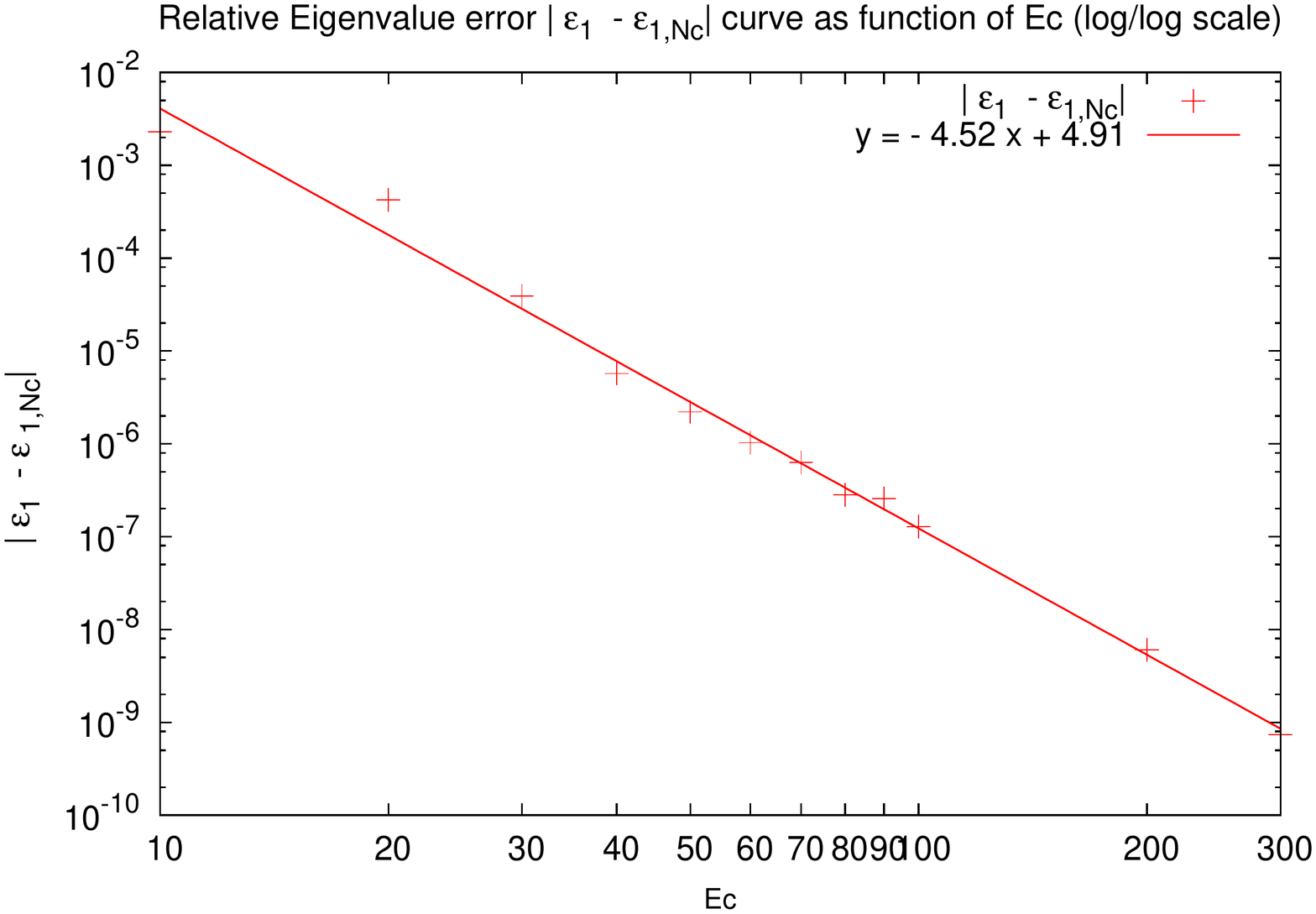}&
\includegraphics[width=7cm]{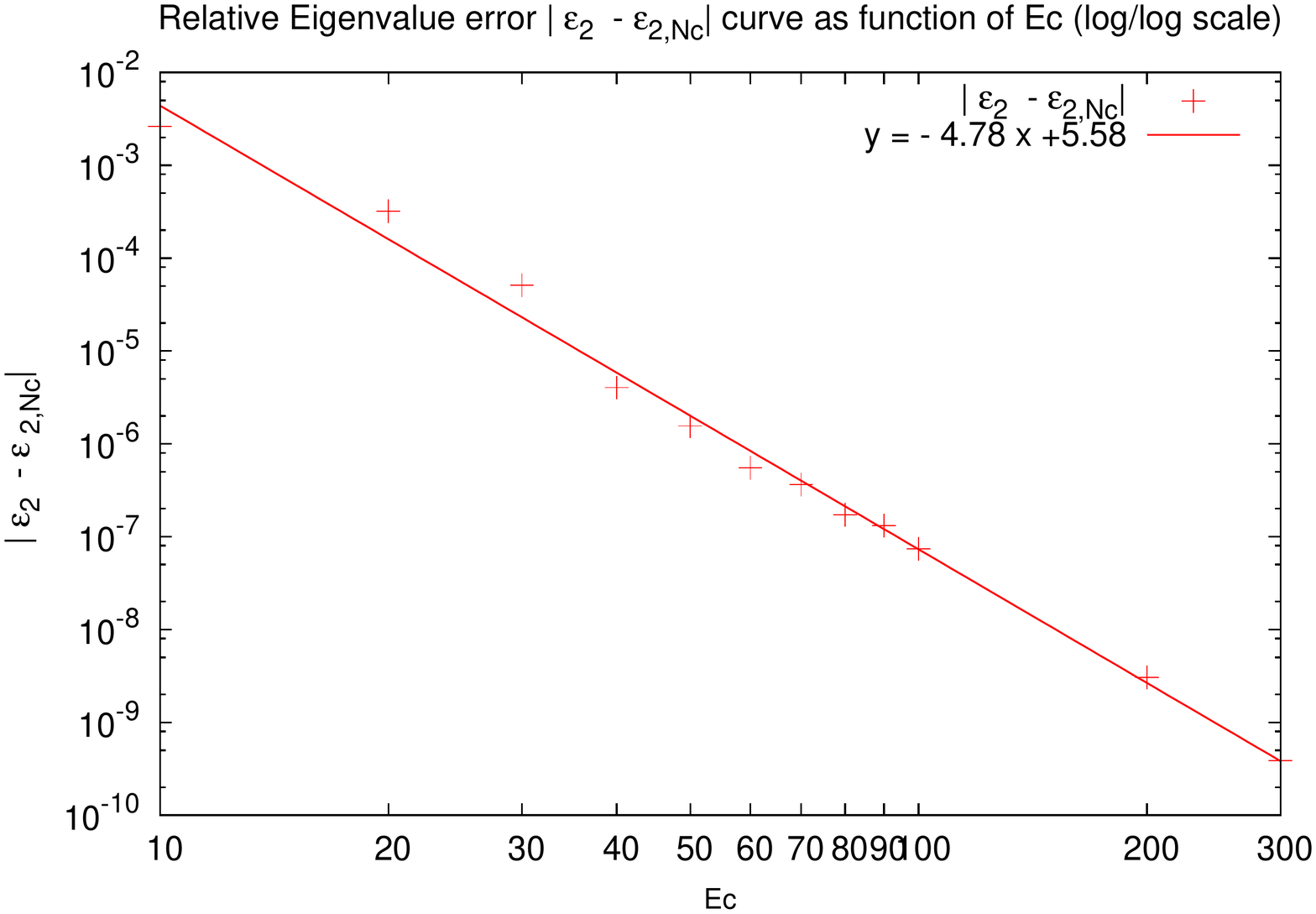}\\
\end{tabular}	 
\caption{Errors on $\| \Phi_{N_c}^0-\Phi^0\|_{H^1_\#}$ (left) and
$\| \Phi_{N_c}^0-\Phi^0\|_{L^2_\#}$ (right) and $|\epsilon_{i,N_c}^0-\epsilon_i^0| $ (bottom) as functions of $E_{\rm c}$ for H$_2$}\label{FigH2b}
 \end{figure}
\newpage

\begin{figure}[h!] 
\begin{tabular}{c  c }
\includegraphics[width=7cm]{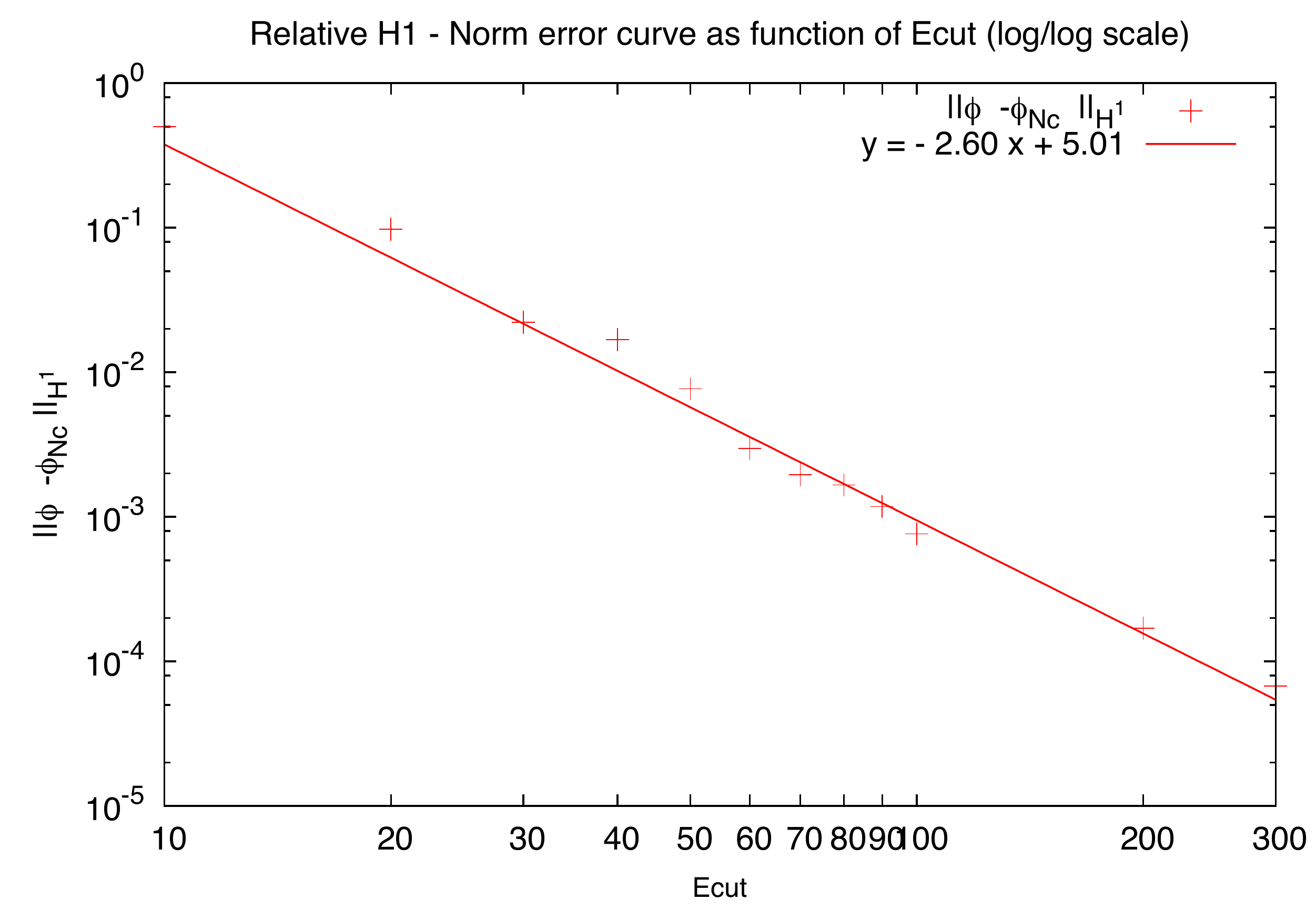}&
\includegraphics[width=7cm]{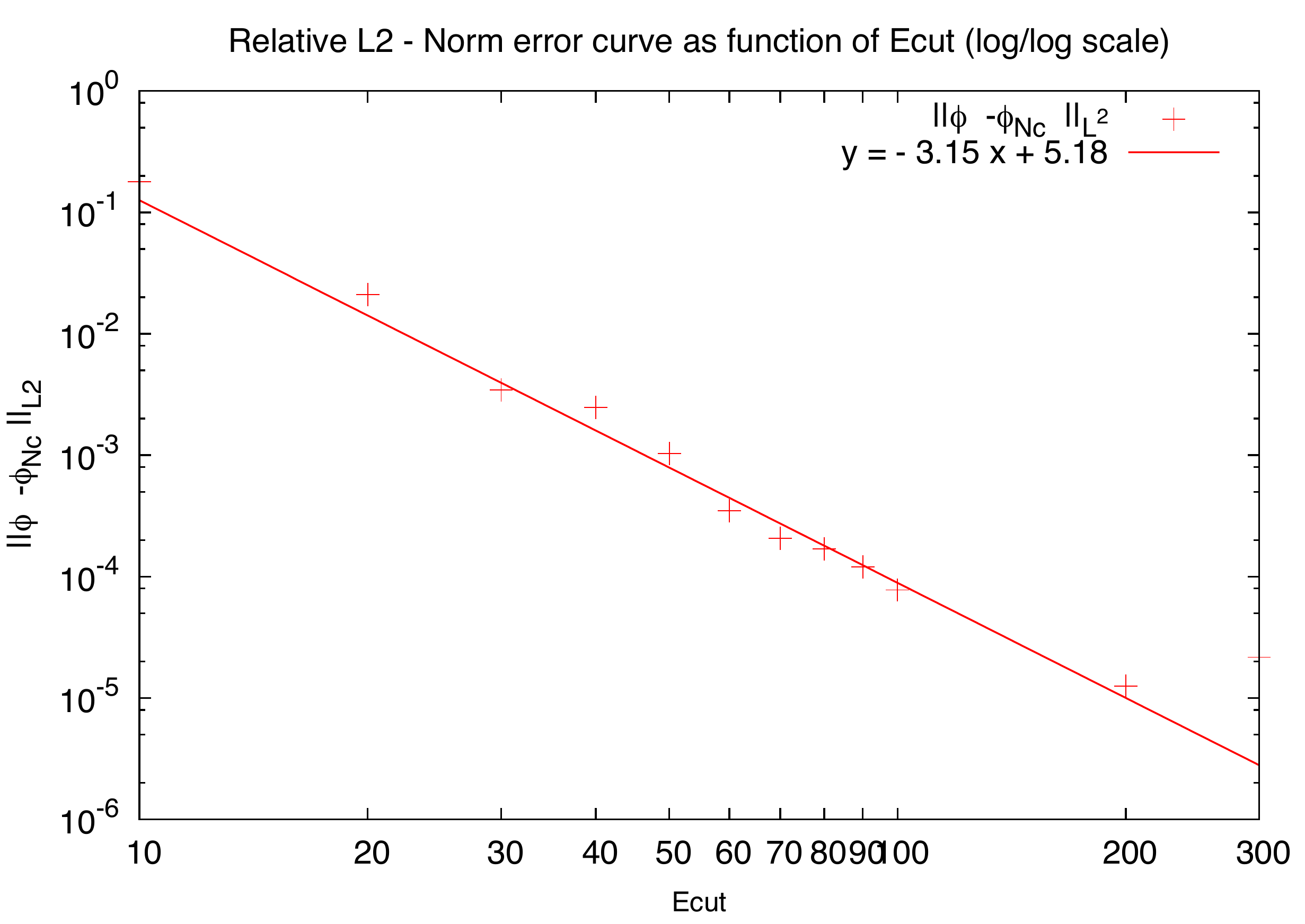}
\end{tabular}	 
\caption{Errors on $\| \Phi_{N_c}^0-\Phi^0\|_{H^1_\#}$ (left) and
$\| \Phi_{N_c}^0-\Phi^0\|_{L^2_\#}$ (right) as functions of $E_{\rm c}$ for N$_2$}\label{FigN2a}
 \end{figure}

\begin{figure}[h!] 
\begin{tabular}{c  c }
\includegraphics[width=7cm]{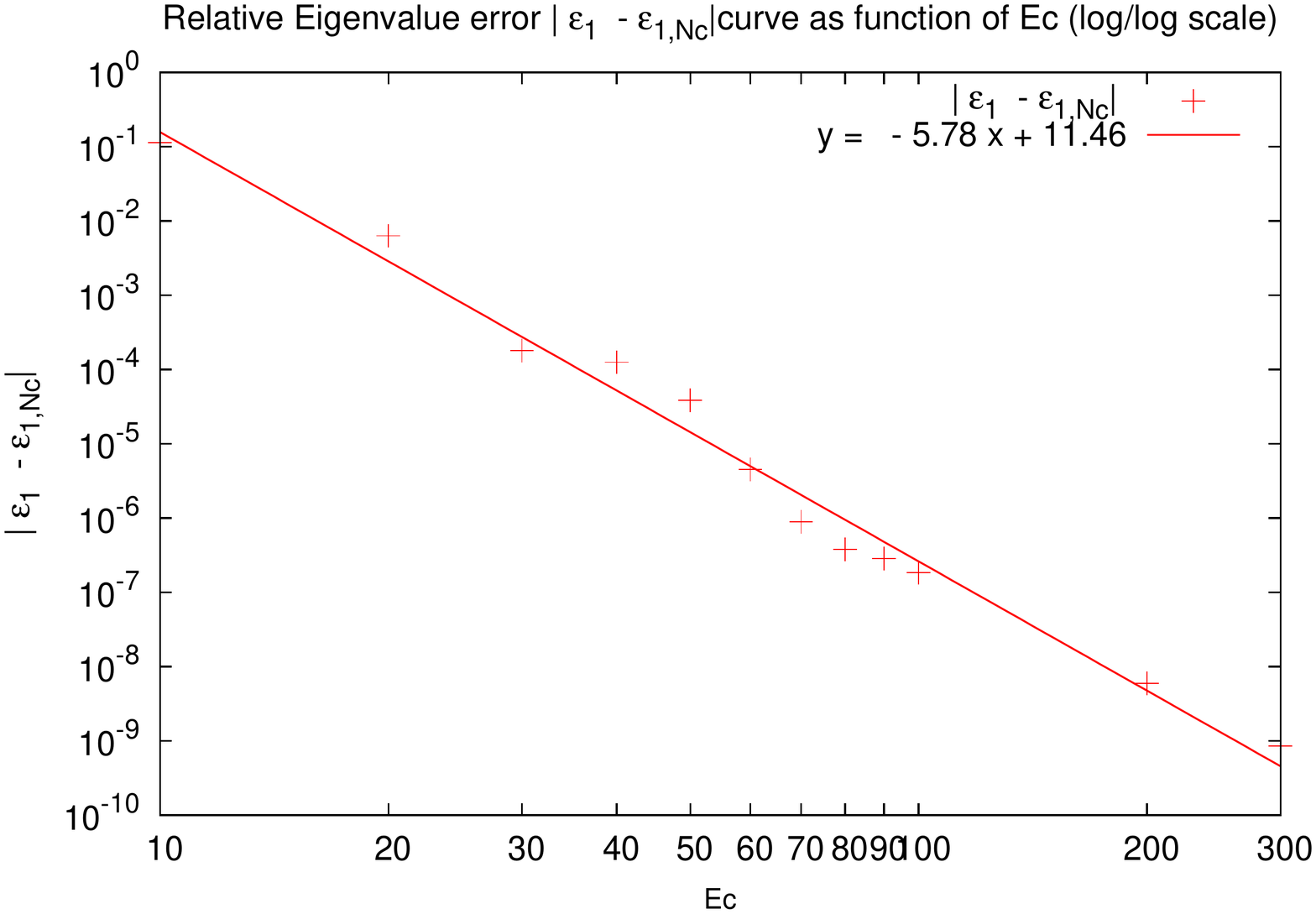}&
\includegraphics[width=7cm]{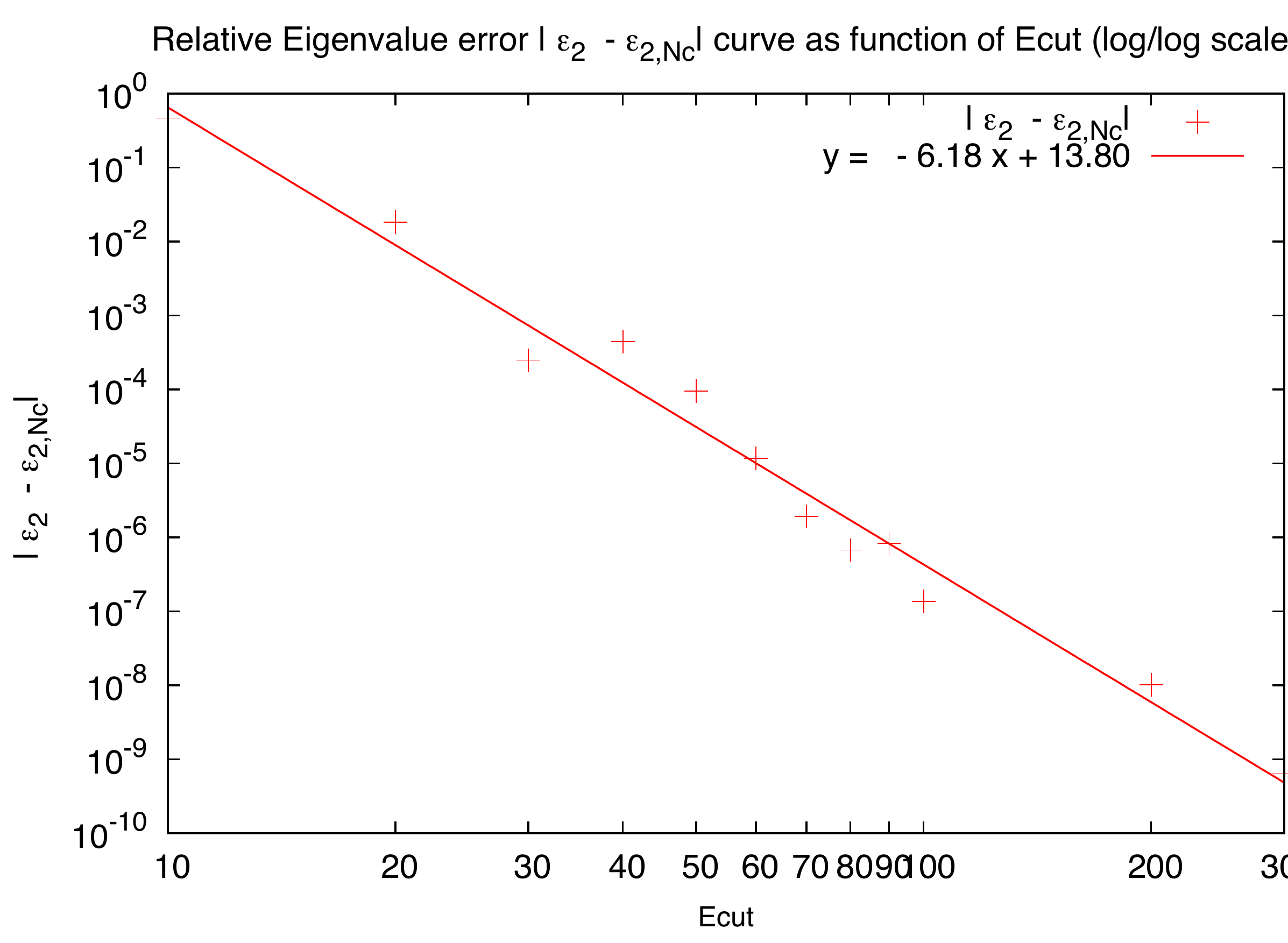}\\
\includegraphics[width=7cm]{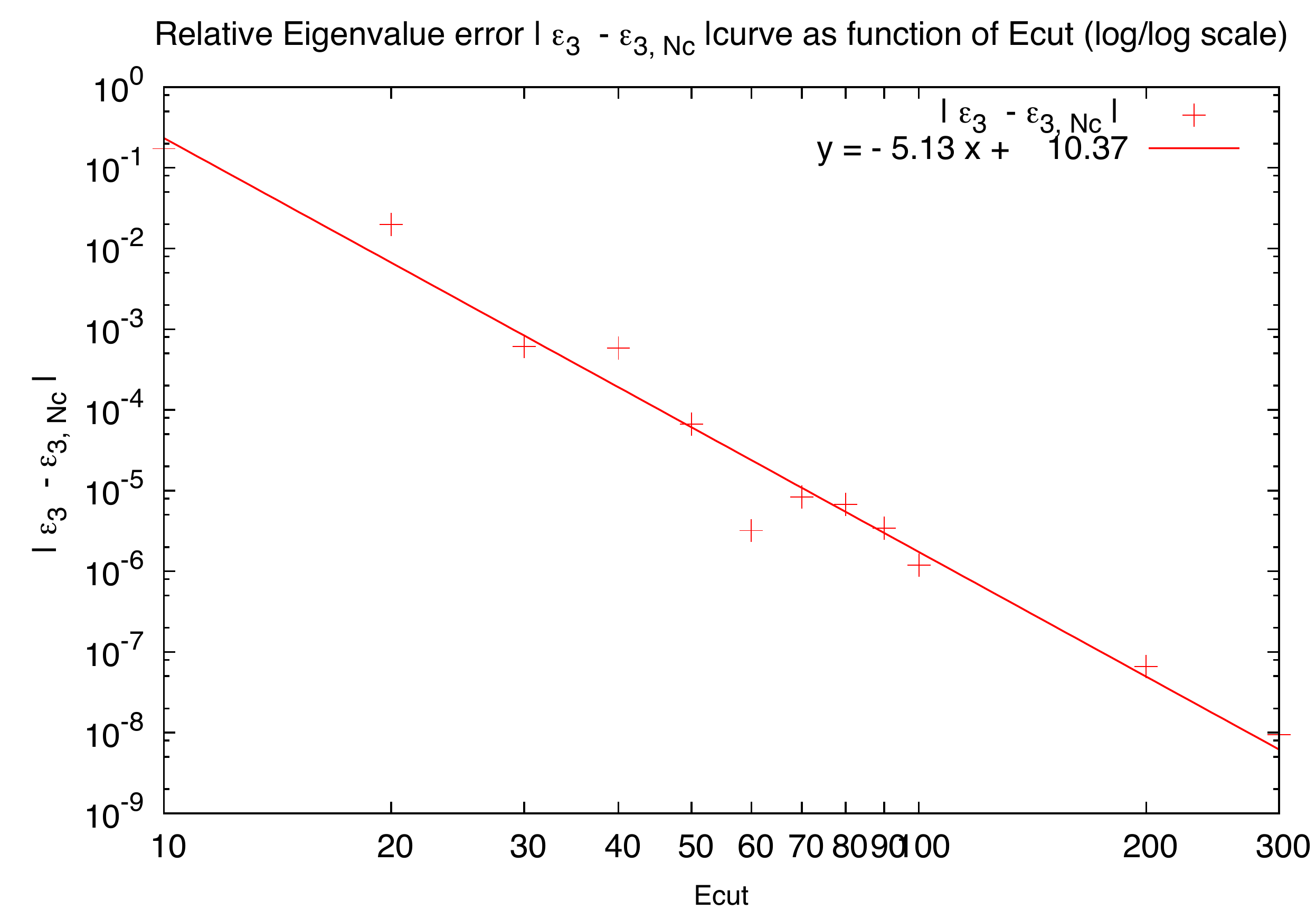}& 
\includegraphics[width=7cm]{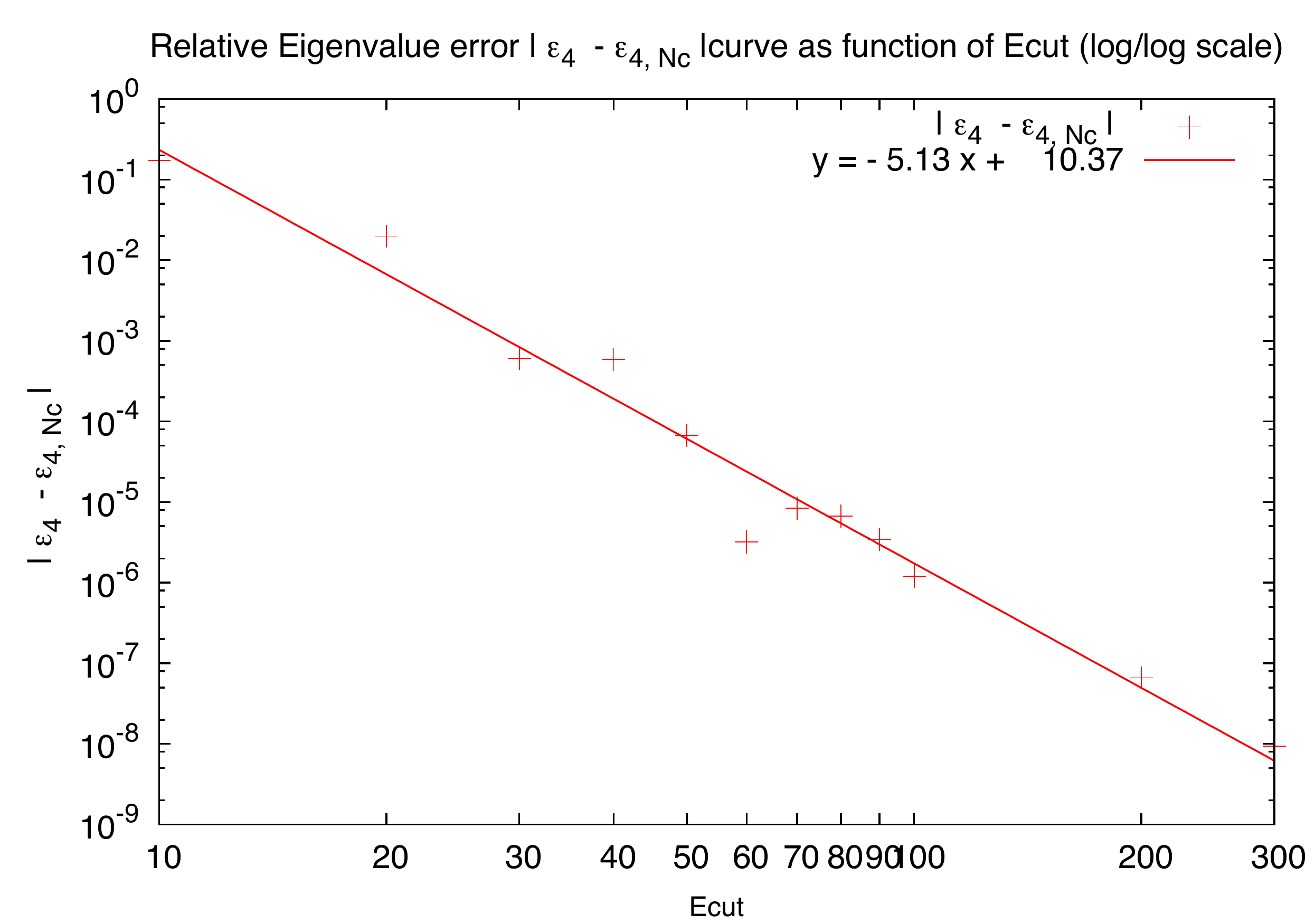}\\
\includegraphics[width=7cm]{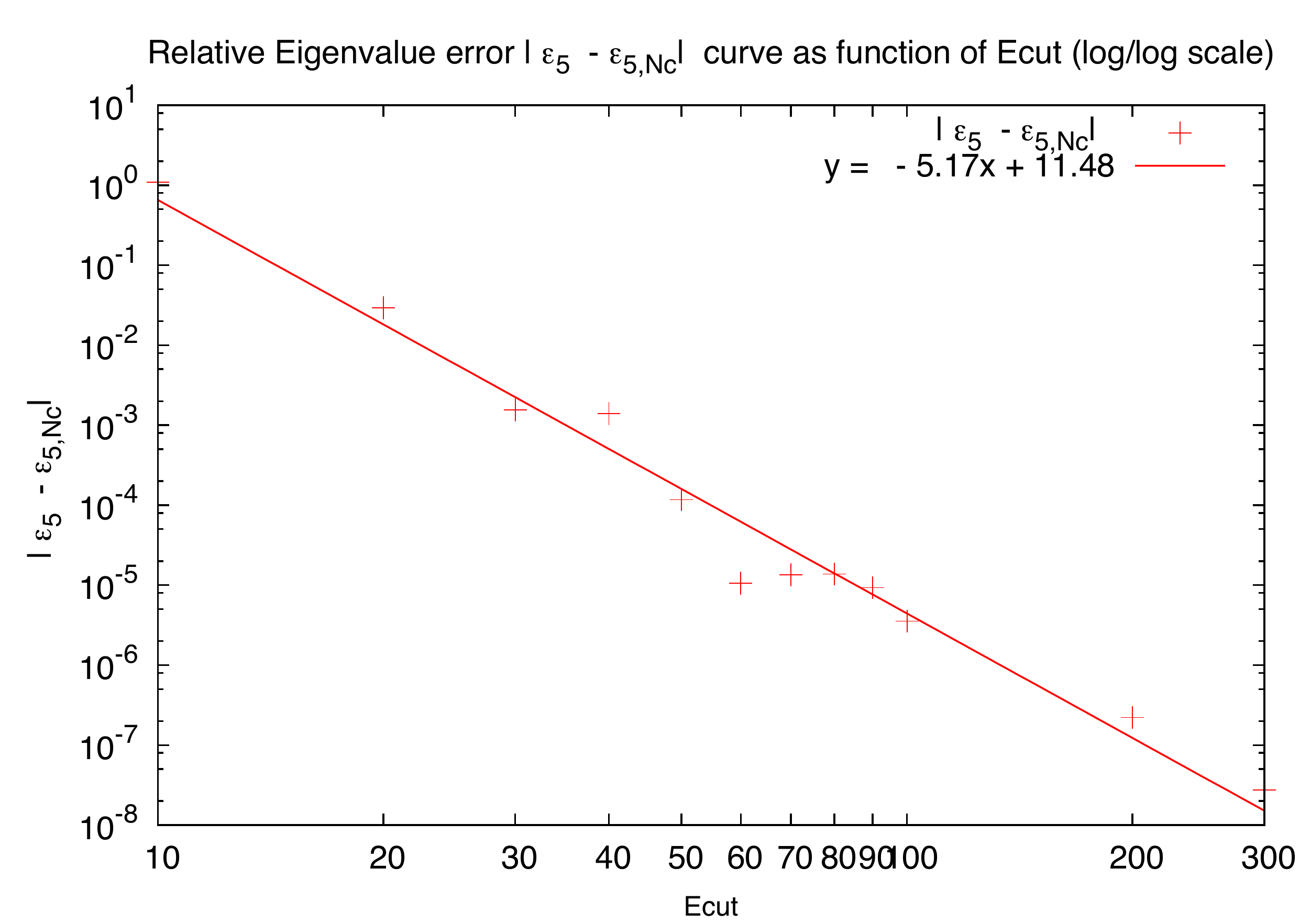}& 
\includegraphics[width=7cm]{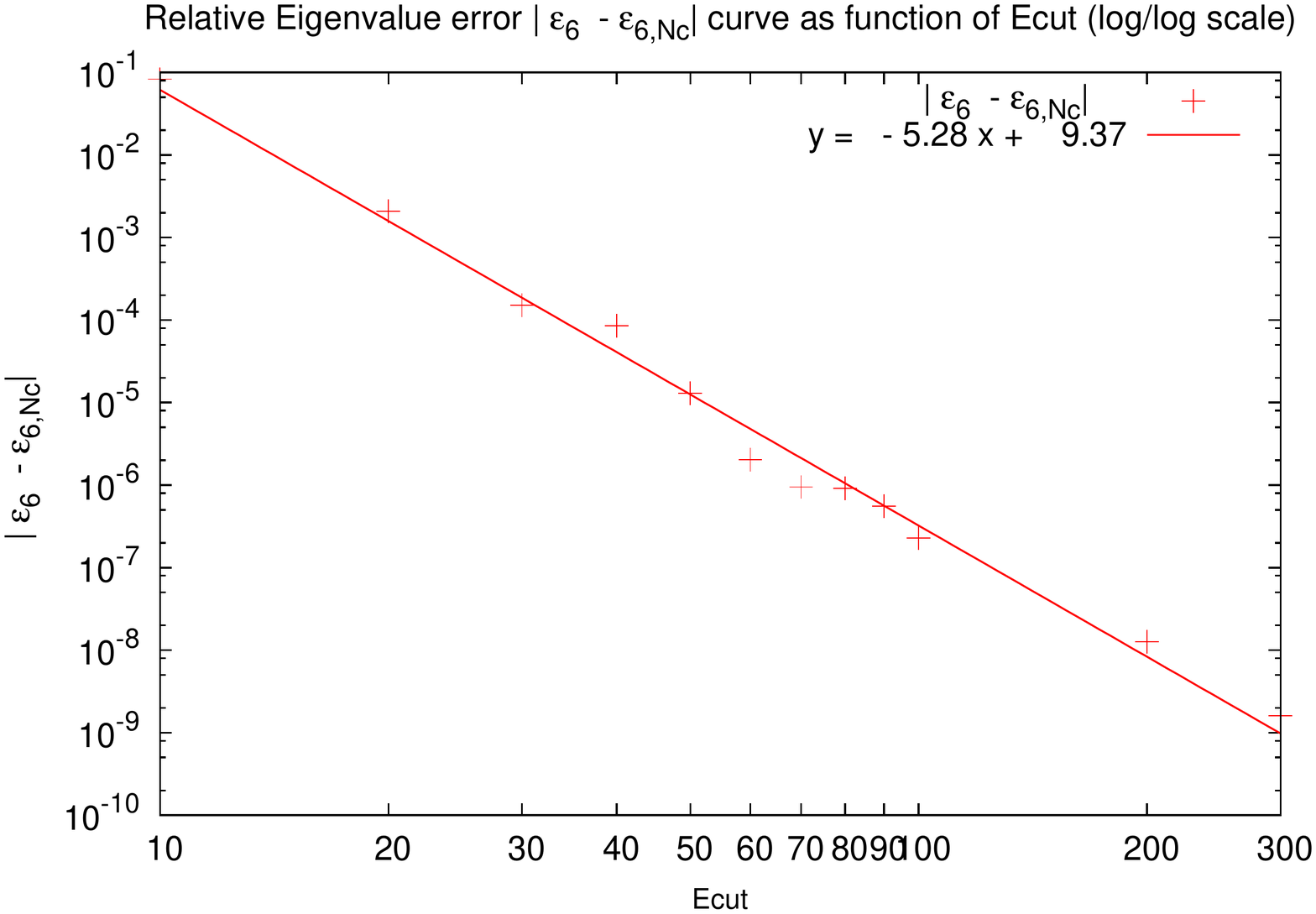} 
\end{tabular}	 
\caption{Errors on $|\epsilon_{i,N_c}^0-\epsilon_i^0| $  as functions of $E_{\rm c}$ for N$_2$}\label{FigN2b}
 \end{figure}

\newpage

\begin{figure}[h!] 
\begin{tabular}{c  c } \\
\includegraphics[width=7cm]{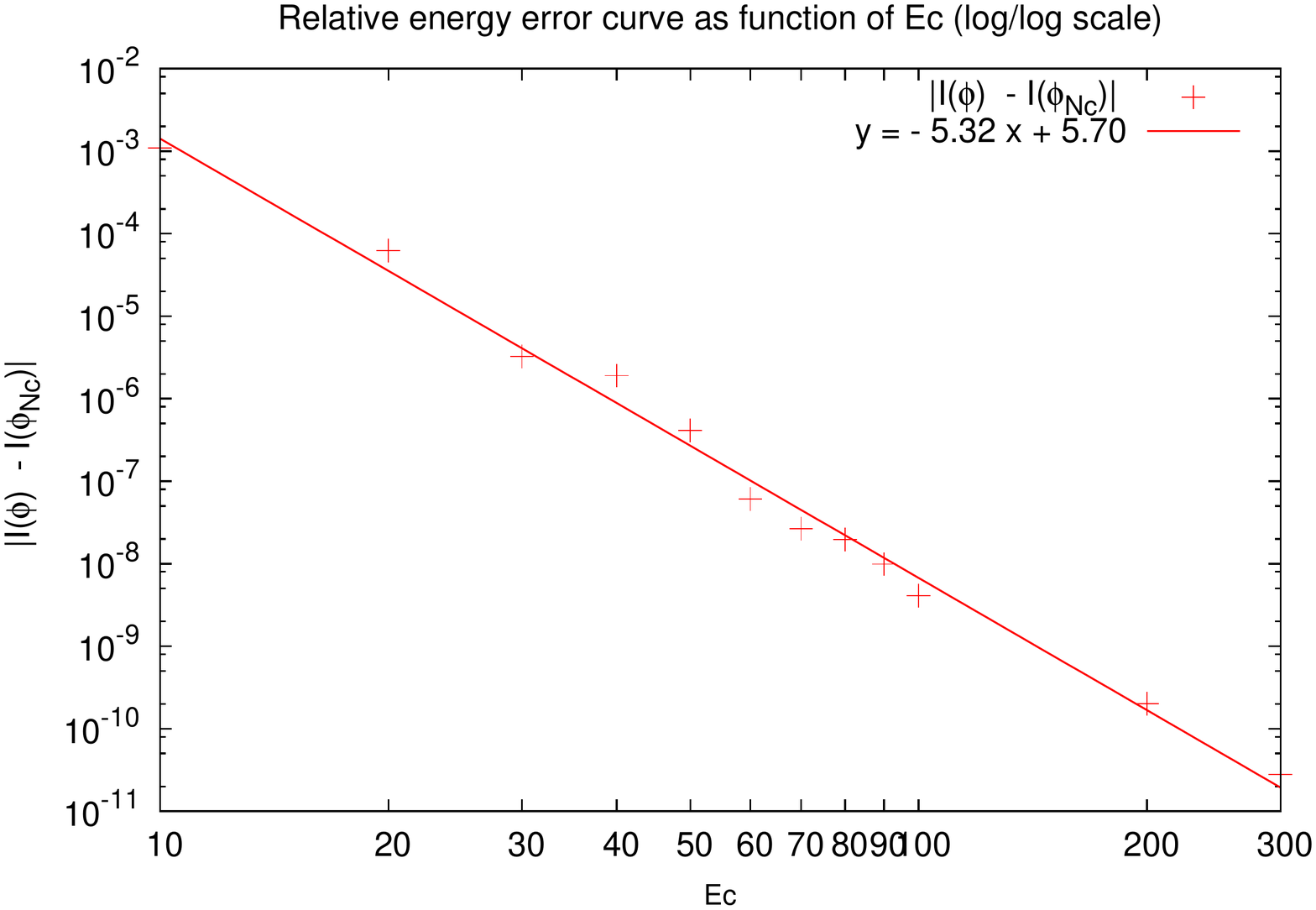}
\end{tabular}	 
\caption{Error on the energy as a function of $E_{\rm c}$  for N$_2$}\label{FigN2}
 \end{figure}

 These results are in good agreement with the {\it a priori} error estimates (\ref{estimationResNum1})-(\ref{estimationResNum2}) for both the H$_2$ and N$_2$ molecules

\section*{Acknowledgements} This work was done while E.C.
was visiting the Division of Applied Mathematics of Brown
University, whose support is gratefully acknowledged. 
This work was also partially supported by the ANR grant LN3M. We are also grateful to V.~Ehrlacher for her useful comments on a preliminary draft of this article.

\end{document}